\documentclass[a4paper,11pt]{article}
\usepackage{amsmath,amsthm,amssymb,dsfont,graphicx,xspace,epsfig}
\usepackage[dvipsnames]{xcolor}
\usepackage[T1]{fontenc}
\usepackage[plain]{fullpage}
\usepackage{thm-restate}
\usepackage[hidelinks]{hyperref}
\usepackage{bm}
\usepackage{color}
\usepackage{comment}
\usepackage{mathrsfs}
\usepackage{cleveref}
\usepackage[shortlabels]{enumitem}
\usepackage{float}
\usepackage{pict2e}
\usepackage{mathtools}
\usepackage{tikz}
\usepackage{subfigure}
\usepackage{xstring,refcount}
\usepackage{authblk}
\usepackage{yfonts}

\usetikzlibrary{fit, backgrounds, matrix, arrows.meta}
\usetikzlibrary{calc,arrows,positioning}
\usetikzlibrary{decorations}
\usetikzlibrary{decorations.pathmorphing}
\usetikzlibrary{decorations.pathreplacing}
\usetikzlibrary{decorations.shapes}
\usetikzlibrary{decorations.text}
\usetikzlibrary{decorations.markings}
\usetikzlibrary{decorations.fractals}
\usetikzlibrary{decorations.footprints}

\tikzset{vertex/.style = {circle,fill=black,minimum size=6pt, inner sep=0pt, outer sep=3pt}}
\tikzset{mvertex/.style = {circle,fill=black,minimum size=5pt, inner sep=0pt, outer sep=2pt}}
\tikzset{svertex/.style = {circle,fill=black,minimum size=4pt, inner sep=0pt, outer sep=2pt}}
\tikzset{labelledvertex/.style = {circle,fill=none,draw,very thick, inner sep=2pt, outer sep=2pt}}
\tikzset{Xrect/.style = {rectangle,fill=none,draw,very thick, minimum width=1.2cm, minimum height=0.8cm, outer sep=0pt}}

\tikzset{vtarc/.style = {line width=2.2pt,->,> = latex}}
\tikzset{tarc/.style = {thick,->,> = latex}}
\tikzset{sarc/.style = {thin,->,> = latex}}
\tikzset{tdigon/.style = {line width=2.2pt,<->,> = latex}}
\tikzset{digon/.style = {thick,<->,> = latex}}
\tikzset{bigvertex/.style = {shape=circle,draw}}

\tikzset{ledge/.style = {gray}}
\tikzset{tedge/.style = {thick}}
\tikzset{vtedge/.style = {line width=1.3pt}}
\tikzset{stedge/.style = {line width=2.2pt}}

\definecolor{g-blue}{rgb}{0.0, 0.5, 1.0}
\definecolor{g-green}{rgb}{0.3, 0.8, 0.3522}

\newtheorem{theorem}{Theorem}[section]
\newtheorem{lemma}[theorem]{Lemma}
\newtheorem{corollary}[theorem]{Corollary}
\newtheorem{proposition}[theorem]{Proposition}
\newtheorem{conjecture}[theorem]{Conjecture}

\newtheorem{claim}{Claim}[theorem]
\newcounter{oldthm}{}
\def\@extractthmnum#1.#2{#2}

\theoremstyle{definition}
\newtheorem{question}[theorem]{Question}

\newenvironment{proofclaim}[1][Proof of claim]
  {\begin{proof}[#1]}
  {\end{proof}}

\newcommand{\defproblem}[3]{
 \vspace{3mm}
\noindent\fbox{
 \begin{minipage}{0.96\textwidth}
 \begin{tabular*}{\textwidth}{@{\extracolsep{\fill}}lr}\sc #1  \\ \end{tabular*}
 {\bf{Input:}} #2 \\
 {\bf{Question:}} #3
 \end{minipage}
 }
 \vspace{3mm}
}

\makeatletter
\newcommand{\overleftrightsmallarrow}{\mathpalette{\overarrowsmall@\leftrightarrowfill@}}
\newcommand{\overrightsmallarrow}{\mathpalette{\overarrowsmall@\rightarrowfill@}}
\newcommand{\overleftsmallarrow}{\mathpalette{\overarrowsmall@\leftarrowfill@}}
\newcommand{\overarrowsmall@}[3]{%
  \vbox{%
    \ialign{%
      ##\crcr
      #1{\smaller@style{#2}}\crcr
      \noalign{\nointerlineskip}%
      $\m@th\hfil#2#3\hfil$\crcr
    }%
  }%
}
\def\smaller@style#1{%
  \ifx#1\displaystyle\scriptstyle\else
    \ifx#1\textstyle\scriptstyle\else
      \scriptscriptstyle
    \fi
  \fi
}
\makeatother

\newcommand{\dig}[1]{\langle #1 \rangle}

\newcommand{\XOR}{\,\triangle\,}

\newcommand{\Pcal}{\mathcal{P}}
\newcommand{\Qcal}{\mathcal{Q}}
\newcommand{\Xcal}{\mathcal{X}}
\newcommand{\Ycal}{\mathcal{Y}}
\newcommand{\Zcal}{\mathcal{Z}}

\newcommand{\Eset}[3]{\delta_{#1}(#2,#3)}

\DeclareMathOperator{\temp}{sinv}
\newcommand{\inv}[2]{\temp_{#1}^{\leq #2}}
\newcommand{\invunb}[1]{\temp_{#1}}

\DeclareMathOperator{\tempac}{acinv}
\newcommand{\acinveqp}[1]{\tempac^{=#1}}
\newcommand{\acinvleqp}[1]{\tempac^{\leq#1}}
\newcommand{\acinvunb}{\tempac}

\newcommand{\InvSeq}[2]{$(#1,\leq\!#2)$-inverting set}
\newcommand{\eqp}[1]{$(=\!#1)$}
\newcommand{\leqp}[1]{$(\leq\!#1)$}

\DeclareMathOperator{\UG}{UG}

\DeclareMathOperator{\Inv}{Inv}
\DeclareMathOperator{\Push}{Push}

\renewcommand{\epsilon}{\varepsilon}
\renewcommand{\emptyset}{\varnothing}

\renewcommand{\tilde}{\widetilde}
\renewcommand{\phi}{\varphi}

\let\ge\geqslant
\let\leq\leqslant
\let\geq\geqslant

\hypersetup{
    colorlinks,
    linkcolor={red!50!black},
    citecolor={blue!50!black},
    urlcolor={blue!80!black}
}

\title{Increasing arc-connectivity by bounded- and fixed-size inversions\thanks{This work is announced at 'Mathematical Foundations of Computer Science' (MFCS) 2026.}}

\author[1]{Florian Hörsch}
\author[,2]{Lucas Picasarri-Arrieta\thanks{Research supported by JSPS KAKENHI JP20A402 and 22H05001, and by JST ASPIRE JPMJAP2302.}}

\affil[1]{CISPA, Saarbrücken, Germany\\ florian.hoersch@cispa.de}
\affil[2]{National Institute of Informatics, Tokyo, Japan\\lpicasarr@nii.ac.jp}
\date{}

\begin{document}

\maketitle

\begin{abstract}
Given an integer $k\geq 1$, a digraph $D$ is $k$-arc-strong if the removal of any set of at most $k-1$ arcs of $D$ yields a strongly connected digraph. 
For a digraph $D$ and some set $X \subseteq V(D)$, the inversion of $X$ is the operation of flipping all arcs both of whose endvertices are in $X$. We initiate the study of establishing arc-connectivity properties by applying inversions of bounded or fixed size.

For fixed-size inversions, we consider the feasibility of the problem by characterizing, for all integers $p \geq 2$ and $k \geq 1$, the digraphs that can be made $k$-arc-strong by applying inversions of size exactly $p$, provided a minimum size of the digraphs.

For bounded-size inversions, the tractability of the feasibility problem follows easily from a famous theorem of Nash-Williams, so we focus on minimising the number of inversions. We prove that for all integers $p\geq 3$ and $k \geq 1$ and any $\epsilon>0$, there exists a polynomial-time $(4k-2+\epsilon)$-approximation algorithm for computing the minimum number of inversions of size at most $p$ that make a given digraph $k$-arc-strong. 
This is in stark contrast to other results on inversion optimization problems. On the other hand, we show that for any $p\geq 3$ and $k \geq 1$ the problem is NP-hard, and, moreover, APX-hard.

As a result on parameterized complexity, we show that for any $k \geq 2$, it is $W[1]$-hard with respect to $p$ to decide whether a given digraph can be made $k$-arc-strong by applying a single inversion of size at most $p$. We also prove that for a given multidigraph, it is $W[1]$-hard with respect to $\ell$ to decide whether it can be made 2-arc-strong by applying $\ell$ inversions of size 2. 

\noindent{}{\bf Keywords:} Bounded-size inversions, Strong connectivity, Approximation algorithms, Parameterized complexity.

\end{abstract}

\section{Introduction}

In this paper, a {\it multidigraph} consists of a vertex set and an arc set, where parallel arcs and digons are allowed. A multidigraph without parallel arcs is called a {\it digraph} and a digraph without digons is called an {\it oriented graph}. More basic notation can be found in Section~\ref{sec:preliminaries}. 

Given a multidigraph $D$ and a set $X\subseteq V(D)$, {\it inverting $X$ in $D$} means flipping the orientation of all arcs in $D$ that have both endvertices in $X$. Inversions fall into the category of reconfiguration operations. A general overview of reconfiguration in graph theory can be found in the survey of Nishimura~\cite{nishimuraALGO11}.
Despite being a relatively young field in graph theory, the theory of inversions has received significant attention in the last 15 years. Most work on inversions has been done with the objective of understanding the minimum number of inversions through which a given oriented graph can be made acyclic. The literature on this subject is already quite rich. In recent years, inversions with different objectives have also been considered. This in particular involves results on the so-called inversion diameter of undirected graphs and inversions with the objective of making an oriented graph highly connected. 

It is also natural to consider inversions of bounded and fixed size. In oriented graphs, inversions involving two vertices correspond to arc flips. While general inversions are known to behave rather chaotically, problems involving arc flips tend to be more tractable. It is hence natural to study fixed-size and bounded-size inversions as an intermediate step between these two concepts. Prior to this work, a few articles have considered bounded-size inversions with the objective of making an oriented graph acyclic, as well as in the context of the inversion diameter. So far, to the best of our knowledge, bounded- or fixed-size inversions have not yet been studied in the context of making a given oriented graph highly connected.

The objective of the present article is to fill this gap. As our first main result, for all fixed positive integers $p$ and $k$, we give a characterization of all digraphs that can be made $k$-arc-strong by inversions of size exactly $p$, provided the digraph is sufficiently large compared to $p$ and $k$. Next, we focus on making a digraph $k$-arc-strong by a minimum number of inversions of size at most $p$. We prove that for any fixed $p \geq 3$ and $k \geq 1$, the problem is APX-hard, but admits a polynomial-time $(4k-2+\epsilon)$-approximation algorithm for arbitrary $\epsilon>0$. Finally, we give some results on the parameterized complexity of the problem when considering either $p$ or the solution size as parameters.

In the following, we first give a more extensive overview of previous work and then present our findings in more detail.

\subsection{Previous work}

We first overview the results on making a given oriented graph acyclic by a minimum number of inversions of arbitrary size, the setting that has attracted most attention.  Given an oriented graph $D$, we denote by $\acinvunb(D)$ the minimum number of inversions to make $D$ acyclic.
Inversions were first considered by Belkhechine et al.~\cite{BELKHECHINE2010703}. 
They focused on tournaments and initiated the study of the asymptotic behaviour of $\acinvunb(n)$, defined as the maximum of $\acinvunb(T)$ over all tournaments $T$ on $n$ vertices. Their bounds were later significantly improved by Alon et al.~\cite{APSSW} and Aubian et al.~\cite{inversion}, who independently showed that $\acinvunb(n) = (1-o(1))n$.

The complexity of computing $\acinvunb(D)$ for a given oriented graph $D$ was first considered by Bang-Jensen, da Silva, and Havet~\cite{dmtcs:7474}. They proved that it is NP-hard to decide whether $\acinvunb(D)\leq 1$ for a given oriented graph $D$. This was extended by Alon et al.~\cite{APSSW} who proved that deciding whether $\acinvunb(D)\leq \ell$ is NP-hard for every fixed $\ell \geq 1$. 
Further, in~\cite{dmtcs:7474}, a conjecture on the behavior of $\acinvunb$ under dijoins was made. Given two oriented graphs $D_1$ and $D_2$, the {\it dijoin $D_1\Rightarrow D_2$} is obtained from vertex-disjoint copies of $D_1$ and $D_2$ by adding one arc from every vertex in $V(D_1)$ to every vertex in $V(D_2)$. The authors conjectured that 
\[
\acinvunb(D_1 \Rightarrow D_2)=\acinvunb(D_1)+\acinvunb(D_2)
\]
holds for all oriented graphs $D_1$ and $D_2$. While some cases for small inversion numbers were solved in~\cite{dmtcs:7474}, the conjecture was generally refuted independently in~\cite{APSSW} and~\cite{inversion}. Some results on special cases of this conjecture have been proved by Behague et al. \cite{Behague_Johnston_Morrison_Ogden_2025}, Wang, Yang, and Lu~\cite{wang2024inversionnumberdijoinsblowup}, and Behague and Gaudart-Wifling~\cite{behague2025casedijoinconjectureinverting}.\medskip

Recently, inversions have also been studied beyond the objective of making an oriented graph acyclic. The concept of inversion diameter was introduced by Havet, Hörsch, and Rambaud~\cite{IGT_2026__3__49_0}. Given an undirected simple graph, its {\it inversion graph} is the graph defined on the set of all its orientations and where two orientations are adjacent if one can be obtained from the other by a single inversion. 
The {\it inversion diameter} of a graph is the diameter of its inversion graph. In~\cite{IGT_2026__3__49_0}, the authors prove some hardness results on computing the inversion diameter of a graph and show that it is functionally tied to some other graph parameters, namely the star chromatic number, the acyclic chromatic number and the oriented chromatic number. This yields in particular that the inversion diameter is bounded for all proper minor-closed classes of graphs. More explicit bounds on the inversion diameter of graphs coming from special classes were obtained, for example planar graphs, planar graphs of bounded girth, and graphs of bounded treewidth. The treewidth bound has been shown to be tight by Wang et al.~\cite{wang2026inversiondiametertreewidth}. Recently, several quantitative improvements on bounds in~\cite{IGT_2026__3__49_0} have been obtained by Arana et al.~\cite{arana2026inversiondiameter2edgecoloredhomomorphisms} as well as Bo et al.~\cite{boArxiv26}.

\medskip

We now come to another setting that plays a crucial role in this article, namely the idea of applying inversions in order to make the thus obtained digraph satisfy certain connectivity properties. To our best knowledge, the only work in which this setting was previously considered is the one of Duron et al.~\cite{https://doi.org/10.1002/jgt.23290}. For a positive integer $k$ and a multidigraph $D$, we use $\invunb{k}(D)$ for the minimum number of inversions which are necessary to make $D$ $k$-arc-strong. In the preceding discussion, we have mainly ignored the question of what happens when instead of considering oriented graphs, we also allow digons or parallel arcs. 
This is justified by the fact that these objects play no role when considering ways to make a multidigraph acyclic. Indeed, parallel arcs are irrelevant, and the presence of a digon immediately rules out the existence of a feasible solution. Similarly, parallel arcs and digons play no role in the study of inversion diameters.
However, when it comes to shooting for a connectivity property, then parallel arcs may make a crucial difference. In particular, it follows directly from a well-known theorem on orientations of Nash-Williams~\cite{nashwilliamsCJM12} that an oriented graph can be made $k$-arc-strong by inversions if and only if its underlying graph is $2k$-edge-connected. As implicitly pointed out in~\cite{https://doi.org/10.1002/jgt.23290}, this result can rather easily be generalized to digraphs. On the other hand, it follows directly from a result of Hörsch and Szigeti~\cite{HORSCH2021103292} that it is NP-hard to decide whether a multidigraph can be made 2-arc-strong by applying inversions. 
In~\cite{https://doi.org/10.1002/jgt.23290}, the authors prove that when aiming to minimize the number of necessary inversions, the problem also becomes hard in oriented graphs. 
More precisely, they prove that for any positive integers $k$ and $\ell$, it is NP-hard to decide whether $\invunb{k}(D)\leq \ell$ for a given oriented graph $D$. Moreover, they prove the following result, showing that even any acceptable approximation is out of reach.

\begin{theorem}[\cite{https://doi.org/10.1002/jgt.23290}]
    \label{inapprac}
    For any two positive integers $k$ and $\ell$, unless $P=NP$ there is no polynomial-time algorithm asserting whether a given oriented graph $D$ can be made $k$-arc-strong by one inversion or cannot be made $k$-arc-strong by at most $\ell$ inversions. 
\end{theorem}

Next, it was proved in~\cite{https://doi.org/10.1002/jgt.23290} that for every fixed positive integer $k$, the maximum of $\invunb{k}(D)$ over all $2k$-edge-connected digraphs $D$ on $n$ vertices is a logarithmic function in $n$.
Further, the authors prove a number of upper bounds for the special case of tournaments.

\paragraph{Inversions of bounded and prescribed size}
 
In all of the above discussion, we have disregarded the size of the sets that are inverted. It is natural to aim to understand the behavior of inversions when the size of the sets that can be inverted is restricted. In the following, for some positive integer $p$, an {\it \eqp{p}-inversion} (respectively {\it \leqp{p}-inversion}) is an inversion of a set of size exactly $p$ (respectively at most $p$). 
The most restrictive meaningful case to consider is the one of inversions of sets of size 2. In oriented graphs, inversions of size 2 correspond exactly to arc flips, the operation of reversing an arc. 
This operation has been deeply studied from many angles. 
From an algorithmic perspective, the behavior of inversions has so far been rather chaotic, as exemplified by the fact that deciding whether $\acinvunb(D)\leq 1$ for a given oriented graph $D$ is NP-hard.
On the other hand, arc flips are comparatively well understood. In this light, considering inversions of bounded or prescribed size is a natural intermediate step between unbounded inversions and arc flips. 

Bounded-size inversions in order to make an oriented graph acyclic were first considered by Yuster~\cite{https://doi.org/10.1002/jgt.23251}. He focused on tournaments and proved a collection of extremal results. For a given oriented graph $D$, we use $\acinveqp{p}(D)$ ($\acinvleqp{p}(D)$) for the minimum number of \eqp{p}-inversions (\leqp{p}-inversions) to make $D$ acyclic. The algorithmic complexity of computing the minimum number of bounded-size inversions to make an oriented graph acyclic was studied by Bang-Jensen et al.~\cite{bangjensen2025makingorientedgraphacyclic}. While computing a minimum set of arcs whose flipping makes a given oriented graph acyclic is one of Karp's 21 NP-hard problems~\cite{karp2010book}, the corresponding result in tournaments, proved independently by Alon~\cite{alonSJDM20} and Charbit, Thomassé, and Yeo~\cite{charbitCPC16}, answered a long-standing open question. In~\cite{bangjensen2025makingorientedgraphacyclic}, it was proved that  computing $\acinvleqp{p}(D)$ for a given tournament $D$ is NP-hard for any fixed $p \geq 2$, using a comparatively easy reduction for $p \geq 3$. The existence of a polynomial-time constant factor approximation for the minimum number of arc flips that makes a given digraph acyclic was a long-standing open problem. Finally, it was proved by Guruswami et al.~\cite{doi:10.1137/090756144} that such an algorithm would contradict the Unique Games Conjecture (UGC). Some simpler hardness proofs were given by Svensson~\cite{v009a024} and Guruswami and Lee~\cite{v012a006}.  It is not difficult to see that for every fixed $p \geq 2$, a polynomial-time constant factor approximation for computing $\acinvleqp{p}$ exists if and only if the answer for the arc flip problem is affirmative. Hence, for any $p \geq 2$, a constant factor approximation for $\acinvleqp{p}$  can be ruled out under UGC. Next, in~\cite{bangjensen2025makingorientedgraphacyclic}, parameterized complexity was also considered. Among others, the authors proved the following result.

\begin{theorem}[\cite{bangjensen2025makingorientedgraphacyclic}]
    \label{acw1}
    Deciding whether a given oriented graph can be made acyclic with a single inversion of size at most $p$ is $W[1]$-hard with respect to $p$.
\end{theorem}
Observe that \Cref{acw1} is incomparable to the result that it is NP-hard to decide whether a given oriented graph can be made acyclic by a single inversion.

Finally, in~\cite{bangjensen2025makingorientedgraphacyclic}, fixed-size inversions were also studied. Here, besides the problem of minimising the number of inversions, also the question of the existence of a feasible solution becomes significantly more interesting. Observe that, when considering an oriented graph, a bounded-size inversion can flip a single arc and hence the problem of deciding the existence of a feasible solution reduces to an orientation problem. This is no longer the case for fixed-size inversions. Nevertheless, using methods from linear algebra, the authors of~\cite{bangjensen2025makingorientedgraphacyclic} prove a characterization of tournaments that can be made acyclic by fixed-size inversions and use it to give a polynomial-time algorithm for general oriented graphs. 

When it comes to bounded-size inversions for other objectives than becoming acyclic, we wish to mention the recent work of Havet, Rambaud, and Silva~\cite{havet2026leqpinversiondiameteroriented}. They consider a variation of the diameter problem for bounded-size inversions and prove a collection of extremal results.

\subsection{Our contributions}
As far as we are aware, no previous efforts have been made to combine the concept of bounded-size and fixed-size inversions with the objective of making a digraph satisfy connectivity requirements. The objective of the present article is to fill this gap. Our results can be divided into two categories. First, we consider feasibility problems. Here, we deal with the question of which digraphs or multidigraphs can be made $k$-arc-strong by fixed-size inversions, irrespective of the number of required inversions. Afterwards, we focus on bounded-size inversions, and consider the problem of making a directed graph $k$-arc-strong by using a minimum number of such inversions. Here, we give some results on the exact computation and the approximation of the minimum number of inversions. We finally give some results related to parameterized complexity.

\paragraph{Feasibility}

We first consider general multidigraphs and the most basic case, which is $p=2$. Observe that for $p=2$, bounded-size and fixed-size inversions give the same means of altering a multidigraph. Moreover, when it comes to feasibility, inversions of size 2 are exactly as powerful as arbitrary inversions. Interestingly, we can obtain the following result as a simple corollary of a result of Hörsch and Szigeti~\cite{HORSCH2021103292}. It rules out the possibility of solving the feasibility problem for general multidigraphs even in the most basic case.

\begin{theorem}
    \label{thm:NP_hardness_2inversion_multidigraphs}
    It is NP-hard to decide whether a given multidigraph can be made $2$-arc-strong by applying \eqp{2}-inversions.
\end{theorem}

We give the proof of \Cref{thm:NP_hardness_2inversion_multidigraphs} in Section~\ref{sec:NP_hardness} for completeness.
In the light of \Cref{thm:NP_hardness_2inversion_multidigraphs}, we now focus on digraphs. For any positive integer $k$, a well-known theorem of Nash-Williams~\cite{nashwilliamsCJM12} (see Theorem~\ref{thm:nashwilliams}) states that a graph has a $k$-arc-strong orientation if and only if it is $2k$-edge-connected. In the following, we say that a multidigraph is {\it $2k$-edge-connected} if its underlying graph is $2k$-edge-connected. 
Clearly, every digraph that can be made $k$-arc-strong by applying \eqp{2}-inversions  is in particular $2k$-edge-connected by the theorem of Nash-Williams. 
Moreover, we immediately obtain that this condition is also sufficient for oriented graphs.  Our next result extends this observation in two ways. First, we allow digraphs rather than only oriented graphs. Second, rather than only considering \eqp{2}-inversions, we consider inversions of any fixed even size. 
On the other hand, we need the digraph in consideration to have a certain minimum size. Formally, we prove the following result. 

\begin{restatable}{theorem}{feasabilityeven}
    \label{thm:invertibility_even_case}
    Let $k$ be a positive integer, $p\geq 2$ be an even integer, and $D$ be a digraph of order $n\geq\max \{p+2,2k+2\}$. Then $D$ can be made $k$-arc-strong using \eqp{p}-inversions if and only if it is $2k$-edge-connected.
\end{restatable}

The case of fixed odd $p$ is more complicated and certain nontrivial obstructions appear. We describe these obstructions and prove that these are the only $2k$-edge-connected digraphs that cannot be made $k$-arc-strong by inversions of fixed odd size, again assuming a minimum size condition. 

We now describe these obstructions. Namely, for an integer $k \geq 1$, a {\it $k$-obstruction} is a digraph $D$, with underlying graph $G$, admitting a partition $(X_1,\dots,X_r,Y)$ of its vertex set into non-empty parts such that, for $X=\bigcup_{i\in [r]} X_i$, we have:
\begin{enumerate}[label=(\roman*)]
    \item for every $i\in [r]$, $d_G(X_i) = 2k$;
    \item for every $x\in X$ and $y\in Y$, there is exactly one edge between $x$ and $y$ in $G$; and
    \item $d^+_D(X) - \frac{1}{2}|X|\cdot |Y|$ is an odd integer.
\end{enumerate}
See Figure~\ref{fig:k_obstructions} for an illustration of $k$-obstructions. Note that, for every fixed $k$, there exist arbitrarily large $2k$-edge-connected $k$-obstructions. 
The following is our main result on feasibility for fixed odd inversions.

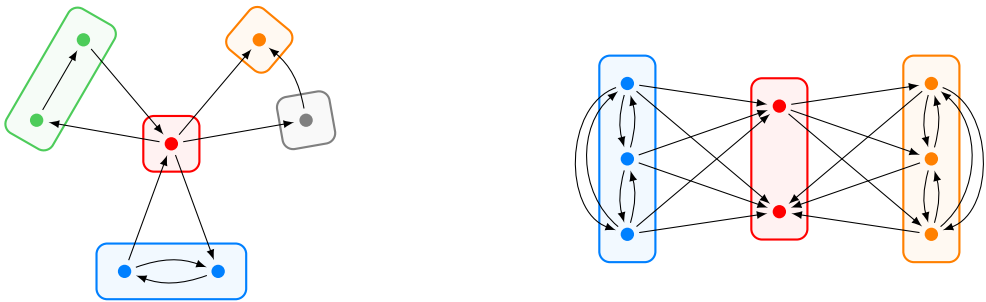
\begin{figure}
    \centering
    \begin{tikzpicture}[region/.style={draw=black, rounded corners, thick,inner sep=6pt, outer sep=2pt}]
        \newcommand{\rad}{1.8}
        \newcommand{\dev}{20}
        \node[mvertex, red] (y) at (0,0) {};
        \begin{scope}[on background layer]
            \node[region,red, fill=red!5, fit=(y)] (Y) {};
        \end{scope}
        \begin{scope}[rotate=30+\dev]
            \draw [draw=orange, rounded corners, thick, fill=orange!5] ($(0:\rad)-(0.35,0.35)$) rectangle ++(0.7,0.7);
            \node[mvertex, orange] (x1a) at (0:\rad) {};
        \end{scope}
        \begin{scope}[rotate=30-\dev]
            \draw [draw=gray, rounded corners, thick, fill=gray!5] ($(0:\rad)-(0.35,0.35)$) rectangle ++(0.7,0.7);
            \node[mvertex, gray] (x1b) at (0:\rad) {};
        \end{scope}
        
        \begin{scope}[rotate=150]
            \node (pad) at ($(0-\dev:\rad)+(-0.35,-0.35)$) {};
            \draw [draw=g-green, rounded corners, thick, fill=g-green!5] (pad) rectangle ++(0.7,1.9);
            \node[mvertex, g-green] (x2a) at (0+\dev:\rad) {};
            \node[mvertex, g-green] (x2b) at (0-\dev:\rad) {};
        \end{scope}
    
        \node[mvertex, g-blue] (x3a) at (-90+\dev:\rad) {};
        \node[mvertex, g-blue] (x3b) at (-90-\dev:\rad) {};

        \begin{scope}[on background layer]
            \node[region,g-blue, fill=g-blue!5, fit=(x3a)(x3b)] (X3) {};
        \end{scope}
        
        \draw[sarc] (y) to (x1a);
        \draw[sarc] (y) to (x1b);
        \draw[sarc] (y) to (x2a);
        \draw[sarc] (x2b) to (y);
        \draw[sarc] (y) to (x3a);
        \draw[sarc] (x3b) to (y);
        
        \draw[sarc] (x1b) to[out=120-\dev, in=-60+\dev] (x1a);
        \draw[sarc] (x2a) to (x2b);
        \draw[sarc] (x3a) to[bend left=20] (x3b);
        \draw[sarc] (x3b) to[bend left=20] (x3a);

        \begin{scope}[xshift=8cm, yshift=-1.2cm]
        \node[mvertex, red] (y1) at (0,0.3) {};
        \node[mvertex, red] (y2) at (0,1.7) {};
        \begin{scope}[on background layer]
            \node[region,red, fill=red!5, fit=(y1)(y2)] {};
        \end{scope}

        \foreach \i in {1,2,3}{
            \node[mvertex, g-blue] (xa\i) at (-2,-1+\i) {};
            \draw[sarc] (xa\i) to (y1);
            \draw[sarc] (xa\i) to (y2);}
            
        \begin{scope}[on background layer]
            \node[region, g-blue, fill=g-blue!5, fit=(xa1)(xa3)] {};
        \end{scope}
        
        \foreach \i in {1,2,3}{
            \node[mvertex, orange] (xb\i) at (2,-1+\i) {};
            \draw[sarc] (xb\i) to (y1);
            \draw[sarc] (y2) to (xb\i);
        }
        
        \begin{scope}[on background layer]
            \node[region, orange, fill=orange!5, fit=(xb1)(xb3)] {};
        \end{scope}
        
        \draw[sarc] (xa1) to[out=140, in=-140] (xa3);
        \draw[sarc] (xa3) to[out=-160, in=160] (xa1);
        
        \draw[sarc] (xb1) to[out=40, in=-40] (xb3);
        \draw[sarc] (xb3) to[out=-20, in=20] (xb1);
        
        \draw[sarc, bend right=15] (xa1) to (xa2);
        \draw[sarc, bend right=15] (xa2) to (xa1);
        \draw[sarc, bend right=15] (xa3) to (xa2);
        \draw[sarc, bend right=15] (xa2) to (xa3);
        
        \draw[sarc, bend right=15] (xb1) to (xb2);
        \draw[sarc, bend right=15] (xb2) to (xb1);
        \draw[sarc, bend right=15] (xb3) to (xb2);
        \draw[sarc, bend right=15] (xb2) to (xb3);
        \end{scope}
    \end{tikzpicture}
    \caption{Two examples of $k$-obstructions, with $k=1$ (left) and $k=3$ (right). In both cases, each color illustrates a part of the partition $(X_1,\dots,X_r,Y)$, the vertices of $Y$ being depicted in red.}
    \label{fig:k_obstructions}
\end{figure}

\begin{restatable}{theorem}{feasabilityodd}
    \label{thm:invertibility_odd_case}
    Let $k$ be a positive integer, $p\geq 3$ be an odd integer and $D$ be a digraph of order $n \geq \max\{p+2,4k+2\}$. Then $D$ can be made $k$-arc-strong using \eqp{p}-inversions if and only if it is $2k$-edge-connected and it is not a $k$-obstruction.
\end{restatable}

\Cref{thm:invertibility_even_case,thm:invertibility_odd_case} are proved in Section~\ref{sec:invertibility}. From these, we can conclude that for any fixed $p$ and $k$, we can decide in polynomial time whether a given digraph can be made $k$-arc-strong by inversions of size exactly $p$. Moreover, we show in Section~\ref{sec:identifying_obstructions} that obstructions can be recognised in polynomial time, which implies the following stronger statement.

\begin{restatable}{corollary}{corkernelkplusp}
    \label{cor:kernel_k+p}
    The problem of deciding whether a given directed graph can be made $k$-arc-strong by inversions of size exactly $p$ admits a kernel of size $O(k+p)$.
\end{restatable}

Theorems~\ref{thm:invertibility_even_case} and~\ref{thm:invertibility_odd_case} raise the question whether the minimum size bound is really necessary or just an artefact of a suboptimal proof technique. The following complexity result shows that at least the bound $n \geq p+2$ is crucial in this characterization, even for oriented graphs.

\begin{restatable}{theorem}{thmNPpush}
    \label{thm:NP_hardness_n-1_inversions}
    It is NP-hard to decide whether a given an oriented graph on $n$ vertices can be made strongly connected by inversions of size exactly $n-1$. 
\end{restatable}

The statement above follows from a hardness result due to Klostermeyer~\cite{klostermeyerAC51} concerning a closely related problem of {\it vertex pushing}. We state the problem and give the simple reduction in Section~\ref{sec:NP_hardness} for completeness.

\paragraph{Exact and Approximate Computation}

Here, we focus on the minimum number of inversions of bounded size needed to make a given digraph $k$-arc-strong. First consider the case of inversions of size 2. For oriented graphs, this case corresponds exactly to the operation of arc flipping. This problem can be solved in polynomial time~\cite{FRANK198297, bang2009} using a powerful orientation theorem of Frank~\cite{FRANK198297}. 
Recall that by \Cref{thm:invertibility_even_case}, for every even $p\geq 2$ and every $k \ge 1$, we have that every digraph of a certain minimum size can be made $k$-arc-strong if and only if its underlying graph is $2k$-edge-connected. Actually, it follows from \Cref{thm:nashwilliams} that is presented later that for $p=2$, this minimum size condition can be omitted. In particular, we obtain that a digraph can be made $k$-arc-strong by $(=2)$-inversions if and only if it can be made $k$-arc-strong by arc reversals.
In this light, it is interesting to observe that there is a significant difference between 2-inversions and arc reversals when it comes to minimizing the number of necessary operations. Indeed, the number of arc flips necessary to establish $k$-arc-strong connectivity can be different from the number of necessary \eqp{2}-inversions. See Figure~\ref{fig:difference_swap_inversion} for an illustration.
Nevertheless, we can use the result of Frank in a slightly different way to establish that also the case of digraphs can be solved in polynomial time.

\begin{figure}
    \centering
    \begin{tikzpicture}
        \node[mvertex, label=left:$v_1$] (v1) at (-150:1.5) {};
        \node[mvertex, label=below:$v_2$] (v2) at (0,0) {};
        \node[mvertex, label=right:$v_3$] (v3) at (-30:1.5) {};
        \node[mvertex, label=above:$v_4$] (v4) at (90:1.5) {};

        \draw[tarc] (v1) to (v2);
        \draw[tarc] (v1) to[out=80, in=-140] (v4);
        \draw[tarc] (v2) to (v3);
        \draw[tarc] (v4) to[out=-40, in=100] (v3);
        
        \draw[tarc] (v2) to[out=105, in=-105] (v4);
        \draw[tarc] (v4) to[out=-75, in=75] (v2);
        
        \draw[tarc] (v1) to[out=7, in=173] (v3);
        \draw[tarc] (v3) to[out=-165, in=-15] (v1);
    \end{tikzpicture}
    \caption{An example of a digraph where the number of necessary \eqp{2}-inversions to make it 2-arc-strong does not correspond to the necessary number of arc reversals. Indeed, the reversal of the arc from $v_1$ to $v_3$ makes the graph 2-arc-strong, but at least two \eqp{2}-inversions are required.}
    \label{fig:difference_swap_inversion}
\end{figure}
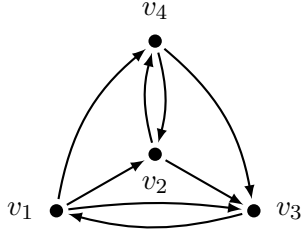

\begin{restatable}{theorem}{twoinfmin}
    \label{thm:poly-2-inversions}
    Given a digraph $D$ and an integer $k$, in polynomial time one can compute the minimum number of \eqp{2}-inversions to make $D$ $k$-arc-strong.
\end{restatable}

The proof of Theorem~\ref{thm:poly-2-inversions} is provided in Section~\ref{sec:edge_connectivity_tools} for completeness (see Theorem~\ref{thm:frank} below).
Our next result displays a drastic change in the behavior of the problem when considering $p \geq 3$. Namely, like for most inversion problems, it turns out that an exact computation of the number of necessary inversions is out of reach. Moreover, the following result also excludes the existence of a PTAS. 

\begin{restatable}{theorem}{thminapproximability}
    \label{thm:inapproximability}
    For every fixed $p\geq 3$, there exists $\epsilon_{p} >0$ such that, unless $P=NP$, for every $k\geq 1$, in polynomial time one cannot approximate the minimum number of \leqp{p}-inversions to make an oriented graph $k$-arc-strong within a factor of $1+\epsilon_{p}$. 
\end{restatable}

The proof of Theorem~\ref{thm:inapproximability} is given in Section~\ref{sec:inapproximability}.
Recall that, by \Cref{inapprac}, there is no hope to approximate the number of inversions of unbounded size that are necessary to make a given digraph $k$-arc-strong within a constant factor for any $k \geq 1$. Moreover, when the objective is to make a given digraph acyclic, then there is also no polynomial-time constant factor approximation for the minimum number of bounded-size inversions due to the result of Guruswami~\cite{doi:10.1137/090756144}. Surprisingly, the situation is much brighter in our setting. We prove the following result.

\begin{theorem}
    \label{thm:APX}
    For all integers $k\geq 1$ and $p \geq 2$ and every $\epsilon>0$, there exists a polynomial-time algorithm that computes a $(4k-2+\epsilon)$-approximation for the number of necessary \leqp{p}-inversions to make a given digraph $k$-arc-strong. 
\end{theorem}

Observe that by \Cref{thm:APX}, there exists a $(2+\epsilon)$-approximation algorithm for any $\epsilon>0$ in the case that we aim for a strongly connected digraph.
Actually, we will show more technical results that yield slightly better approximation ratios for many special cases. In particular, we obtain the following result.

\begin{theorem}
    \label{thm:APX_p=4}
    There exists a polynomial-time algorithm that computes a $(\frac{3}{2}+\epsilon)$-approximation for the minimum number of \leqp{4}-inversions to make a given digraph strongly connected. 
\end{theorem}

We prove Theorems~\ref{thm:APX} and \ref{thm:APX_p=4} in Section~\ref{sec:APX_algorithm}.

\paragraph{Parameterized Complexity}

Besides approximation, parameterized complexity is another common way to make progress on NP-hard problems.  In the current setting, the inversion size $p$ and the solution size $\ell$ may serve as parameters. Observe that all the problems in consideration admit trivial XP algorithms in $p$ and $\ell$, which consist of brute force enumerating all appropriate inversions. Hence, it is natural to ask about FPT results. We prove two results in that regard. First, we consider the case that $p$ serves as a parameter and $\ell$ is fixed. We show that for any $k\ge 2$, there is no hope for an FPT algorithm even in oriented graphs.

\begin{restatable}{theorem}{Whardnessinvkpbyp}
    \label{thm:W1_hardness_invkp_by_p}
    For every $k \geq 2$, it is $W[1]$-hard with respect to $p$ to decide whether a given oriented graph can be made $k$-arc-strong by a single inversion.
\end{restatable}

Observe that this result is incomparable with~\Cref{thm:NP_hardness_n-1_inversions}.
Next, we consider the case where $p$ is fixed and $\ell$ serves as the parameter. Here, the cases of oriented graphs and digraphs remain open. However, we show that for multidigraphs not even the simplest case $p=2$ is tractable.

\begin{restatable}{theorem}{Whardnessinvkpbyell}
    \label{thm:W1_hardness_invkp_by_ell}
    Given a multidigraph $D$, it is $W[1]$-hard with respect to $\ell$ to decide whether $D$ can be made $2$-arc-strong by at most $\ell$ \eqp{2}-inversions.
\end{restatable}

Observe that this result is incomparable with \Cref{thm:NP_hardness_2inversion_multidigraphs}.
The proofs of Theorems~\ref{thm:W1_hardness_invkp_by_p} and~\ref{thm:W1_hardness_invkp_by_ell} are given in Section~\ref{sec:W1_hardness}.

\section{Preliminaries}
\label{sec:preliminaries}

Here, we give some preliminaries we need for our main results. More precisely, in \Cref{sec:notation}, we collect some notation and definitions, and in \Cref{sec:edge_connectivity_tools}, we give a collection of preliminary results used later on.

\subsection{Notation}
\label{sec:notation}

Let $k$ be a positive integer. The set of integers $\{1,\dots,k\}$ is denoted by $[k]$. 
We use $\mathbb{Z}_{\geq 0}$ and $\mathbb{R}_{\geq 0}$ for the set of nonnegative integers and nonnegative reals, respectively.
For any set $V$, we denote by $\binom{V}{k}$ (respectively $\binom{V}{\leq k}$) the set of subsets of $V$ of size exactly $k$ (respectively at most~$k$). Given two sets $U,V$, we denote by $U\XOR V$ the symmetric difference of $U$ and $V$, that is $U\XOR V = (U\cup V) \setminus (U\cap V)$.

\medskip

Let $G$ be a multigraph and let $k$ be an integer. For an edge $e=uv \in E(G)$, we say that $u$ and $v$ are the {\it endvertices} of $e$ and that $e$ {\it links} $u$ and $v$. For an edge $e \in E(G)$, its {\it multiplicity $\mu_G(e)$ in $G$} is the number of edges in $E(G)$ having the same endvertices as $e$.
A {\it simple edge} of $G$ is an edge of $G$ with multiplicity one. Given a vertex $v$, its neighborhood in $G$ is denoted by $N_G(v)$ and its closed neighborhood by $N_G[v]$ (that is, $N_G[v] = N_G(v)\cup \{v\}$). The graph $G$ is {\it $k$-edge-connected} if the removal of any set of at most $k-1$ edges of $G$ yields a connected graph. 
We say that $G$ is {\it bipartite} if it admits a partition $(A,B)$ of its vertex-set such that every edge $e\in E(G)$ has exactly one extremity in $A$. Such a partition is called a {\it bipartition} of $G$. A graph $G$ such that for every $\{u,v\}\in \binom{V(G)}{2}$, we have that $E(G)$ contains exactly one edge linking $u$ and $v$ is called a {\it complete graph} or a {\it clique}. For a positive integer $n$, we use $K_n$ for the unique clique on $n$ vertices. Given a coloring $\phi\colon E(K_n)\rightarrow S$ for some set $S$, we say that a subgraph $H$ of $K_n$ is {\it monochromatic} under $\phi$ if $|\phi(E(H))|=1$.

Given two vertices $u,v \in V(G)$, we denote by $\lambda_G(u,v)$ the maximum number of pairwise edge-disjoint paths from $u$ to $v$.
The {\it edge-connectivity} of $G$, denoted by $\lambda(G)$, is the largest integer $k$ such that $G$ is $k$-edge-connected, with the convention $\lambda(K_0) = \lambda(K_1) = +\infty$. 
Given two disjoint sets $X$ and $Y$ of vertices, we say that an edge $e\in E(G)$ {\it links} $X$ and $Y$ if $e$ links a vertex in $X$ and a vertex in $Y$. We denote the set of edges linking $X$ and $Y$ by $\Eset{G}{X}{Y}$, and its size by $d_G(X,Y)$.
Given a set of vertices $\emptyset \subsetneq S\subsetneq V(G)$, 
$\Eset{G}{S}{V(G)\setminus S}$
is called a {\it cut} of $G$, and for the sake of conciseness it is denoted by $\delta_G(S)$. The size of the cut $\delta_G(S)$ is denoted by $d_G(S)$.
The set of vertices of $V(G)\setminus S$ with a neighbor in $S$ are  denoted by $N_G(S)$.
For a single vertex $v$, we use $d_G(v)$ for $d_G(\{v\})$.

Note that $G$ is $k$-edge-connected for some positive integer $k$ if and only if all cuts of $G$ have size at least $k$. For a set $X \subseteq V(G)$, we use $G[X]$ for the subgraph of $G$ induced by $X$. The set $X\subseteq V(G)$ is called {\it independent} in $G$ if $E(G[X])=\emptyset$. We use $\alpha(G)$ for the maximum size of an independent set in $G$. Given two disjoint sets $X_1,X_2 \subseteq V(G)$, we use $G[X_1,X_2]$ for the subgraph of $G$ whose vertex set is $X_1 \cup X_2$ and that contains all edges of $E(G)$ that have one endvertex in $X_1$ and one endvertex in $X_2$. We say that $G[X_1,X_2]$ is {\it complete} if for every $x_1\in X_1$ and every $x_2 \in X_2$, we have that $E(G[X_1,X_2])$ contains at least one edge linking $x_1$ and $x_2$.

\medskip

In this paper, digraphs may contain two arcs in opposite directions between the same pair of vertices, but no parallel arcs, while a {\it multidigraph} may contain both opposite and parallel arcs. A pair of opposite arcs between $u$ and $v$ is called a {\it digon} and is denoted by $\dig{u,v}$.
A {\it simple arc} is an arc that does not belong to a digon.
An {\it oriented graph} is a digraph without any digon.
Let $D$ be a digraph. We say that an arc $a \in A(D)$ {\it links} two vertices $u,v \in A(D)$ if $a$ goes from $u$ to $v$ or $a$ goes from $v$ to $u$. If $a$ links two vertices $u$ and $v$, we say that $u$ and $v$ are the {\it endvertices} of $a$. {\it Reversing} an arc from a vertex $u$ to a vertex $v$ means replacing it by an arc from $v$ to $u$. The {\it underlying graph} $\UG(D)$ of $D$ is the multigraph obtained from $D$ by removing the orientations. Note that the underlying graph of $D$ contains two parallel edges between $u$ and $v$ whenever $\dig{u,v}$ is a digon of $D$.
We say that $D$ is {\it $k$-edge-connected} if its underlying graph is. Further, $D$ is {\it $k$-arc-strong} if the removal of any set of at most $k-1$ arcs of $D$ yields a strongly connected digraph. 
Given a set of vertices $\emptyset \subsetneq S\subsetneq V(D)$, the set of arcs going from $S$ to $V(D)\setminus S$ is a {\it dicut}, and it is denoted by $\delta^+_D(S)$. Further, we use  $\delta^-_D(S)$ for $\delta^+_D(V(D)\setminus S), d_D^+(S)$ for $|\delta^+_D(S)|$ and $d_D^-(S)$ for $|\delta^-_D(S)|$. 
Again, for a single vertex $v$, we use $d_D^+(v)$ and $d_D^-(v)$ for $d_D^+(\{v\})$ and $d_D^-(\{v\})$ respectively.
Similarly to the undirected case, observe that $D$ is $k$-arc-strong if and only if all dicuts of $D$ have size at least $k$. For some $a\in A(D)$, {\it duplicating $a$} is the operation of adding another arc with the same tail and the same head as $a$. An oriented graph whose underlying graph is a complete graph is called a {\it tournament}. We use $\alpha(D)$ for $\alpha(\UG(D))$.

Given two vertices $u,v \in V(D)$, we denote by $\lambda_D(u,v)$ the maximum number of pairwise arc-disjoint directed paths from $u$ to $v$. For a vertex $u$ and a set of vertices $Q$, we denote by $\lambda_{D}(u,Q)$ (resp. $\lambda_D(Q,u)$) the maximum number of pairwise arc-disjoint directed paths from $u$ to $Q$ (respectively from $Q$ to $u$). Given $Q\subseteq V(D)$, we say that $D$ is {\it $2$-arc-strong in} $Q$ if, for every pair of distinct vertices $u,v\in Q$, we have $\lambda_D(u,v) \geq 2$ and $\lambda_D(v,u)\geq 2$.

\medskip

The {\it inversion} of a set $X\subseteq V(D)$ of vertices consists of reversing the direction of every arc of $D[X]$, and we say that we {\it invert} $X$ in $D$.
The inversion of $X$ is an {\it \eqp{p}-inversion} (respectively \leqp{p}-inversion) if $X$ has size $p$ (respectively at most $p$).
After applying the inversion of $X$ in $D$, the resulting digraph is denoted by $\Inv(D, X)$.
If $\Xcal$ is a set of
subsets of $V(D)$, then $\Inv(D; \Xcal)$ is the digraph obtained after inverting each $X\in \mathcal{X}$ one after another (the order of the inversions being irrelevant).
Observe that, if $\dig{u,v}$ is a digon of $D$, then $D$ is isomorphic to $\Inv(D,\{u,v\})$. For a set $X \subseteq V(D)$ and a set $\mathcal{X}\subseteq 2^{V(D)}$, we say that the inversion of $\mathcal{X}$ {\it simulates} the inversion of $X$ if $\Inv(D;\mathcal{X})=\Inv(D,X)$. We here wish to emphasize that for a digraph $D'$, we consider that $D=D'$ holds if $V(D')=V(D)$ and for every ordered pair of vertices $(u,v)\in V(D)$, we have that $A(D)$ and $A(D')$ contain the same number of arcs from $u$ to $v$. In particular, inverting $\{u,v\}$ for a digon $\dig{u,v}$ results in the same digraph.
We say that an edge $e \in E(\UG(D))$ is {\it reversed} by a set $\mathcal{X}$ of inversions, if the corresponding an arc in $D$ is reversed. 

Given two integers $k\geq 1$ and $p\geq 2$, we say that $D$ is {\it $(k,p)$-invertible} if there exists a set $\Xcal$ of subsets of $V(D)$, each of size exactly $p$, such that $\Inv(D;\mathcal{X})$ is $k$-arc-strong. A collection $\Xcal\subseteq \binom{V(D)}{\leq p}$ such that $\Inv(D;\mathcal{X})$ is $k$-arc-strong is called a {\it \InvSeq{k}{p}} of $D$. 
We denote by $\inv{k}{p}(D)$ the minimum size of such a set, with the convention $\inv{k}{p}(D)=+\infty$ when no such set exists.

\subsection{Preliminary results}
\label{sec:edge_connectivity_tools}

We first recall the well-known result that the edge-connectivity of a graph can be computed in polynomial time, see for instance~\cite[Theorem~15.10]{schrijver2002}.

\begin{theorem}
    \label{thm:edge_connectivity_poly}
    The edge-connectivity of a graph can be computed in polynomial time.
\end{theorem}

The following well-known result implies that a digraph $D$ is $(k,2)$-invertible if and only if it is $2k$-edge-connected.

\begin{theorem}[Nash-Williams~\cite{nashwilliamsCJM12}]
    \label{thm:nashwilliams}
    A multigraph $G$ has a $k$-arc-strong orientation if and only if it is $2k$-edge-connected. Moreover, if $G$ has at most $2$ edges between every pair of vertices, then there exists such a $k$-arc-strong orientation without any parallel arcs.
\end{theorem}

The second part of the statement above is actually not stated in~\cite{nashwilliamsCJM12}, but it can be proved as follows. Start from $G$ and, for every pair of vertices $u,v$ with two edges between them, orient them in such a way that $\dig{u,v}$ is a digon. In the resulting partial orientation, for every set of vertices $X$, the number of arcs entering $X$ is equal to the number of arcs leaving $X$. Therefore, the following generalisation of Nash-Williams' Theorem due to Jackson~\cite{jacksonJGT12} (see also~\cite[Theorem~11.9.1]{bang2009}) shows that the partial orientation of $G$ can be completed into a $k$-arc-strong orientation of $G$.

\begin{theorem}[Jackson~\cite{jacksonJGT12}]
    Let $D=(V,A)$ and $G=(V,E)$ be a digraph and a graph on the same vertex set. If, for every $\emptyset \subsetneq X\subsetneq V$, 
    $d^+_D(X) + \frac{1}{2}d_G(X) \geq k$,
    then there exists an orientation $D'$ of $G$ such that $D\cup D'$ is $k$-arc-strong.
\end{theorem}

We say that a multigraph $G$ is {\it minimally $k$-edge-connected} if $\lambda(G) = k$ and $\lambda(G-e) \leq k-1$ for every edge $e$ of $G$.
Note that, unless $G\in \{K_0,K_1\}$, we have that  $G$ is minimally $k$-edge-connected if and only if it is $k$-edge-connected and every edge of $G$ belongs to a cut of size $k$. The following well-known result of Mader draws a connection between the edge-connectivity and the density of a graph.
The graph $G$ is {\it $k$-degenerate} if every non-empty subgraph $H$ of $G$ contains a vertex of degree at most $k$.

\begin{theorem}[Mader~\cite{maderMA191}]
    \label{thm:mader}
    Every minimally $k$-edge-connected multigraph is $k$-degenerate. 
\end{theorem}

In particular, if $G$ is a minimally $k$-edge-connected multigraph of order $n$, then it follows from the theorem above that 
$|E(G)|\leq k(n-1)$, and if $G$ is a graph, then $|E(G)| \leq kn - \binom{k+1}{2}$.
We also make use of the following similar proposition.

\begin{proposition}
    \label{prop:density_non_k_edge_connected}
    Let $G$ be a multigraph of order $n$ and $k\geq 1$ be an integer. If every subgraph $H$ of $G$ on at least two vertices satisfies $\lambda(H) \leq k$, then
    $|E(G)| \leq k(n-1)$.
\end{proposition}
\begin{proof}
    We proceed by induction on $n$, the result being trivial when $n\leq 1$. Since $\lambda(G) \leq k$, $G$ admits a cut $\delta_G(S)$ of size at most $k$. Let $G_1$ and $G_2$ be the subgraphs of $G$ induced by $S$ and $V(G)\setminus S$ respectively. It follows by induction that
    \[
        |E(G)| \leq |E(G_1)| + |E(G_2)|+k \leq k(|V(G_1)| + |V(G_2)| -1) = k(n-1).\qedhere
    \]
\end{proof}

Similarly, a digraph $D$ is {\it minimally $k$-arc-strong} if it is $k$-arc-strong and removing any arc of $D$ yields a digraph that is not $k$-arc-strong. Analogously to the theorem of Mader, Dalmazzo~\cite{dalmazzo1977} proved that every minimally $k$-arc-strong (multi)digraph has at most $2k(n-1)$ arcs. A short proof of this result, based on Edmonds' Branching Theorem~\cite{edmonds1973} and given in~\cite[Theorem~5.6.1]{bang2009}, implies that such a digraph is actually the union of $2k$ oriented trees, hence implying the following slightly more general result.

\begin{theorem}[Dalmazzo~\cite{dalmazzo1977}]
    \label{thm:dalmazzo}
    Let $D$ be a minimally $k$-arc-strong digraph. If $H$ is a subdigraph of $D$ of order $n$, then $H$ has at most $2k(n-1)$ arcs.
\end{theorem}

We further make use of the following easy proposition.

\begin{proposition}
    \label{prop:showing_k_strong}
    Let $D$ be a digraph, $k\geq 1$ an integer, and $\emptyset \subsetneq W\subseteq V(D)$ such that $D[W]$ is $k$-arc-strong. If, for every $v\in V(D)\setminus W$, we have $\lambda_D(v,W) \geq k$ and $\lambda_D(W,v) \geq k$, then $D$ is $k$-arc-strong.
\end{proposition}

We conclude this section with a short proof of Theorem~\ref{thm:poly-2-inversions}, that we restate here in the following more technical form for convenience.

\begin{theorem}
    \label{thm:frank}
    For every integer $k\geq 1$, there exists a polynomial-time algorithm that, given a $2k$-edge-connected digraph $D$, outputs $\inv{k}{2}(D)$ and a \InvSeq{k}{2} of $D$ of this size.
\end{theorem}

Theorem~\ref{thm:frank} follows from a powerful orientation theorem of Frank. In order to state it formally, we need some preliminaries. Given a graph $G$, we let $Q(G)\subseteq E(G)\times V(G)$ be the set that contains $(e,v)$ for every $e \in E(G)$ and every endvertex $v$ of $e$. Given a graph $G$, a weight function $w:Q(G)\rightarrow \mathbb{Z}_{\geq 0}$, and an orientation $\vec{G}$ of $G$, the {\it weight} of $\vec{G}$ is the sum of $w(e,v)$ over all $e \in E(G)$ and endvertices $v$ of $e$ such that $e$ is oriented towards $v$ in $\vec{G}$. We are now ready to state the orientation theorem of Frank from \cite{frankNHMS66} and derive Theorem~\ref{thm:frank}.

\begin{theorem}[\cite{frankNHMS66}]
    \label{csadadf}
    Let $G$ be a $2k$-edge-connected graph for some integer $k\geq 1$ and $w\colon Q(G)\rightarrow \mathbb{Z}_{\geq 0}$ a weight function. Then a minimum weight $k$-arc-strong orientation of $G$ can be computed in polynomial time.
\end{theorem}

\begin{proof}[Proof of Theorem~\ref{thm:frank}]
    Let $D$ be a digraph and let $G$ be its underlying graph. We define a weight function $w:Q(G)\rightarrow \mathbb{Z}_{\geq 0}$ in the following way: let $(e,v)\in Q(G)$. If $e$ is oriented towards $v$ in $D$, we set $w(e,v)=0$. If $e$ is not oriented towards $v$ in $D$ and $\mu_G(e)>1$, we set $w(e,v)=|A(D)|+1$. If $e$ is not oriented towards $v$ in $D$ and $\mu_G(e)=1$ then we set $w(e,v)=1$. This finishes the description of $w$. Clearly, $G$ and $w$ can be computed from $D$ in polynomial time. By \Cref{csadadf}, it hence suffices to prove, for every $1\leq \ell\leq |A(D)|$, that $G$ admits a $k$-arc-strong orientation of weight at most $\ell$ if and only if $\inv{k}{2}(D)\leq \ell$.

    First suppose that $\inv{k}{2}(D)\leq \ell$, so there exists a \InvSeq{k}{2} $\mathcal{X}$ of $D$ with $|\mathcal{X}|\leq \ell$. We may suppose that $\mathcal{X}$ is inclusionwise minimal with that property. Let $D'=\Inv(D;\mathcal{X})$ and observe that $D'$ is in particular an orientation of $G$. Moreover, by minimality of $\mathcal{X}$ and as $D$ is a digraph, for every $\{u,v\}\in \mathcal{X}$, we have that $A(D)$ contains exactly one arc linking $u$ and $v$. Hence, by construction, we have that the weight of $D'$ is at most $\ell$.

    Now suppose that $G$ has a $k$-arc-strong orientation $\vec{G}$ of weight at most $\ell$. Let $\mathcal{X}$ be the set that contains both endvertices of every $e \in E(G)$ that has a distinct orientations in $D$ and $\vec{G}$. By the construction of $w$ and \Cref{thm:nashwilliams}, we obtain that every arc in $D$ that is contained in a digon also exists in $\vec{G}$.  By construction, this yields that the size of $\mathcal{X}$ is exactly the weight of $\vec{G}$. Moreover, we have $\Inv(D;\mathcal{X})=\vec{G}$. In particular, we obtain that $\Inv(D;\mathcal{X})$ is $k$-arc-strong. This finishes the proof.
\end{proof}

\section{NP-hardness results}
\label{sec:NP_hardness}

In this section, we show how Theorems~\ref{thm:NP_hardness_2inversion_multidigraphs} and~\ref{thm:NP_hardness_n-1_inversions} are derived from two closely related known results. The two proofs are given in the two next subsections.

\subsection{Proof of Theorem~\ref{thm:NP_hardness_2inversion_multidigraphs}}

Here, we give the proof of \Cref{thm:NP_hardness_2inversion_multidigraphs}. Formally, we prove the NP-hardness  of the following problem.

\defproblem{Multidigraph (2,2)-Invertibility (M22INV)}{A multidigraph $D$.}{Is there a set $\mathcal{X}\subseteq \binom{V(D)}{2}$ such that $\Inv(D;\mathcal{X})$ is 2-arc-strong?}

The reduction relies on the NP-hardness of a problem introduced in \cite{HORSCH2021103292}. In a strongly connected digraph $D$, we say that an arc $a \in A(D)$ is {\it deletable} if $D\setminus a$ is strongly connected. In \cite{HORSCH2021103292}, the following problem was considered:

\defproblem{Deletability Orientation (DO)}{A graph $G$, a set $F \subseteq E(G)$.}{Is there a strongly connected orientation $\vec{G}$ of $G$ such that every element of $F$ corresponds to a deletable arc in $\vec{G}$?}

In \cite{HORSCH2021103292}, the following result was proved.

\begin{theorem}[\cite{HORSCH2021103292}]
    \label{estf}
    DO is NP-hard.
\end{theorem}

We wish to remark that in \cite{HORSCH2021103292}, the result is stated for multigraphs. However, the result for graphs is an immediate consequence of the reduction.
We are now ready to prove the following restatement of \Cref{thm:NP_hardness_2inversion_multidigraphs}.

\begin{theorem}
    M22INV is NP-hard.
\end{theorem}
\begin{proof}
    We prove the result by a reduction from DO, which is NP-hard by \Cref{estf}. Let $(G,F)$ be an instance of DO and let $\{v_1,\ldots,v_n\}$ be an arbitrary enumeration of $V(G)$. We now construct an instance $D$ of M22INV. First, we set $V(D)=V(G)$. Next, for every $i,j \in [n]$ with $i<j$ such that $F$ contains an edge linking $v_i$ and $v_j$, we let $A(D)$ contain an arc from $v_i$ to $v_j$. Further, for every $i,j \in [n]$ with $i<j$ such that $E(G)\setminus F$ contains an edge linking $v_i$ and $v_j$, we let $A(D)$ contain two parallel arcs from $v_i$ to $v_j$. This finishes the description of $D$. It is easy to see that $D$ can be computed from $G$ in polynomial time. It hence suffices to prove that $D$ is a yes-instance of M22INV if and only if $(G,F)$ is a yes-instance of DO.

First suppose that $(G,F)$ is a yes-instance of DO, so there exists a strongly connected orientation $\vec{G}$ of $G$ in which every $f \in F$ corresponds to a deletable arc $\vec{f}$. We now let $\mathcal{X}$ contain the set $\{v_i,v_j\}$ for all $i,j \in [n]$ with $i<j$ such that $A(\vec{G})$ contains at least one arc from $v_j$ to $v_i$. 
Let $D'=\Inv(D;\mathcal{X})$. We show that $D'$ is $2$-arc-strong. Let $a \in A(D')$ and let $i,j \in [n]$ with $i<j$ such that $v_i$ and $v_j$ are the endvertices of $a$. It follows by construction that $v_i$ and $v_j$ are linked by an edge $e \in E(G)$. If $e \in E(G)\setminus F$, then we obtain that $\vec{G}$ is a directed subgraph of $D' \setminus a$. As $\vec{G}$ is strongly connected, so is $D' \setminus a$. If $e \in F$, then we obtain that $\vec{G}\setminus \vec{e}$ is a directed subgraph of $D' \setminus a$. As $\vec{G}\setminus \vec{e}$ is strongly connected, so is $D' \setminus a$. In either case, it follows that $D' \setminus a$ is strongly connected. It follows that $D'$ is
2-arc-strong and hence $D'$ is a yes-instance of M22INV.

Now suppose that $D$ is a yes-instance of M22INV, so there exists a collection $\mathcal{X}\subseteq \binom{V(D)}{2}$  such that $D'$ is 2-arc-strong, where $D'=\Inv(D;\mathcal{X})$. We now define an orientation $\vec{G}$ of $G$. Let $e \in E(G)$ and let $i,j \in [n]$ with $i<j$ such that $v_i$ and $v_j$ are the endvertices of $e$. If $\{v_i,v_j\}\in \mathcal{X}$, we orient $e$ from $v_j$ to $v_i$. Otherwise, we orient $e$ from $v_i$ to $v_j$. Let $\vec{G}$ be the thus obtained orientation. First observe that $D'$ can be obtained from $\vec{G}$ by duplicating certain arcs. Hence, as $D'$ is strongly connected, so is $\vec{G}$. Now consider some $f \in F$ and let $\vec{f}$ be the corresponding arc in $\vec{G}$. Further, let $a$ be the unique arc in $A(D')$ that has the same endvertices as $f$. Then $D'\setminus a$ can be obtained from $\vec{G}\setminus \vec{f}$ by duplicating certain arcs. As $D'$ is 2-arc-strong, we have that $D'\setminus a$ is strongly connected and hence so is $\vec{G}\setminus \vec{f}$. It follows that $(G,F)$ is a yes-instance of DO.
\end{proof}

\subsection{Proof of Theorem~\ref{thm:NP_hardness_n-1_inversions}}
\label{subsec:NP_hardness}

Here, we give the proof of Theorem~\ref{thm:NP_hardness_n-1_inversions} that we first recall here for convenience.

\thmNPpush*

Given a digraph $D$ and a vertex $v$, {\it pushing $v$ in $D$} means flipping the orientation of every arc  of $D$ incident to $v$. For a set $X\subseteq V(D)$ of vertices, we denote by $\Push(D;X)$ the digraph obtained after pushing all vertices in $X$, one after the other. Note that the resulting orientation does not depend on the order in which the vertices of $X$ are pushed.
We consider the following problem with the corresponding hardness result due to Klostermeyer~\cite{klostermeyerAC51}.

\defproblem{Strong Pushing}{An oriented graph $D$.}{Is there a set $X\subseteq V(D)$ such that $\Push(D;X)$ is strongly connected?}

\begin{theorem}[\cite{klostermeyerAC51}]
    \label{thm:NP_hardness_strong_pushing}
    The {\sc Strong Pushing} problem is NP-hard.
\end{theorem}

Now Theorem~\ref{thm:NP_hardness_n-1_inversions} follows from Theorem~\ref{thm:NP_hardness_strong_pushing} together with the following.

\begin{lemma}
    Let $D=(V,A)$ be a digraph. Then $D$ is a yes-instance of {\sc Strong Pushing} if and only if $D$ is $(1,n-1)$-invertible.
\end{lemma}
\begin{proof}
    Let $H$ be the digraph obtained from $D$ by reversing all arcs of $D$. 
    Let $v$ be any vertex of $D$, then observe that $\Push(D;\{v\})$ is exactly $\Inv(H, V\setminus \{v\})$. 
    With an easy induction, we thus obtain that, for every set $X\subseteq V(D)$,
    \[
    \Push(D;X) = \left\{
    \begin{array}{ll}
        \Inv(D;\mathcal{X}) & \text{if $|X|$ is even,}\\
        \Inv(H;\mathcal{X}) & \text{otherwise,}
    \end{array}
    \right.
    \]
    where $\mathcal{X} = \{V\setminus x : x\in X\}$.
    Clearly, a digraph is strongly connected if and only if the digraph obtained from it by reversing all its arcs is strongly connected. The statement thus follows from the equality above.
\end{proof}

\section{Characterising the \texorpdfstring{$(k,p)$}{(k,p)}-invertible digraphs}
\label{sec:invertibility}

This section is devoted to the proof of Theorems~\ref{thm:invertibility_even_case} and~\ref{thm:invertibility_odd_case} and Corollary~\ref{cor:kernel_k+p}.
That is, for every pair of integer $k\geq 1$ and $p\geq 2$, we give an exact characterisation of the $(k,p)$-invertible digraphs of order $n$, assuming that $n$ is an integer large enough compared to $k$ and $p$. 
When $p$ is even, it turns out that all (sufficiently large) digraphs are $(k,p)$-invertible. This follows from the following known proposition, see Propositions~3.4 and~3.5 in~\cite{bangjensen2025makingorientedgraphacyclic}.

\begin{proposition}[\cite{bangjensen2025makingorientedgraphacyclic}]
    \label{prop:simulating_2_inversions}
    Let $p\geq 2$ be an even integer and $D$ be a digraph of order $n\geq p+2$. If $D$ contains a digon or a pair of non-adjacent vertices, then any \eqp{2}-inversion can be simulated by a set of \eqp{p}-inversions.
\end{proposition}

We can now derive Theorem~\ref{thm:invertibility_even_case}, that we first recall here for convenience.

\feasabilityeven*

\begin{proof}
    First suppose that $D$ contains two non-adjacent vertices or at least one digon. Then any \eqp{2}-inversion can be simulated by a set of \eqp{p}-inversions by Proposition~\ref{prop:simulating_2_inversions}. The statement now follows from \Cref{thm:nashwilliams}. Otherwise, $D$ is a tournament. Since $D$ has order at least $2k+2$, removing an arbitrary arc of $D$ yields a $2k$-edge-connected digraph $D'$ with two non-adjacent vertices. By the above, we obtain that $D'$ can be made $k$-arc-strong by applying a set of \eqp{p}-inversions. The same set also makes $D$ $k$-arc-strong.
\end{proof}

When $p$ is odd, the situation is different, as explained in the introduction. 
Recall that, for a positive integer $k$, a {\it $k$-obstruction} is any digraph $D$, with underlying graph $G$, admitting a partition $(X_1,\dots,X_r,Y)$ of its vertex set into non-empty parts such that, for $X=\bigcup_{i\in [r]} X_i$, we have:
\begin{enumerate}[label=(\roman*)]
    \item for every $i\in [r]$, $d_G(X_i) = 2k$;
    \label{enumitem:kobstruction:i}
    \item for every $x\in X$ and $y\in Y$, there is exactly one edge linking $x$ and $y$ in $G$;
    \label{enumitem:kobstruction:ii}
    \item $d^+_D(X) - \frac{1}{2}|X|\cdot |Y|$ is an odd integer.
    \label{enumitem:kobstruction:iii}
\end{enumerate}
Recall that an illustration of $k$-obstructions was provided in Figure~\ref{fig:k_obstructions}. Given a $k$-obstruction, we call a partition $(X_1,\ldots,X_r,Y)$ of $V(D)$ satisfying Properties~\ref{enumitem:kobstruction:i},~\ref{enumitem:kobstruction:ii} and~\ref{enumitem:kobstruction:iii} a {\it $k$-certificate} for $D$.
Our goal is now to prove Theorem~\ref{thm:invertibility_odd_case}, that we first recall here for convenience.

\feasabilityodd*

We first give the following simple result that will be reused in \Cref{subsec:3_invertibility}.
\begin{proposition}
    \label{g7zuhoi}
    Let $k$ be a positive integer, $D$ be a $k$-arc-strong digraph, $G$ the underlying graph of $D$, and $(X_1,\ldots,X_r,Y)$ a partition of $V(D)$ satisfying Properties~\ref{enumitem:kobstruction:i} and~\ref{enumitem:kobstruction:ii}. Then $d_D^+(X)=\frac{1}{2}|X||Y|$.
\end{proposition}
\begin{proof}
    Property~\ref{enumitem:kobstruction:i} and $D$ being $k$-arc-strong guarantee that, for every $i\in [r]$ we have $d_{D}^+(X_i) = d^-_{D}(X_i) = k$.  Therefore, for $X=\bigcup_{i\in [r]}X_i$, we have
    \[
    \sum_{x\in X} d^+_{D}(x) = \sum_{i\in [r]} \Big( d^+_{D}(X_i) + |E(G[X_i])|\Big) = \sum_{i\in [r]} \Big( d^-_{D}(X_i) + |E(G[X_i])|\Big) = \sum_{x\in X} d^-_{D}(x),
    \]
    and in particular
    \[
    |E(G[X])| + d^+_{D}({X}) = \sum_{x\in X} d^+_{D}(x) = \sum_{x\in X} d^-_{D}(x) = |E(G[X])| + d^-_{D}(X),
    \]
    which implies that the number of arcs going from $X$ to $Y$ is equal to the number of arcs going from $Y$ to $X$. Since the number of edges in between $X$ and $Y$ in $G$ is precisely $|X|\cdot |Y|$ by Property~\ref{enumitem:kobstruction:ii}, we have 
    $d_{D}^+(X) = \frac{1}{2}|X|\cdot |Y|$
\end{proof}

We are now ready to prove the easy direction of Theorem~\ref{thm:invertibility_odd_case}, namely that $k$-obstructions are not $(k,p)$-invertible.

\begin{lemma}
    \label{lem:obstructions_are_not_inverible}
    For any odd integer $p\geq 3$, if a digraph $D$ is a $k$-obstruction then it is not $(k,p)$-invertible. 
\end{lemma}
\begin{proof}
    Let $D$ be an arbitrary $k$-obstruction with underlying graph $G$, and let $(X_1,\dots,X_r,Y)$ be a $k$-certificate for $D$. 
    Observe that, $p$ being odd, starting from $D$, any \eqp{p}-inversion reverses an even number of edges between $X$ and $Y$, yielding another $k$-obstruction. Therefore, $D$ being arbitrary, we only have to prove that $D$ is not $k$-arc-strong.
    
    For this, let $D'$ be an arbitrary $k$-arc-strong orientation of $G$. By~\Cref{g7zuhoi}, we have 
    $d_{D'}^+(X) = \frac{1}{2}|X|\cdot |Y|$, and
    Property~\ref{enumitem:kobstruction:iii} guarantees that $D\neq D'$. Since $D'$ was taken arbitrarily, it follows that $D$ is not $k$-arc-strong. 
\end{proof}

Now Theorem~\ref{thm:invertibility_odd_case} follows from the combination of Theorem~\ref{thm:nashwilliams}, Lemma~\ref{lem:obstructions_are_not_inverible}, and the following result.

\begin{restatable}{lemma}{leminvertibility}
    \label{lem:invertibility}
    Let $k$ be a positive integer, $p\geq 3$ an odd integer, and $D$ a $2k$-edge-connected digraph of order $n\geq \max\{4k+2,p+2\}$. If $D$ is not $(k,p)$-invertible, then it is a $k$-obstruction.
\end{restatable}

The proof of \Cref{lem:invertibility} is done in three steps, presented respectively in the next three subsections. In Section~\ref{subsec:3_invertibility}, we prove that \Cref{lem:invertibility} holds for $p=3$. In Section~\ref{subsec:simulating_3_inversions}, we prove that, for every odd integer $p\geq 3$, and in any digraph $D$ of order $n\geq p+2$ admitting an independent set of size $3$, every \eqp{3}-inversion can be simulated by a set of \eqp{p}-inversions. We put the pieces together and prove \Cref{lem:invertibility} in Section~\ref{subsec:invertibility}. 

\medskip

Finally, in Section~\ref{sec:identifying_obstructions}, we prove that one can decide in polynomial time whether a given digraph $D$ is a $k$-obstruction. This then allows to derive Corollary~\ref{cor:kernel_k+p}.

\subsection{Characterising the non \texorpdfstring{$(k,3)$}{(k,3)}-invertible digraphs}
\label{subsec:3_invertibility}

This subsection is devoted to the proof of the following lemma, which implies the case $p=3$ of \Cref{lem:invertibility}. 

\begin{lemma}
    \label{lemma:3_invertibility}
    Let $D$ be a $2k$-edge-connected digraph of order $n \geq 4k+2$. If $D$ is not $(k,3)$-invertible, then it is a $k$-obstruction.
\end{lemma}

Before giving the main proof of \Cref{lemma:3_invertibility}, we give a preliminary result that will be reused in \Cref{sec:identifying_obstructions}. To this end, given a graph $G$ and an integer $k$, a subpartition $(X_1,\ldots,X_r)$ of $V(G)$ is called {\it $k$-regular} if $d_G(X_i)=k$ holds for all $i \in [r]$. For some $X \subseteq V(G)$, a {\it $k$-regular partition} of $X$ is a partition of $X$ that is a $k$-regular subpartition of $V(G)$. We prove the following simple result.
\begin{proposition}
    \label{eedqe}
    Let $k$ be a positive integer, $G$ a $k$-edge-connected graph and $X \subseteq V(G)$. Then the following two statements are equivalent:
    \begin{enumerate}[label=(\roman*)]
        \item there exists a $k$-regular partition of $X$ in $G$,
        \label{enumitem:eedqe:1}
        \item for every $x \in X$, there exists a set $S_x\subseteq X$ with $x \in S_x$ and $d_G(S_x)=k$.
        \label{enumitem:eedqe:2}
    \end{enumerate}
\end{proposition}
\begin{proof}
Clearly, \ref{enumitem:eedqe:1} implies \ref{enumitem:eedqe:2}. For the other direction, let 
    $\Xcal = \{X_1,\dots,X_r\}$ be a minimum-size set such that $\bigcup_{i=1}^r X_i=X$ and, for every $i\in [r]$, we have $d_G(X_i) = k$. Such a collection exists by \ref{enumitem:eedqe:2}. Moreover, if $\Xcal$ is a partition of $X$, there is nothing to prove. Assume that this is not the case, that is $X_i\cap X_j \neq \emptyset$ for some $i,j \in [r]$ with $i\neq j$. We then have
        \[
            d_G(X_i \cap X_j) + d_G(X_i \cup X_j) \leq d_G(X_i) + d_G(X_j)  = 2k.
        \]
        Since $G$ is $k$-edge-connected, this implies $d_G(X_i \cap X_j) = d_G(X_i \cup X_j) = k$. Therefore, removing $\{X_i,X_j\}$ from $\Xcal$ and adding $\{X_i \cup X_j\}$ contradicts the choice of $\Xcal$, and the statement follows.
\end{proof}
We are now ready to give the main proof of \Cref{lemma:3_invertibility}.
\begin{proof}[Proof of \Cref{lemma:3_invertibility}]
    Assume for a contradiction that there exists a $2k$-edge-connected digraph of order $n$ at least $4k+2$ that is neither $(k,3)$-invertible nor a $k$-obstruction. Among all such counterexamples, let $D=(V,A)$ be one of for which $|A|$ is minimum, and we let $G$ denote its underlying graph. 
    We say that a pair $\{x,y\} \subseteq V$ of vertices is {\it free} if one of the following holds:
    \begin{itemize}[itemsep=0pt]
        \item $x$ and $y$ are non-adjacent,
        \item $\dig{x,y}$ is a digon of $D$, or
        \item there is exactly one edge $e$ between $x$ and $y$ in $G$, and there exists a set of triplets  $\Xcal = \{X_1,\dots,X_\ell\}$ such that $D$ and $\Inv(D;\Xcal)$ differ exactly on the orientation of $e$. 
    \end{itemize}
    We make use of the following observation repeatedly.
    \begin{claim}
        \label{claim:free_triplets}
        For any triplet $\{x,y,z\} \subseteq V$, if $\{x,y\}$ and $\{y,z\}$ are free pairs, then $\{x,z\}$ is also a free pair.
    \end{claim}
    \begin{proofclaim}
        Consider two sets $\Xcal$ and $\Ycal$ simulating the inversion of $\{x,y\}$ and $\{y,z\}$, respectively.
        Observe that such sets exist as the empty set has this property for a pair linked by no arc or a digon. Then, inverting $\Xcal \XOR \Ycal \XOR \{\{x,y,z\}\}$ in $D$ simulates the inversion of $\{x,z\}$.
    \end{proofclaim}

    We first derive the following claim from the choice of $D$.
    \begin{claim}
        \label{claim:every_e_belongs_to_mincut}
        Every edge $e=\{u,v\}$ of $G$ satisfies at least one of the following properties:
        \begin{enumerate}[itemsep=0pt, label={\rm (\alph*)}]
            \item $e$ belongs to a cut of size $2k$; 
            \label{enumitem:claim:every_e_belongs_to_mincut:1}
            \item $\dig{u,v}$ is a digon of $D$, and $e$ belongs to a cut of size at most $2k+2$ of $G$; or
            \label{enumitem:claim:every_e_belongs_to_mincut:2}
            \item $\{u,v\}$ is free and for every $a\in \{u,v\}$ there exists a cut $\delta_G(S_a)$ of size $2k+1$ containing $e$ such that $a\in S_a$ and $|S_a|\leq 2k$.
            \label{enumitem:claim:every_e_belongs_to_mincut:3}
        \end{enumerate}
    \end{claim}
    \begin{proofclaim}
        Let $D'$ be obtained from $D$ by deleting all arcs whose set of endvertices is $\{u,v\}$ and let $G'$ be the underlying graph of $D'$. 
        
        If $\lambda(G') \leq 2k-1$, then~\ref{enumitem:claim:every_e_belongs_to_mincut:1} or~\ref{enumitem:claim:every_e_belongs_to_mincut:2} is satisfied. Henceforth, we assume that $D'$ is $2k$-edge-connected. Note that $D'$ is not $(k,3)$-invertible, for otherwise $D$ would be as well. Hence, by minimality of $|A|$, $D'$ must be a $k$-obstruction.
        Therefore, there exists a $k$-certificate $(X_1,\dots,X_r,Y)$ of $D'$.

        The vertices $u$ and $v$ do not belong to the same part of $(X_1,\dots,X_r,Y)$, for otherwise $D$ itself would be a $k$-obstruction. Moreover, there is no edge between $u$ and $v$ in $G'$, so neither $u$ nor $v$ belongs to $Y$. It follows that $u\in X_i$ and $v\in X_j$ for some $i\neq j$. Observe that
        \[
            d_{G}(X_i) = d_{G'}(X_i) + \mu_G(e)  = 2k+ \mu_G(e),
        \]
        where $\mu_G(e)$ is the multiplicity of $e$ in $G$. Now, if $\dig{u,v}$ is a digon of $D$, then~\ref{enumitem:claim:every_e_belongs_to_mincut:2} follows. Henceforth, we may thus assume that $\{u,v\}$ is a simple edge of $G$. In particular, we have
        \[
            d_{G}(X_i) =2k+1.
        \]
        Moreover, since $G'[X_i,Y]$ is complete and $|Y|\geq 1$, we have $|X_i|\leq 2k$. Symmetrically, $X_j$ contains $v$, the cut $\delta_G(X_j)$ has size $2k+1$ and contains $e$, and $|X_j|\leq 2k$. 
        
        It remains to prove that $e$ is free to show that~\ref{enumitem:claim:every_e_belongs_to_mincut:3} is satisfied. Assume for a contradiction that $e$ is not free. If there exists a vertex $w$ that is non-adjacent to both $u$ and $v$, then $e$ is free by Claim~\ref{claim:free_triplets}. Hence, $N_G[u] \cup N_G[v]  = V$. In particular, this implies
        \[
            X_i \cup X_j \cup N_G(X_i) \cup N_G(X_j) = V.
        \]
        Let $y$ be an arbitrary vertex of $Y$. Recall that $X_i\cup X_j$ is complete to $Y$. It follows that
        \begin{align*}
            |N_G(X_i) \cup N_G(X_j)| &= \Big|\big(N_G(X_i)\cup N_G(X_j)\big)\setminus \{y\}\Big| + 1\\
            &\leq \big|N_G(X_i)\setminus \{y\}\big| + \big|N_G(X_j)\setminus \{y\}\big| + 1\\
            &\leq d_G(X_i) + d_G(X_j) - |X_i\cup X_j|+1.
        \end{align*}
        Moreover, there is at least one edge (namely $e$) between $X_i$ and $X_j$. Therefore,
        \begin{align*}
            n&= |X_i \cup X_j \cup N_G(X_i) \cup N_G(X_j)|\\
            &\leq |X_i \cup X_j| +  |N_G(X_i) \cup N_G(X_j)| -2\\
            &\leq d_G(X_i) + d_G(X_j) - 1\\
            &= 4k+1,
        \end{align*}
        a contradiction.
    \end{proofclaim}
    
    Clearly, $G$ contains a non-free edge, for otherwise any $k$-arc-strong orientation $D'$ of $G$ not containing parallel arcs can be obtained from $D$ by reversing every arc of $A(D)\setminus A(D')$.
    Thus, let us fix an edge $uw$ of $G$ that is not free. By Claim~\ref{claim:every_e_belongs_to_mincut}, there exists a cut $\delta_G(U)$ of $G$ of size $2k$ containing $uw$. Let $W=V\setminus U$, and suppose without loss of generality that $u\in U$ and $w\in W$.  We define:
    \begin{itemize}[itemsep=0pt]
        \item $U_0 = U\setminus N_G(w)$,
        \item $U_1 = (U\cap N_G(w))\setminus N_G[u]$, and
        \item $U_2 = U \cap N_G(w)\cap N_G(u)$,
    \end{itemize}
    and similarly:
    \begin{itemize}[itemsep=0pt]
        \item $W_0 = W\setminus N_G(u)$,
        \item $W_1 = (W\cap N_G(u))\setminus N_G[w]$, and
        \item $W_2 = W\cap N_G(u) \cap N_G(w)$. 
    \end{itemize}

    Observe that $u$ dominates $U_0$ and $w$ dominates $W_0$, for otherwise there exists some vertex outside $N[u]\cup N[w]$, and then $\{u,w\}$ would be free by Claim~\ref{claim:free_triplets}, a contradiction.
    
    By symmetry of the roles of $u$ and $w$, we assume without loss of generality that $|W_0|\geq |U_0|$. We finally let $Y=\{w\} \cup W_1$ and $X= V\setminus Y$, see Figure~\ref{fig:U_W_decomposition} for an illustration.

    \begin{figure}[htbp]
    \centering
	\begin{tikzpicture}[region/.style={draw=black,rounded corners, semithick,inner sep=5pt, outer sep=2pt}]
        \draw[gray, dashed] (0,3) -- (0,-2.5);
		\node[vertex, g-blue, label=left:$u$] (u) at (-2.375,0) {};
		\node[vertex, red, label=right:$w$] (w) at (2.375,0) {};
        
        \node[svertex, g-blue] (u1a) at (-3.8,2) {};
        \node[svertex, g-blue] (u1b) at (-3.4,2) {};
        \node[svertex, g-blue] (u1c) at (-3,2) {};
        
        \node[svertex, g-blue] (u2a) at (-1.75,2) {};
        \node[svertex, g-blue] (u2b) at (-1.35,2) {};
        \node[svertex, g-blue] (u2c) at (-0.95,2) {};

        \node[svertex, g-blue] (u0a) at (-3.175,-1.5) {};
        \node[svertex, g-blue] (u0b) at (-2.775,-1.5) {};
        \node[svertex, g-blue] (u0c) at (-2.375,-1.5) {};
        \node[svertex, g-blue] (u0d) at (-1.975,-1.5) {};
        \node[svertex, g-blue] (u0e) at (-1.575,-1.5) {};
        
        \node[svertex, g-blue] (w0a) at (3.175,-1.5) {};
        \node[svertex, g-blue] (w0b) at (2.775,-1.5) {};
        \node[svertex, g-blue] (w0c) at (2.375,-1.5) {};
        \node[svertex, g-blue] (w0d) at (1.975,-1.5) {};
        \node[svertex, g-blue] (w0e) at (1.575,-1.5) {};
       
        \node[svertex, red] (w1a) at (3.8,2) {};
        \node[svertex, red] (w1b) at (3.4,2) {};
        \node[svertex, red] (w1c) at (3,2) {};
        
        \node[svertex, g-blue] (w2a) at (1.75,2) {};
        \node[svertex, g-blue] (w2b) at (1.35,2) {};
        \node[svertex, g-blue] (w2c) at (0.95,2) {};
		
        \begin{scope}[on background layer]
        \node[region,g-blue, fill=g-blue!5, fit=(u0a)(u0e), label=below:{$U_0$}] (U0) {};
        \node[region,g-blue, fill=g-blue!5, fit=(w0a)(w0e), label=below:{$W_0$}] (W0) {};
        \node[region,g-blue, fill=g-blue!5, fit=(u1a)(u1b)(u1c), label=above:{$U_1$}] (U1) {};
        \node[region,g-blue, fill=g-blue!5, fit=(u2a)(u2b)(u2c), label=above:{$U_2$}] (U2) {};
        \node[region,red, fill=red!5, fit=(w1a)(w1b)(w1c), label=above:{$W_1$}] (W1) {};
        \node[region,g-blue, fill=g-blue!5, fit=(w2a)(w2b)(w2c), label=above:{$W_2$}] (W2) {};
        \end{scope}

        \draw[tedge] (u) -- (w);
        
        \draw[tedge, dotted] (u) -- (U1);
        \draw[tedge] (u) -- (U2);
        \draw[tedge] (u) -- (W2);
        \draw[tedge] (u) -- (W1);
        \draw[tedge] (u) -- (U0);
        \draw[tedge, dotted] (u) -- (W0);
        
        \draw[stedge, white] (w) -- (U1);
        \draw[tedge] (w) -- (U1);
        \draw[stedge, white] (w) -- (U2);
        \draw[tedge] (w) -- (U2);
        \draw[stedge, white] (w) -- (W2);
        \draw[tedge] (w) -- (W2);
        \draw[stedge, white] (w) -- (W1);
        \draw[tedge, dotted] (w) -- (W1);
        \draw[stedge, white] (w) -- (U0);
        \draw[tedge, dotted] (w) -- (U0);
        \draw[tedge] (w) -- (W0);
	\end{tikzpicture}
    \caption{The partition of $V$ into $(U_0,U_1,U_2,\{u\},W_0,W_1,W_2,\{w\})$. The dashed vertical line illustrates the partition $(U,W)$. The sets $X$ and $Y$ contain the blue and red vertices, respectively. Between a vertex and a set of vertices, a solid line illustrates that all possible edges are present and a dotted line that none are present.}
    \label{fig:U_W_decomposition}
    \end{figure}
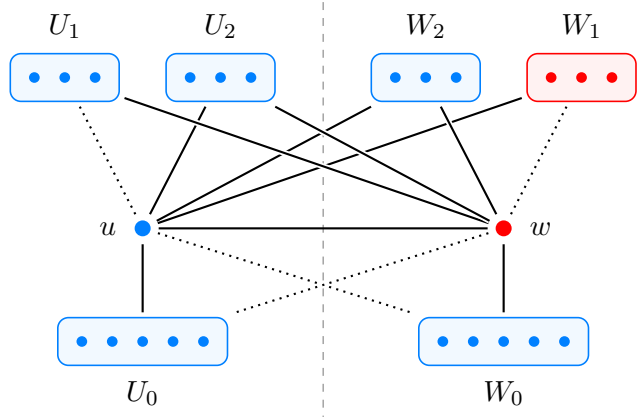

    Our next goal is to show that for any $x\in X$ and $y\in Y$, $E(G)$ contains exactly one edge linking $x$ and $y$. We deduce it from the next sequence of claims.

    \begin{claim}
        \label{claim:freeness_U_1}
        Every pair $\{s,t\} \subseteq U$ is free.
    \end{claim}
    \begin{proofclaim}
        Let $\ell = |U_1| + |U_2| + |W_2|+|W_1|$. Since $w$ dominates $U_1\cup U_2 \cup \{u\}$ and $u$ dominates $W_1\cup W_2 \cup \{w\}$, observe that 
        \[
            d_{G}(U) \geq \ell+1 + d_G(U,W_0).
        \]
        This implies that $W_0$ contains a vertex $w'$ such that $N_G(w') \cap U = \emptyset$. Indeed, if this is not the case, then every vertex in $W_0$ has a neighbor in $U$, and using the fact that $|W_0| \geq |U_0|$ we obtain
        \begin{align*}
            d_G(U) \geq \ell+1 + |W_0|\geq \ell+1 +\frac{1}{2}(n-\ell-2)\geq \frac{1}{2}n > 2k,
        \end{align*}
        a contradiction.
        Therefore, for every edge $\{s,t\} \subseteq U$, the inversion $\{s,t,w'\}$ simulates the inversion of $\{s,t\}$, showing the claim. 
    \end{proofclaim}
    
    \begin{claim}
        \label{claim:freeness_U_2}
        The set $U_0$ is empty.
    \end{claim}
    \begin{proofclaim}
        Assume that $U_0 \neq \emptyset$ and let $u'\in U_0$. By definition, $u'$ is not adjacent to $w$, and by Claim~\ref{claim:freeness_U_1} $\{u,u'\}$ is free. Therefore, $\{u,w\}$ is free by Claim~\ref{claim:free_triplets}, a contradiction.  
    \end{proofclaim}

    By Claim~\ref{claim:freeness_U_2}, we have $X = \{u\} \cup U_1\cup U_2\cup W_0 \cup W_2$. Moreover, it follows from Claim~\ref{claim:freeness_U_1} that $w$ is adajcent to all vertices in $X$.
    We now show through two claims that, furthermore, $G[X,Y]$ is complete, and that the edges between $X$ and $Y$ are non-free. 

    \begin{claim}
        \label{claim:X_w_non_free_1}
        For every vertex $s\in W_0 \cup W_2$, the pair $\{u,s\}$ is free. 
    \end{claim}
    \begin{proofclaim}
        The statement is clear when $s\in W_0$, as by definition $u$ has no neighbor in $W_0$. Let us thus fix an arbitrary vertex $s$ in $W_2$. We claim that $s$ has a non-neighbor $w'$ in $W_0$. To see this, recall that 
        \[
            |U_1| + |U_2| + |W_2| +|W_1|\leq d_G(U) -1 \leq 2k-1.
        \]
        Since $U_0$ is empty by \Cref{claim:freeness_U_2} and $n \geq 4k+2$, we thus have
        \[
            |W_0| = n - (|U_1| + |U_2| + |W_2| +|W_1|) - 2 \geq 2k+1.
        \]
        Hence, if $s$ dominates $W_0$ then there exist, in $G$, at least $2k+1+\mu_G(sw)$ pairwise edge-disjoint paths going from $s$ to $w$, namely all $2$-paths $(s,w',w)$ for $w'\in W_0$ and the $\mu_G(sw)$ $1$-paths $(s,w)$.
        Hence, every cut of $G$ separating $s$ and $w$ has size at least $2k+1+\mu_G(sw)$, a contradiction to Claim~\ref{claim:every_e_belongs_to_mincut}.
        Let us thus fix a vertex $w'\in W_0 \setminus N_G(s)$. Since $W_0 \cap N_G(u) = \emptyset$, it follows that inverting $\{u,s,w'\}$ simulates the inversion of $\{u,s\}$. 
    \end{proofclaim}

    \begin{claim}
        \label{claim:X_w_non_free_2}
        For every $x\in X$, and $y\in Y$, the pair $\{x,y\}$ is not free.
    \end{claim}
    \begin{proofclaim}
        First assume that $y=w$. If $x=u$, then $\{u,w\}$ is not free by assumption. If $x\in X\setminus \{u\}$, by either Claim~\ref{claim:freeness_U_1} or Claim~\ref{claim:X_w_non_free_1}, we get that $\{u,x\}$ is free. Therefore, since $\{u,w\}$ is non-free, by Claim~\ref{claim:free_triplets} it follows that $\{w,x\}$ is non-free.

        Now assume that $y\neq w$, that is $y\in W_1$. We already proved that $\{x,w\}$ is not free and, by definition of $W_1$, $y$ is not adjacent to $w$, so $\{w,y\}$ is free. Therefore, by Claim~\ref{claim:free_triplets}, $\{x,y\}$ cannot be free. The claim follows.
    \end{proofclaim}
    
    It follows from Claim~\ref{claim:X_w_non_free_2} that $G[X,Y]$ is complete, and that it contains only simple edges, hence showing that the partition $(X,Y)$ satisfies Property~\ref{enumitem:kobstruction:ii} of $k$-obstructions.
    We now show that $X$ can be partitioned in such a way that Property~\ref{enumitem:kobstruction:i} is satisfied as well.
    \begin{claim}
        \label{claim:partition_X}
        There exists a $2k$-regular partition of $X$ in $G$. 
    \end{claim}
    \begin{proofclaim}
        By \Cref{eedqe}, it suffices to prove that, for every $x\in X$, there exists a set $S_x \subseteq X$ containing $x$ such that $d_G(S_x) = 2k$. Let $y$ be an arbitrary vertex in $Y$. By Claim~\ref{claim:X_w_non_free_2}, $\{x,y\}$ is not free, which together with Claim~\ref{claim:every_e_belongs_to_mincut} implies that $\{x,y\}$ belongs to a cut of size $2k$, that is, there exists a set $S_x\subseteq V(D)$ such that $x\in S_x$, $y\notin S_x$, and $d_G(S_x)= 2k$. It remains to justify that $S_x\subseteq X$. Assume for a contradiction that $Y\cap S_x \neq \emptyset$. Since $Y$ is complete to $X$, we have
        \begin{align*}
            d_G(S_x) &\geq \underbrace{|Y\setminus S_x|}_{\geq 1}\cdot  \,|X\cap S_x| + \underbrace{|Y\cap S_x|}_{\geq 1} \cdot \,|X\setminus S_x|\\ 
            &\geq |X|\\
            &= n-|Y|.
        \end{align*}
        Recall that $Y = W_1\cup \{w\}$ is complete to $\{u\}$. Therefore, 
        \[
            |Y|\leq d_G(U) \leq 2k,
        \]
        which together with the inequality above implies
        \[
            d_G(S_x) \geq n-2k \geq 2k+1,
        \]
        a contradiction.
    \end{proofclaim}

    From now on, let us fix a $2k$-regular partition $(X_1,\dots,X_r)$ of $X$ in $G$.
    Recall that, by definition of being a counterexample, $D$ is not a $k$-obstruction. Since the partition $(X_1,\dots,X_r,Y)$ of $V$ satisfies Properties~\ref{enumitem:kobstruction:i} and~\ref{enumitem:kobstruction:ii}, either $|X|\cdot |Y|$ is odd, or $d_D^+(X) - \frac{1}{2}|X|\cdot |Y|$ is an even integer. Using the degree condition of Property~\ref{enumitem:kobstruction:i}, we exclude the former case.

    \begin{claim}
        \label{claim:XY_even}
        We have that $|X|\cdot |Y|$ is even, and that $d_D^+(X) - \frac{1}{2}|X|\cdot |Y|$ is an even integer.
    \end{claim}
    \begin{proofclaim}
        Assume that $|Y|$ is odd, and let us show that $|X|$ is even. Let $G'$ be the graph with vertex set $X$ and containing the edges of $G$ whose extremities belong to different classes of the partition $(X_1,\dots,X_r)$. 
        Note that, for every $i\in [r]$, 
        \[
        d_{G'}(X_i) = d_{G}(X_i)-|X_i|\cdot |Y| = 2k-|X_i|\cdot |Y|,
        \]
        which is even if and only if $|X_i|$ is even. Therefore, we have:
        \[
            2\cdot |E(G')| = \sum_{x\in X} d_{G'}(x)=\sum_{i=1}^rd_{G'}(X_i)
            \equiv  \sum_{i=1}^r|X_i| \pmod 2.
        \]
        Since $|X| = \sum_{i=1}^r |X_i|$, we obtain that $|X|$ is even.
    \end{proofclaim}
    
    In what remains, let $D'$ be an arbitrary $k$-arc-strong digraph whose underlying graph is $G$. Among all digraphs reachable from $D$ by a set of \eqp{3}-inversions, let $\tilde{D}$ be one closest to $D'$ in the following sense:
    \begin{enumerate}[label=(\arabic*)]
        \item the number $m_1$ of edges of $G[X,Y]$ oriented differently in $D'$ and $\tilde{D}$ is minimum; and
        \label{enumitem:choice_Dtilde:1}
        \item with respect to~\ref{enumitem:choice_Dtilde:1}, the number $m_2$ of edges of $G[X]$ oriented differently in $D'$ and $\tilde{D}$ is minimum.
        \label{enumitem:choice_Dtilde:2}
    \end{enumerate}
    
    Assume first that $m_1 \geq 1$. By \Cref{g7zuhoi}, $D'$ being $k$-arc-strong implies that $d_{D'}^+(X) = \frac{1}{2}|X|\cdot |Y|$. By Claim~\ref{claim:XY_even}, $d_{D}^+(X) - \frac{1}{2}|X|\cdot |Y|$ is even, and so is $d_{\tilde{D}}^+(X) - \frac{1}{2}|X|\cdot |Y|$ as any \eqp{3}-inversion reverses an even number of edges between $X$ and~$Y$.
    Therefore, $d_{D'}^+(X)$ and $d_{\tilde{D}}^+(X)$ have the same parity, and in particular $m_1\geq 2$. Let $E_1$ be the set of edges of $G[X,Y]$ whose orientations differ in $D'$ and $\tilde{D}$. 
    If there exist two distinct incident edges $\{a,b\},\{b,c\}\in E_1$, then inverting $\{a,b,c\}$ on $\tilde{D}$ yields a digraph reachable from $D$ by a set of \eqp{3}-inversions which contradicts~\ref{enumitem:choice_Dtilde:1}.
    Therefore, the edges of $E_1$ are pairwise disjoint. Let $\{a,b\},\{c,d\}$ be any two distinct edges of them, then inverting successively $\{a,b,c\}$ and $\{b,c,d\}$ leads again to a contradiction to~\ref{enumitem:choice_Dtilde:1}.

    Assume now that $m_2\geq 1$, and let $xz$ be an edge of $G[X]$ oriented differently in $D'$ and $\tilde{D}$, and let $i,j$ be such that $x\in X_i$, $z\in X_j$. Clearly, by~\ref{enumitem:choice_Dtilde:2}, $xz$ is not free. This implies that $N_G[x]\cup N_G[z] = V$, which in particular implies
    \[
        X_i \cup X_j \cup N_G(X_i) \cup N_G(X_j) = V.
    \]
    Since $Y$ dominates $X_i \cup X_j$, we have
    \begin{align*}
            n&=|X_i \cup X_j \cup N_G(X_i) \cup N_G(X_j)|\\
            &\leq |X_i \cup X_j| + |N_G(X_i) \cup N_G(X_j)|\\
            &\leq |X_i \cup X_j| + d_G(X_i) + d_G(X_j) - (|X_i \cup X_j|-1)\\
            &= d_G(X_i) + d_G(X_j)+1\\
            &= 4k+1,
    \end{align*}
    a contradiction.
    Therefore, we have $m_1=m_2=0$. Since $\tilde{D}\neq D'$ (as $D$ is not $(k,3)$-invertible), there exists in $G$ an edge between two vertices $y_1,y_2\in Y$. Since $|X|\geq n-2k \geq 2k+2$, there exist $2k+3$ pairwise edge-disjoint paths from $y_1$ to $y_2$, namely all paths the form $(y_1,x,y_2)$ for $x\in X$ and the $1$-path $(y_1,y_2)$. This contradicts Claim~\ref{claim:every_e_belongs_to_mincut} and concludes the proof of the lemma.
\end{proof}

\subsection{Simulating \texorpdfstring{\eqp{3}}{(=3)}-inversions with \texorpdfstring{\eqp{p}}{(=p)}-inversions}
\label{subsec:simulating_3_inversions}

In this section, we show that, for an odd integer $p$, \eqp{3}-inversions can often be simulated by a set of \eqp{p}-inversions. This allows us to derive \Cref{lem:invertibility} from Lemma~\ref{lemma:3_invertibility} in the next section.
\begin{lemma}
    \label{lemma:simulating_3_inversions}
    Let $p\geq 3$ be an odd integer and $D$ be a digraph of order $n\geq p+2$. If $\alpha(D) \geq 3$, then for every triplet $S=\{x,y,z\} \subseteq V(D)$, the inversion of $S$ in $D$ can be simulated by a set of \eqp{p}-inversions.
\end{lemma}
\begin{proof}
    We assume that $p\geq 5$, the statement being trivial for $p=3$.
    Let $G$ denote the underlying graph of $D$.
    We distinguish two cases, depending on the value of $p \bmod 4$.
    
    \medskip
    
    \noindent {\bf Case 1:} $p \equiv 3\pmod 4$.

    \medskip

    Let $X$ be an arbitrary set of $p-1$ vertices that is disjoint from $S$, which exists as $n\geq p+2$. Let $X_1,\dots,X_{\ell}$ be all subsets of $X$ of size precisely $p-3$, and for each $i\in [\ell]$, let $Y_i=X_i\cup S$. By construction, each $Y_i$ has size exactly $p$. Let us briefly show that inverting the $Y_i$'s simulates the inversion of $S$ by considering all possible cases for an edge $e\in E(G)$:
    \begin{itemize}[itemsep=0pt]
        \item if $e$ is incident to a vertex of $V(D) \setminus (S\cup X)$, then $e$ is not reversed;
        \item if $e$ is incident to a vertex in $X$ and a vertex in $S$, then $e$ is reversed
        $\binom{p-2}{p-4} = \frac{1}{2}(p-2)(p-3)$
        times, which is an even number as $p-3\equiv 0 \pmod 4$;
        \item if $e$ belongs to $G[X]$, then it is reversed
        $\binom{p-3}{p-5} = \frac{1}{2}(p-3)(p-4)$ 
        times, which is an even number again; and
        \item if $e$ belongs to $G[S]$ then it is reversed $\binom{p-1}{p-3} = \frac{1}{2}(p-1)(p-2)$ times,
        which is odd as $p-2$ is odd and $p-1\equiv 2 \pmod 4$.
    \end{itemize}
    
    \noindent {\bf Case 2:} $p \equiv 1\pmod 4$.

    \medskip
    
    In this case, using an argument similar to the previous case, we obtain that \eqp{5}-inversions can be simulated by a set of \eqp{p}-inversions. 
    \begin{claim}
        \label{claim:simulating_5_inversions}
        For any set $R \subseteq V(D)$ of size $5$, 
        there exists a set of \eqp{p}-inversions simulating the inversion of $R$.
    \end{claim}
    \begin{proofclaim}
        Let $X$ be an arbitrary set of $p-3$ vertices disjoint from $R$, which exists as $n\geq p+2$. Let $X_1,\dots,X_{\ell}$ be all subsets of $Z$ of size precisely $p-5$, and for each $i\in [\ell]$, let $Y_i=X_i\cup R$. An analysis identical to the one from Case~1 shows that inverting $\Ycal = \{Y_1,\dots,Y_{\ell}\}$ simulates the inversion of $R$.
    \end{proofclaim}

    Using Claim~\ref{claim:simulating_5_inversions}, we obtain the following.
    
    \begin{claim}
        \label{claim:simulating_disjoint_3_inversions}
        For any two disjoint sets of triplets $R,R' \subseteq V(D)$, 
        there exists a set of \eqp{p}-inversions reversing the arcs of $D[R]$ and $D[R']$ and no other arc.
    \end{claim}
    \begin{proofclaim}
        Let us denote $R=\{x,y,z\}$. For every $u\in R$, let $\mathcal{X}_u$ be a set of \eqp{p}-inversions simulating the inversion of $(R\setminus u) \cup R'$, which exists by Claim~\ref{claim:simulating_5_inversions}. It is straightforward to check that
        $\mathcal{X}_x \XOR \mathcal{X}_y \XOR  \mathcal{X}_z$
        satisfies the desired statement.
    \end{proofclaim}
    
    We now make use of the hypothesis $\alpha(D)\geq 3$. Let us fix an independent set $I$ of size~$3$. Observe that the statement of the lemma is trivial when $I=S$, and that it follows from Claim~\ref{claim:simulating_disjoint_3_inversions} for any set $S$ disjoint from $I$. It thus remains to consider the sets $S$ intersecting $I$ on one or two vertices.

    Consider first any set $S$ with $|S\cap I| = 2$. Let $S'$ be any set of size $3$ disjoint from $S\cup I$, which exists as $n\geq p+2\geq 7$, while $|S\cup I|= 4$. By Claim~\ref{claim:simulating_disjoint_3_inversions}, applied with $R=I$ and $R'=S'$, there exists a set $\mathcal{X}$ of \eqp{p}-inversions simulating the inversion of $S'$. Applying the claim a second time to $R=S$ and $R'=S'$ yields a second set $\mathcal{X}'$ of \eqp{p}-inversions reversing the arcs of $D[S']$ and $D[S]$. Inversing  $\mathcal{X} \XOR \mathcal{X'}$ thus simulates the inversion of $S$.

    Now assume that $S$ is such that $|S\cap I|=1$ and let $S'$ be any set of size $3$ disjoint from $S$ and containing two vertices of $I$, which exists as $n\geq p+2\geq 6$. Since $|S'\cap I|=2$, we already justified that there exists a set $\mathcal{X}$ of \eqp{p}-inversions simulating the inversion of $S'$. Again, by Claim~\ref{claim:simulating_disjoint_3_inversions} applied with $R=S$ and $R'=S'$, there exists a set $\mathcal{X}'$ of \eqp{p}-inversions reversing the arcs of $D[S']$ and $D[S]$. Inverting $\mathcal{X} \XOR \mathcal{X'}$ thus simulates the inversion of $S$. The lemma follows.
\end{proof}

\subsection{Deriving \texorpdfstring{\Cref{lem:invertibility}}{Lemma 25}}
\label{subsec:invertibility}

With Lemmas~\ref{lemma:3_invertibility} and~\ref{lemma:simulating_3_inversions} in hand, we now derive \Cref{lem:invertibility}, which we first recall here for convenience.

\leminvertibility*

\begin{proof}
    Assume for a contradiction that, for some integer $k\geq 1$ and some odd integer $p\geq 3$, there exists a $2k$-edge-connected digraph $D$ of order $n\geq \max \{4k+2,p+2\}$ that is neither $(k,p)$-invertible nor a $k$-obstruction. Among all such digraphs, let $D=(V,A)$ be one for which $|A|$ is minimum and let $G$ be its underlying graph.

    Combining Lemmas~\ref{lemma:3_invertibility} and~\ref{lemma:simulating_3_inversions}, we have $p\geq 5$ and $\alpha(D) \leq 2$. We next show that the deletion of any arc destroys the $2k$-edge-connectivity of $D$.
    \begin{claim}\label{duqwhuq}
        For every $a \in A(D)$, we have that $D-a$ is not $2k$-edge-connected. 
    \end{claim}
    \begin{proofclaim}
    Suppose for the sake of a contradiction that there exists some $a \in A$ such that $D-a$ is $2k$-edge-connected.
    By choice of $D$,we obtain that $D-a$ is a $k$-obstruction.
    Let $(X_1,\dots,X_r,Y)$ be a $k$-certificate of $D-a$. Observe then that for every $i \in [r]$, we have $d_G(X_i)=2k$, unless $X_i$ contains exactly one endvertex of $a$, in which case $d_G(X_i)=2k+1$. Since $Y$ is complete to $X=\bigcup_{i\in [r]}X_i$, we get that $|Y| \leq 2k+1$, for otherwise $d_G(X_1) \geq 2k+2$. Moreover, we have $|X_i|\leq d_G(X_i) \leq 2k+1$ for every $i\in [r]$. 
    We now distinguish two cases.
    \begin{description}
        \item[Case 1:] {\it The endvertices of $a$ belong to the same part of $(X_1,\dots,X_r,Y)$, or at least one of them belongs to $Y$.}

        Let $x_1\in X_1$ be an arbitrary vertex. There exist a vertex $x_2\in X\setminus X_1$ non-adjacent to $x_1$ in $G$, for otherwise $V = X_1 \cup N_G(X_1)$
        which, using the fact that $Y$ is complete to $X_1$, implies that
        \[
            n = |X_1|+ |N_G(X_1)| \leq |X_1| +d_G(X_1) - (|X_1|-1) = d_G(X_1)+1 \leq 2k+2,
        \]
        a contradiction. Assume without loss of generality that $x_2\in X_2$. Since $x_1$ and $x_2$ are non-adjacent and $\alpha(D) \leq 2$, we have $V= N[x_1] \cup N[x_2]$, and in particular
        \[
            V = X_1 \cup N_G(X_1) \cup X_2 \cup N_G(X_2).
        \]
        If $d_G(X_1) = d_G(X_2)= 2k$ then 
        \[
            n\leq |X_1\cup X_2| + d_G(X_1) + d_G(X_2) - (|X_1\cup X_2|-1) = d_G(X_1) + d_G(X_2)+1 \leq 4k+1,
        \]
        a contradiction. Hence, as we are in Case~1, $a$ is incident to a vertex in $X_1\cup X_2$ and a vertex $y\in Y$. Note that, in this case, $a$ belongs to a digon, as after removing $a$ we get that $X$ is still complete to $Y$. Therefore, some vertex $y\in Y$ is incident to at least $|X_1\cup X_2|+1$ edges with an extremity in $X_1\cup X_2$. We thus have
        \begin{align*}
            n& \leq |X_1 \cup X_2| + |N_G(X_1\cup X_2)\setminus \{y\}| +1 \\
            &\leq |X_1 \cup X_2| + d_G(X_1\cup X_2) -(|X_1\cup X_2|+1) + 1\\
            &\leq d_G(X_1) + d_G(X_2) \\
            &= 4k+1,
        \end{align*}
        a contradiction. This concludes Case~1.

        \item[Case 2:] {\it The extremities of $a$ belong to $X_i$ and $X_j$ for some $i,j \in [r]$ with $i\neq j$.}

        Assume without loss of generality that $i=1$ and $j=2$. 
        Since $a$ is incident to a vertex of $X_1$ and a vertex of $X_2$, note that 
        \begin{align*}
            |X_1 \cup X_2 \cup N_G(X_1) \cup N_G(X_2)|\leq |X_1 \cup X_2|+ |N_G(X_1) \cup N_G(X_2)| - 2.
        \end{align*}
        Let $y$ be an arbitrary vertex of $Y$. 
        If there exist two vertices $x_1\in X_1$ and $x_2\in X_2$ that are non-adjacent, then, as $\alpha(D)\leq 2$, we have $V=N_G[x_1] \cup N_G[x_2]$, and in particular
        \[
            V = X_1 \cup N_G(X_1) \cup X_2 \cup N_G(X_2).
        \]
        Together with the previous inequality, it follows that
        \begin{align*}
            n&\leq |X_1 \cup X_2|+ |N_G(X_1) \cup N_G(X_2)| - 2\\
            &\leq |X_1 \cup X_2|+ |(N_G(X_1) \cup N_G(X_2))\setminus \{y\}| - 1\\
            &\leq  |X_1|+|X_2|+ |N_G(X_1) \setminus \{y\}|+ |N_G(X_2) \setminus \{y\}| - 1\\
            &\leq d_G(X_1) + d_G(X_2) - 1\\
            &= 4k+1,
        \end{align*}
        a contradiction.
        Hence, $G[X_1,X_2]$ is complete. 
        
        We claim that $|X_1|=|X_2|=1$. Assume for a contradiction, and by symmetry, that $|X_1|\geq 2$, and let $x_1\in X_1$ be an arbitrary vertex of $X_1$. There exists a vertex $x_3\in X\setminus X_1$ that is not adjacent to $x_1$, for otherwise $V=X_1\cup N_G(X_1)$, which again implies
        \[
            n \leq d_G(X_1) +1= 2k+2,
        \]
        a contradiction. Since $G[X_1,X_2]$ is complete, we assume without loss of generality that $x_3\in X_3$. Again, we have $V = X_1 \cup N_G(X_1) \cup X_3 \cup N_G(X_3)$. Since the vertices in $X_1$ are all adjacent to the vertices in $X_2$ and $|X_1|\geq 2$, we have
        \begin{align*}
            n &\leq |X_1 \cup X_3|+ |N_G(X_1) \cup N_G(X_3)| \\
            &\leq |X_1 \cup X_3|+ d_G(X_1) + d_G(X_3) -(|X_1\cup X_3| - 1) - (|X_1|-1)\\
            &\leq  d_G(X_1) + d_G(X_3)\\
            &= 4k+1,
        \end{align*}
        a contradiction. Therefore $|X_1|=|X_2|=1$. Fix an arbitrary vertex $x\in X_3$, which exists as $|Y\cup X_1\cup X_2|\leq 2k+3 \leq n-1$. Let $x_4$ be a vertex in $X\setminus (X_1 \cup X_2 \cup X_3)$ that is non adjacent to $x_3$, which exists for otherwise $V=X_1 \cup X_2 \cup X_3 \cup N_G(X_3)$ and
        \[
            n \leq 2 + |X_3|+|N_G(X_3)| \leq 2+ d_G(X_3) +1 = 2k+3,
        \]
        a contradiction. We assume without loss of generality that $x_4\in X_4$. Since $a$ is not incident to any vertex in $X_3\cup X_4$, as in the previous case we have
        \begin{align*}
            n&\leq |X_3\cup X_4| + d_G(X_3) + d_G(X_4) - (|X_3\cup X_4|-1)\\
            &\leq 4k+1,
        \end{align*}
        a contradiction.\qedhere
    \end{description}
 \end{proofclaim}
   By \Cref{duqwhuq}, every edge of $G$ belongs to a cut of size $2k$. By Theorem~\ref{thm:mader}, $G$ contains a vertex $u$ of degree at most $2k$, and it has degree exactly $2k$ as $G$ is $2k$-edge-connected. Let $C=V\setminus N_G[u]$. Since $\alpha(G)\leq 2$, $G[C]$ contains a complete graph of order at least
    \[
        n-(d_G(u)+1) \geq 2k+1.
    \]
    as a spanning subgraph.
    Assume first that two distinct vertices $x,y\in C$ have a neighbor in $N_G(u)$. Then there exist $2k+1$ pairwise edge-disjoint path from $x$ to $y$, namely the $1$-path $(x,y)$, the $2$-paths of the form $(x,z,y)$ for $z\in C\setminus \{x,y\}$, and a path of length at most $4$ whose internal vertices belong to $N_G[u]$. This is a contradiction, as every edge belongs to a cut of size $2k$. A similar argument shows that no two vertices in $C$ are linked by two parallel edges, so $G[C]$ is a complete graph.

    Hence, as $G$ is connected, there is exactly one vertex $x$ in $C$ with neighbors in $N_G(u)$. Let $y\in C\setminus x$. Since $d_G(y)=2k$, by symmetry of the roles of $y$ and $u$ we deduce that $N_G[u]$ is a clique $C'$ of size $2k+1$ of $G$. By symmetry of the roles of $C'$ and $C$, only one vertex $x'$ of $C'$ has neighbors in $C$.

    Observe that removing the edges between $x$ and $x'$ disconnects $G$.
    Since $G$ is $2k$-edge-connected, we deduce that $k=1$ and $\dig{x,x'}$ is a digon of $D$. Hence $n=6$, a contradiction as $n\geq p+2\geq 7$. The result follows.
\end{proof}

\subsection{Identifying \texorpdfstring{$k$}{k}-obstructions}
\label{sec:identifying_obstructions}

The objective of this section is to derive Corollary~\ref{cor:kernel_k+p}. With \Cref{lem:invertibility} on hand, the last ingredient we need is a polynomial-time algorithm to decide whether a given $2k$-edge-connected digraph is a $k$-obstruction. More precisely, the technical difficulty is to prove the following result.

\begin{lemma}\label{mainident}
    Given a $2k$-edge-connected digraph $D$ of order $n$ for some positive integers $k$ and $n$ with $n \geq 4k+2$, one can decide in polynomial time whether $D$ is a $k$-obstruction. 
\end{lemma}

In order to prove Lemma~\ref{mainident}, we first need one more extra definition.
Namely, if for some digraph $D$ and $Y \subseteq V(D)$, there exists a partition $(X_1,\ldots,X_r)$ of $V(D)\setminus Y$ such that $(X_1,\ldots,X_r,Y)$ is a $k$-certificate for $D$, we say that $Y$ can be {\it extended} to a $k$-certificate for~$D$.

The strategy to prove Lemma~\ref{mainident} consists of proving that the problem can be decided in polynomial time if $Y$ is already known and then showing that only a small number of candidates for $Y$ needs to be considered. For the first part, we need to show that we can decide in polynomial time whether a given set of vertices admits a $2k$-regular partition. 
Using \Cref{eedqe}, we show that this problem can be reduced to a small number of minimum cut computations.

\begin{proposition}\label{ctufhij}
    Given a positive integer $k$, a $k$-edge-connected graph $G$ and a set $X \subseteq V(G)$, one can decide in polynomial time whether $X$ admits a $k$-regular partition.
\end{proposition}
\begin{proof}
    By \Cref{eedqe}, it suffices to check whether for every $x \in X$, there exists a set $S_x\subseteq X$ such that $x \in S_x$ and $d_G(S_x)=k$.

    Fix some $x \in X$ and let $G'$ be the graph obtained from $G$ by contracting $V(G)\setminus X$ into a single vertex $y$. As $G$ is $k$-edge-connected and by construction, it now suffices to check whether $\lambda_{G'}(x,y)\leq k$. This can be done in polynomial time using a standard minimum cut algorithm, see for example \cite{schrijver2002}.

We do this for every $x \in X$. As the algorithm runs in polynomial time for every $x \in X$ and $|X|\leq |V(G)|$, the entire algorithm runs in polynomial time.
\end{proof}

We now conclude that we can decide efficiently whether a given set can be extended to a $k$-certificate.

\begin{lemma}\label{rtfugihjkm}
    Given an integer $k$, a $2k$-edge-connected digraph $D$ with $n \geq 4k+2$ and a set $Y \subseteq V(D)$, one can decide in polynomial time whether $Y$ can be extended to a $k$-certificate.
\end{lemma}
\begin{proof}
    Let $X=V(D)\setminus Y$ and $G$ be the underlying graph of $D$. We first check whether $G$ contains exactly one edge linking $x$ and $y$ for all $x \in X$ and $y \in Y$ and whether $d^+_D(X) - \frac{1}{2}|X|\cdot |Y|$ is an odd integer. Clearly, $G$ and $X$ can be computed in polynomial time and both checks can be executed in polynomial time. If  any of these checks fail, we correctly report that $Y$ cannot be extended to a $k$-certificate. Otherwise, by the definition of $k$-certificates, it suffices to check whether $X$ admits a $2k$-regular partition. By Proposition~\ref{ctufhij}, this can also be checked in polynomial time. Hence the entire algorithm runs in polynomial time.
\end{proof}
    
The next result restricts the set of candidates for the set $Y$ that need to be considered. It shows that $Y$ either consists of a single vertex or is an explicitly specified set.
\begin{proposition}\label{zvuzbnml}
    Let $k$ be a positive integer, $D$ a $2k$-edge-connected digraph of order $n$ with $n \geq 4k+2$, let $G$ be the underlying graph of $D$ and $(X_1,\ldots,X_r,Y)$ a $k$-certificate for $D$ with $|Y|\geq 2$. Then $Y=\{v\in V(G):d_G(v)\geq 2k+1\}$.
\end{proposition}
\begin{proof}
    First suppose that $|Y|\geq 2k+1$. Then, as  $(X_1,\ldots,X_r,Y)$ is a $k$-certificate, we have $d_G(X_1)\geq |X_1|\cdot |Y|\geq |Y|\geq 2k+1$, a contradiction. It follows that $|Y|\leq 2k$.

   Let $X=\bigcup_{i \in [r]}X_i$ and consider some $y \in Y$. As $G$ contains exactly one edge linking $y$ and $x$ for all $x \in X$ and by $|Y|\leq 2k$ and $n \geq 4k+2$, we have $d_G(y)\geq |X|=n-|Y|\geq 2k+1$.

   Now consider some $x \in X$. By symmetry, we may suppose that $x \in X_1$. As $|Y|\geq 2$ and $(X_1,\ldots,X_r,Y)$ is a $k$-certificate, this yields 
   \begin{align*}
       2k=d_G(X_1)&= d_G(x)-d_{G[X_1]}(x)+d_{G-x}(X_1 \setminus \{x\})\\
       &\geq d_G(x)-2\cdot|X_1 \setminus \{x\}|+|Y|\cdot|X_1 \setminus \{x\}|\\
       &\geq d_G(x).
   \end{align*}
   Hence, the statement follows.
\end{proof}

We are now ready to give the main proof Lemma~\ref{mainident}.
\begin{proof}[Proof of Lemma~\ref{mainident}]
    Let a $2k$-edge-connected digraph $D$ be given and let $G$ be its underlying graph. 
    We check for every $v \in V(D)$ whether $\{v\}$ can be extended to a $k$-certificate in $D$. Moreover, we check whether $\{v \in V(D):d_G(v)\geq 2k+1\}$ can be extended to a $k$-certificate in $D$. If one of these sets can be extended to a $k$-certificate, we report that $D$ is a $k$-obstruction. Otherwise, we report that $D$ is not a $k$-obstruction. It follows from Proposition~\ref{zvuzbnml} that the output of the algorithm is correct. Moreover, by Lemma~\ref{rtfugihjkm}, all of these checks can  be executed in polynomial time. As we only execute a polynomial number of checks, it follows that the entire algorithm runs in polynomial time.
\end{proof}

We are finally able to derive Corollary~\ref{cor:kernel_k+p}, that we first recall here for convenience.

\corkernelkplusp*

\begin{proof}
    Let a digraph $D$, an integer $p\ge 2$, and a positive integer $k$ be given. If $n < \max\{p+2,4k+2\}$, then $D$ itself is a kernel of the desired size. We may hence suppose that $n \geq  \max\{p+2,4k+2\}$. By \Cref{lem:invertibility}, we obtain that $D$ is a yes-instance of {\sc Invertibility} if and only if $D$ is $2k$-edge-connected and $D$ is not a $k$-obstruction. By Theorem~\ref{thm:edge_connectivity_poly} and Lemma~\ref{mainident}, we can decide in polynomial time whether this is the case. If this is the case, we output a rotative tournament on $2k+1$ vertices and otherwise, we output an edgeless graph on two vertices. It is not difficult to see that this is indeed a kernel of appropriate size.
\end{proof}

\section{Inapproximability}
\label{sec:inapproximability}

This section is dedicated to proving that for any fixed $p\geq 3$,
the problem consisting of computing $\inv{k}{p}$ does not admit a PTAS. More formally, we prove Theorem~\ref{thm:inapproximability}, that we recall here for convenience.

\thminapproximability*

The proof of Theorem~\ref{thm:inapproximability} is separated into two parts. First in Section~\ref{sec:inapproximability:p=3}, we prove the statement for $p=3$ by a reduction from the problem of packing copies of $P_3$ in a bipartite graph, see Theorem~\ref{inappr1}. Next, in Section~\ref{sec:inapproximability:pgeq4}, we prove the statement for $p\geq 4$ by a reduction from a restricted version of the hypergraph matching problem, see Theorem~\ref{inappr2}. Theorems~\ref{inappr1} and~\ref{inappr2} together imply Theorem~\ref{thm:inapproximability}.

\subsection{The case \texorpdfstring{$p=3$}{p=3}}
\label{sec:inapproximability:p=3}
 
This section is dedicated to proving Theorem~\ref{thm:inapproximability} for the case $p=3$. We first introduce the problem we reduce from.
Given a graph $G$, a {\it $P_3$-packing} in $G$ is a collection of pairwise vertex-disjoint subgraphs of $G$ each of which is isomorphic to $P_3$. 
We consider the following problem.

\defproblem{$P_3$-Packing (P3P)}{A graph $G$ and an integer $t$.}{Does $G$ admit a $P_3$-packing of size $t$?}

An instance $(G,t)$  of P3P is called {\it bipartite} if $G$ is bipartite. Our reduction is based on the following inapproximability result due to Monnot and Toulouse~\cite{MONNOT2007677}.
\begin{lemma}[\cite{MONNOT2007677}]
    \label{etrztuzibjn}
    There exist $\epsilon'>0$ and $c>0$ such that, unless $P=NP$, there is no polynomial-time algorithm that, given a bipartite instance $(G,t)$ of P3P with $t \geq c|V(G)|$, outputs `yes' if $G$ admits a $P_3$-packing of size $t$ and outputs `no' if $G$ does not admit a $P_3$-packing of size $(1-\epsilon')t$.
\end{lemma}
The main difficulty of our reduction is contained in the following result which shows how to compute a corresponding instance of the inversion problem when an appropriate instance of P3P is given.

\begin{lemma}\label{tf86zguoi}
    Let $k$ be a fixed positive integer. In polynomial time, given a bipartite graph $G$ one can compute an oriented graph $D$ whose size is polynomial in the size of $G$ such that for $n=|V(G)|$, $x=\inv{k}{3}(D)$ and $y$ being the maximum size of a $P_3$-packing in $G$, we have $x=\lceil\frac{n-y}{2}\rceil$.
\end{lemma}
\begin{proof}
    Let a bipartite graph $G$ be given. We first compute a bipartition $(A,B)$ of  $G$. We now construct an oriented graph $D$ as follows. First, we let $V(D)$ contain $V(G)$. Next, for every $v \in V(G)$, we let $V(D)$ contain a set $\{w_v^1,\ldots,w_v^k,z_v^1,\ldots,z_v^{k-1}\}$ of $2k-1$ new vertices. Further, for every $Q\in \binom{V(G)}{2}$, we let $V(D)$ contain a vertex $w_Q$. We let $W$ be the set that consists of $w_v^{i}$ for all $v \in V(G)$ and $i\in [k], z_{v}^{i}$ for all $v \in V(G)$ and $i \in [k-1]$ and $w_Q$ for all $Q\in \binom{V(G)}{2}$. 
    
    We then add arbitrary arcs to $D$ such that $D[W]$ is a $k$-arc-strong tournament. 
    Further, for every $a \in A$ and $i \in [k]$, we add an arc from $a$ to $w_a^{i}$. For every $a \in A$ and $i \in [k-1]$, we add an arc from $z_a^{i}$ to $a$. For every $b \in B$ and $i \in [k]$, we add an arc from $w_b^{i}$ to $b$. For every $b \in B$ and $i \in [k-1]$, we add an arc from $b$ to $z_b^{i}$. Next, for every $Q \in \binom{V(G)}{2}$ and every $a\in Q \cap A$, we add an arc from $a$ to $w_Q$ and for every $Q \in \binom{V(G)}{2}$ and every $b\in Q \cap B$, we add an arc from $w_Q$ to $b$. Finally, for every $a\in A$ and $b \in B$ such that $a$ and $b$ are linked by an edge of $G$, we let $A(D)$ contain an arc from $a$ to $b$.

    By construction, we have that $D$ is an oriented graph. Next, it is well-known that a bipartition of a bipartite graph can be computed in polynomial time. Hence, as $k$ is fixed and by construction, we obtain that $D$ can be constructed from $G$ in polynomial time. It remains to prove that $x=\lceil \frac{n-y}{2}\rceil$.

    First let $\mathcal{P}$ be a $P_3$-packing of size $y$ in $G$. We let $\mathcal{X}_1\subseteq \binom{V(D)}{3}$ be the set that consists of $V(P)$ for every $P \in \mathcal{P}$. 
    Let $V(\mathcal{P}) = \bigcup_{P\in \Pcal} V(P)$.
    Observe that, as the paths in $\mathcal{P}$ are vertex-disjoint, we have $|V(G)\setminus V(\mathcal{P})|=n-3|\mathcal{P}|=n-3y$. Hence there exists a set $\mathcal{Q}\subseteq \binom{V(G)}{2}$ of size $\lceil\frac{n-3y}{2}\rceil$ such that $\bigcup_{Q\in  \mathcal{Q}}Q=V(G)\setminus V(\mathcal{P})$. We now let $\mathcal{X}_2\subseteq \binom{V(D)}{3}$ be the set that consists of $Q \cup w_Q$ for every $Q \in \mathcal{Q}$. We set $\mathcal{X}=\mathcal{X}_1 \cup \mathcal{X}_2$. Observe that, as $y$ is integer, we have $|\mathcal{X}|=|\mathcal{X}_1|+|\mathcal{X}_2|=y+\lceil\frac{n-3y}{2}\rceil=\lceil\frac{n-y}{2}\rceil$.

    Let $D'=\Inv(D;\mathcal{X})$. We now show that $D'$ is $k$-arc-strong. First observe that $|X\cap W|\leq 1$ holds for every $X \in \mathcal{X}$. It hence follows that $D'[W]=D[W]$ and hence $D'[W]$ is $k$-arc-strong by construction. Hence, by Proposition~\ref{prop:showing_k_strong}, it suffices to prove that $\lambda_{D'}(v,W)\geq k$ and $\lambda_{D'}(W,v)\geq k$ hold for every $v \in V(G)$. Fix some $v \in V(G)$. By symmetry, we may suppose that $v \in A$. By construction, for every $i \in [k]$ and $X \in \mathcal{X}$, we have that $X$ does not contain $w_v^{i}$. It follows that $A(D')$ contains an arc from $v$ to $w_v^{i}$ for $i \in [k]$. This yields that $\lambda_{D'}(v,W)\geq k$. Next observe that for every $i \in [k-1]$ and $X \in \mathcal{X}$, we have that $X$ does not contain $z_v^{i}$. It follows that $A(D')$ contains an arc from $z_v^{i}$ to $v$ for $i \in [k-1]$. By \Cref{prop:showing_k_strong}, it hence suffices to prove that $D'$ contains a directed path from $W$ to $v$ that does not use any of these arcs. 
    
    First suppose that $v \in V(G)\setminus V(\mathcal{P})$.
    Then, by construction, there exist some $Q \in \mathcal{Q}$ with $v \in Q$, and exactly one set $X \in \mathcal{X}$ such that $\{v,w_Q\}\subseteq X$. It follows that $A(D')$ contains an arc from $w_Q$ to $v$, forming the desired directed path. Now suppose that $v \in V(P)$ for some $P \in \mathcal{P}$. Then, as $P$ is a path, there exists some $b \in N_P(v)$. Observe that $b \in B$ as $G$ is a bipartite graph. Then there exists exactly one $X \in \mathcal{X}$ such that $\{v,b\}\subseteq X$. It follows that $A(D')$ contains an arc from $b$ to $v$. Moreover, observe that for every $X \in \mathcal{X}$, we have that $X$ does not contain $w_b^{1}$. It follows that $A(D')$ contains an arc from $w_b^{1}$ to $b$. The concatenation of these two arcs is a directed path of the desired form. It follows that $\lambda_{D'}(W,v)\geq k$. We obtain that $D'$ is $k$-arc-strong. This yields $x\leq |\mathcal{X}|=\lceil\frac{n-y}{2}\rceil$.
    
    \medskip

    Now let $\mathcal{X}\subseteq \binom{V(D)}{3}$ be a set such that $\Inv(D;\mathcal{X})$ is $k$-arc-strong with $|\mathcal{X}|=x$. Without loss of generality, we may suppose that for every $X \in \mathcal{X}$ and every $v \in X$, there exists some $v'\in X\cap N_{\UG(D)}(v)$. Let $D'=\Inv(D;\mathcal{X})$. Now let $\mathcal{X}_0$ consist of all elements $X$ of $\mathcal{X}$ that satisfy $|X|=3$ and $X \subseteq V(G)$. Further, let $\mathcal{X}_1$ be a maximal set of elements of $\mathcal{X}_0$ that are pairwise disjoint and let $\mathcal{P}$ be the collection of subgraphs of $G$ that consists of $G[X]$ for every $X \in \mathcal{X}_1$. For every $P\in \mathcal{P}$, as $|V(P)|=3$, by the assumption on $\mathcal{X}$ and as $G$ is bipartite, we obtain that $P$ is a 3-path. As the sets in $\mathcal{X}_1$ are pairwise disjoint by assumption, we obtain that $\mathcal{P}$ is 3-path packing in $G$.

    Now let $\mathcal{X}_2\subseteq 2^{V(G)}$ be the set that contains $X \cap(V(G)\setminus V(\mathcal{P}))$ for every $X\in \Xcal\setminus \mathcal{X}_1$. By assumption, we have that $|X|\leq 2$ for every $X \in \mathcal{X}_2$.
    \begin{claim}\label{yxcfvg}
        $\bigcup_{X\in \Xcal_1 \cup \Xcal_2} X =V(G)$.
    \end{claim}
    \begin{proofclaim}
        Clearly, we have $\bigcup_{X\in \Xcal_1 \cup \Xcal_2}X\subseteq V(G)$. Now suppose for the sake of a contradiction that there exists some $v \in V(G)$ that is not contained in $\bigcup_{X\in \Xcal_1 \cup \Xcal_2}X$.  We then obtain by construction that $v$ is not contained in $X$ for any $X \in \mathcal{X}$. Without loss of generality, we may suppose that $v \in A$. This yields $d_{D'}^-(v)=d_D^-(v)=k-1$, a contradiction to $D'$ being $k$-arc-strong.
    \end{proofclaim}
    By Claim~\ref{yxcfvg}, together with the fact that $|X|=3$ for all $X \in \mathcal{X}_1$ and $|X|\leq 2$ for all $X \in \mathcal{X}_2$, we have $n \leq 3|\mathcal{X}_1|+2|\mathcal{X}_2|$. Moreover, as $\mathcal{P}$ is a 3-path packing in $G$, we have $|\mathcal{X}_1|=|\mathcal{P}|\leq y$. This yields $x=|\mathcal{X}|=|\mathcal{X}_1|+|\mathcal{X}_2|\geq |\mathcal{X}_1|+\frac{n-3|\mathcal{X}_1|}{2}=\frac{n-|\mathcal{X}_1|}{2}\geq \frac{n-y}{2}$. As $x$ is integer, we obtain that $x \geq \lceil\frac{n-y}{2}\rceil$. This finishes the proof.
\end{proof}

We are now ready to prove the following more technical restatement of Theorem~\ref{thm:inapproximability} for $p=3$. 
With Lemma~\ref{tf86zguoi} at hand, it remains to verify its algorithmic consequences.

\begin{theorem}
    \label{inappr1}
    There exists $\epsilon>0$ such that, for every $k \geq 1$, unless $P=NP$, there does not exist a polynomial-time algorithm that, given a $2k$-edge-connected oriented graph $D$ and a positive integer $t$, returns `yes' if $\inv{k}{3}(D)\leq t$ and `no' if $\inv{k}{3}(D)\geq (1+\epsilon)t$.
\end{theorem}
\begin{proof}
    Let $\epsilon=\min\{\frac{1}{2}\epsilon'c,\frac{1}{2}\}$, where $\epsilon'$ and $c$ are the constants from Lemma~\ref{etrztuzibjn} and fix some $k \geq 1$. Suppose that there exists a polynomial-time algorithm $A$ that, given a $2k$-edge-connected oriented graph $D$ and a positive integer $t$, returns `yes' if $\inv{k}{3}(D)\leq t$ and `no' if $\inv{k}{3}(D)\geq (1+\epsilon)t$. We will show that $P=NP$ using Lemma~\ref{etrztuzibjn}. Let $(G,t')$ be a bipartite instance of P3P with $n=|V(G)|$ and $t'\geq cn$. We describe a polynomial-time algorithm that outputs `yes' if $G$ admits a $P_3$-packing of size $t'$ and outputs `no' if $G$ does not admit a $P_3$-packing of size $(1-\epsilon')t'$.
    
    First, if $t'\leq \frac{2}{\epsilon}$, we solve the problem by a brute force approach and output an appropriate answer, which can clearly be done in time $O(n^{6/\epsilon})$. Henceforth, we may hence suppose that $t'> \frac{2}{\epsilon}$.    
    Next, if $n-t'\leq \frac{4}{\epsilon}$, then we output `no'. In order to justify this, suppose for the sake of a contradiction that $G$ admits a $P_3$-packing $\mathcal{P}$ of size $t'$. As $\mathcal{P}$ is a $P_3$-packing, we obtain $2t'=\sum_{P \in \mathcal{P}}|V(P)|-t'\leq n-t'\leq \frac{4}{\epsilon}$, so $t'\leq \frac{2}{\epsilon}$, a contradiction. Hence our output is of the desired form. From now on, we may suppose that $n-t'\geq \frac{4}{\epsilon}$ and hence that $\epsilon \lceil\frac{n-t'}{2}\rceil\geq 2$.
    
    By Lemma~\ref{tf86zguoi}, in polynomial time, one can compute an oriented graph $D$ whose size is polynomial in the size of $G$ such that for $n=|V(G)|$, $x=\inv{k}{3}(D)$ and $y$ being the maximum size of a $P_3$-packing in $G$, we have $x=\lceil\frac{n-y}{2}\rceil$. We now apply $A$ to $D$ and $t=\lceil\frac{n-t'}{2}\rceil$ and output the output of $A$. This finishes the description of our algorithm. As the size of $D$ is polynomial in the size of $G$ and $A$ runs in polynomial time by assumption, we obtain that our entire algorithm runs in polynomial time. It remains to prove that the output is of the desired form.

    First suppose that $y\geq t'$. It then follows that $x=\lceil\frac{n-y}{2}\rceil\leq \lceil\frac{n-t'}{2}\rceil=t$. It follows by assumption that $A$ outputs `yes' and so our algorithm outputs `yes'.

    Now suppose that $y\leq (1-\epsilon')t'$. 
    As $t'\geq cn$, by definition of $\epsilon$ and since $\epsilon \lceil\frac{n-t'}{2}\rceil\geq 2$, we obtain that
    \begin{align*}
        x=\left\lceil\frac{n-y}{2}\right\rceil
        &\geq \frac{n-(1-\epsilon')t'}{2}\\
        &\geq \lceil\frac{n-t'}{2}\rceil-1+\frac{\epsilon't'}{2}\\
        &\geq \lceil\frac{n-t'}{2}\rceil-1+\frac{\epsilon'cn}{2}\\
        &\geq  \lceil\frac{n-t'}{2}\rceil-1+\epsilon n\\
        &\geq  \lceil\frac{n-t'}{2}\rceil+2 \epsilon \lceil \frac{n-t'}{2}\rceil-2\\
        &\geq  \lceil\frac{n-t'}{2}\rceil+2 \epsilon \lceil \frac{n-t'}{2}\rceil-\epsilon\lceil \frac{n-t'}{2}\rceil\\
        &=(1+\epsilon)\lceil\frac{n-t'}{2}\rceil.
    \end{align*}
    Hence, by assumption, we have that $A$ outputs `no' and hence our algorithm outputs `no'. Hence, the output of our algorithm is always appropriate. It follows from Lemma~\ref{etrztuzibjn} that $P=NP$.
    \end{proof}
    
\subsection{The case \texorpdfstring{$p\geq 4$}{p at least 4}}
\label{sec:inapproximability:pgeq4}

Here, we deal with the remaining cases, that is, $p \geq 4$. While we follow similar ideas as for the case $p=3$, we start from a problem in hypergraphs. A {\it hypergraph matching} is a collection of pairwise disjoint hyperedges. We consider the following algorithmic problem.

\defproblem{Hypergraph Matching (HM)}{A hypergraph $\mathcal{H}$ and an integer $t$.}{Does $\mathcal{H}$ admit a hypergraph matching of size $t$?}

A hypergraph $\mathcal{H}$ is called {\it $s$-uniform} for some positive integer $s$ if $|e|=s$ holds for all $e \in E(\mathcal{H})$. 
An instance $(\mathcal{H},t)$ of HM is called $s$-uniform if $\mathcal{H}$ is $s$-uniform.
For our reduction, we need the following result on the inapproximability of HM in a restricted case, which is due to Chleb\'ik and Chleb\'ikov\'a~\cite{CHLEBIK2006320}. 

\begin{lemma}[\cite{CHLEBIK2006320}]
    \label{tzzigzgi}
    Unless $P=NP$, there does not exist a polynomial-time algorithm whose input is a 3-uniform instance $(\mathcal{H},t)$ of HM such that every $v \in V(\mathcal{H})$ is contained in at most two hyperedges of $E(\mathcal{H})$, that outputs `yes' if $\mathcal{H}$ admits a hypergraph matching of size $t$ and `no' if $\mathcal{H}$ does not admit a hypergraph matching of size $\frac{94}{95}t$.
\end{lemma}

In order to make use of this result, we need to show that we can make the extra assumption that the target value in the given instance is linear in its number of vertices. To this end, we need the following simple result.

\begin{proposition}\label{ysetrxdtcfgv}
    Let $\mathcal{H}$ be a 3-uniform hypergraph on $n$ vertices such that every $v \in V(\mathcal{H})$ appears in at least one and in at most two hyperedges of $E(\mathcal{H})$. Then $\mathcal{H}$ admits a hypergraph matching of size at least $\frac{1}{9}n$.
\end{proposition}
\begin{proof}
    Suppose otherwise and that $\mathcal{H}$ is a minimum counterexample to that statement. Let $e_0$ be an arbitrary hyperedge in $E(\mathcal{H})$ and let $\mathcal{H}_0$ be obtained from $\mathcal{H}$ by deleting all vertices contained in $e_0$ and all hyperedges sharing at least one vertex with $e_0$. We further obtain $\mathcal{H}_1$ from $\mathcal{H}_0$ by deleting all vertices in $V(\mathcal{H}_0)$ that are not contained in any hyperedge of $E(\mathcal{H}_0)$. 
    
    By construction and as $\mathcal{H}$ is 3-uniform, we have $|V(\mathcal{H}_0)|=|V(\mathcal{H})|-3$. Next, observe that, as every vertex in $V(\mathcal{H})$ is contained in at least one hyperedge of $E(\mathcal{H})$, every vertex in $V(\mathcal{H}_0) \setminus V(\mathcal{H}_1)$ is contained in a hyperedge in $E(\mathcal{H})$ that intersects $e_0$ in at least one vertex. As every vertex in $V(\mathcal{H})$ is contained in at most two hyperedges of $E(\mathcal{H})$ and $|e_0|=3$, there at most 3 such hyperedges. Hence, as $\mathcal{H}$ is 3-uniform, we obtain that $|V(\mathcal{H}_0) \setminus V(\mathcal{H}_1)|\leq 6$. It follows that $|V(\mathcal{H}) \setminus V(\mathcal{H}_1)|= |V(\mathcal{H}) \setminus V(\mathcal{H}_0)|+|V(\mathcal{H}_0) \setminus V(\mathcal{H}_1)|\leq 3+6=9$. Next, as $E(\mathcal{H}_1)\subseteq E(\mathcal{H})$, we obtain that $\mathcal{H}_1$ is 3-uniform and every $v \in V(\mathcal{H}_1)$ is contained in at most two hyperedges in $E(\mathcal{H}_1)$. Moreover, by construction, we have that every $v \in V(\mathcal{H}_1)$ is contained in at least one hyperedge in $E(\mathcal{H}_1)$. Hence, as $\mathcal{H}_1$ is smaller than $\mathcal{H}$, we obtain that $\mathcal{H}_1$ admits a hypergraph matching $M_1$ of size at least $\frac{1}{9}|V(\mathcal{H}_1)|$. Let $M=M_1 \cup e_0$. Observe that, as $e_0$ is disjoint from all hyperedges in $E(\mathcal{H}_1)$ by construction, we have that $M$ is a hypergraph matching in $\mathcal{H}$. Moreover, we have $|M|=|M_0|+1\geq \frac{1}{9}|V(\mathcal{H}_1)|+1=\frac{1}{9}(|V(\mathcal{H}_1)|+9)\geq \frac{1}{9}|V(\mathcal{H}_1)|$. This contradicts $\mathcal{H}$ being a counterexample.
\end{proof}

We are now ready to show that the extra assumption can be imposed on the hypergraph in case $s=3$.

\begin{lemma}\label{dtfuziguohipjo}
    Unless $P=NP$, there does not exist a polynomial-time algorithm whose input is a 3-uniform instance $(\mathcal{H},t)$ of HM such that $t \geq \frac{1}{9}n$ and every $v \in V(\mathcal{H})$ is contained in at most two hyperedges of $E(\mathcal{H})$, and that outputs `yes' if $\mathcal{H}$ admits a hypergraph matching of size $t$ and `no' if $\mathcal{H}$ does not admit a hypergraph matching of size at least $\frac{94}{95}t$.
\end{lemma}
\begin{proof}
    Suppose that there exists a polynomial-time algorithm $A$ whose input is a 3-uniform instance $(\mathcal{H},t)$ of HM such that $t \geq \frac{1}{9}n$ and every $v \in V(\mathcal{H})$ is contained in at most two hyperedges of $E(\mathcal{H})$, and that outputs `yes' if $H$ admits a hypergraph matching of size $t$ and `no' if $\mathcal{H}$ does not admit a hypergraph matching of size at least $\frac{94}{95}t$. We will show that $P=NP$ using Lemma~\ref{tzzigzgi}. Let $(\mathcal{H},t)$  be a 3-uniform instance of HM such that every $v \in V(\mathcal{H})$ is contained in at most two hyperedges of $E(\mathcal{H})$. Let $\mathcal{H}'$ be obtained from $\mathcal{H}$ by deleting all vertices that are not contained in any hyperedge of $\mathcal{H}$. Observe that every vertex in $V(\mathcal{H}')$ is contained in at least one and at most 2 hyperedges of $\mathcal{H}'$. Next, if $t \leq \frac{1}{9}|V(\mathcal{H}')|$, we let our algorithm output `yes'. We may hence suppose that $t \geq \frac{1}{9}|V(\mathcal{H}')|$. We can now apply $A$ to $(\mathcal{H}',t)$ and output the output of $A$. As $\mathcal{H}'$ can be computed from $\mathcal{H}$ in polynomial time, we can check in polynomial time whether $t \geq \frac{1}{9}|V(\mathcal{H}')|$ holds in polynomial time and $A$ runs in polynomial time, we obtain that our entire algorithm runs in polynomial time.

    For the correctness of our algorithm, first suppose that $\mathcal{H}$ admits a hypergraph matching of size $t$. If $t\leq \frac{1}{9}|V(\mathcal{H}')|$, then our algorithm outputs `yes' by construction. We may hence suppose that $t\geq \frac{1}{9}|V(\mathcal{H}')|$. Then $A$ is applied to $(\mathcal{H}',t)$. As $M$ is also a hypergraph matching of size $t$ in $\mathcal{H}'$, we obtain that $A$ outputs `yes' and hence our algorithm outputs `yes'. In either case, our algorithm outputs `yes'.

    Now suppose that $\mathcal{H}$ does not admit a hypergraph matching of size at least $\frac{94}{95}t$. As every hypergraph matching in $\mathcal{H}'$ is also a hypergraph matching in $\mathcal{H}$, it follows that $\mathcal{H}'$ does also not admit a hypergraph matching of size at least $\frac{94}{95}t$. It follows from Proposition~\ref{ysetrxdtcfgv} that $t\geq \frac{94}{95}t\geq \frac{1}{9}|V(\mathcal{H}')|$. Hence $A$ is applied to $(\mathcal{H}',t)$. Again, as every hypergraph matching in $\mathcal{H}'$ is also a hypergraph matching in $\mathcal{H}$ and by the assumption on $A$, it follows that $A$ outputs `no'. Hence our algorithm outputs `no'.
    This yields that our algorithm has the desired properties and $P=NP$ follows by Lemma~\ref{tzzigzgi}.
\end{proof}

We are now ready to generalize the result on the adapted version of HM from 3-uniform hypergraphs to $s$-uniform hypergraphs for arbitrary $s\geq 4$.

\begin{lemma}\label{dutfzigui}
    For every integer $s\geq 3$, there exist $\epsilon'_s>0$ and $c_s>0$ such that, unless $P=NP$, there is no polynomial-time algorithm that, given an $s$-uniform instance $(\mathcal{H},t)$ of HM with $t \geq c_sn$, outputs `yes' if $\mathcal{H}$ admits a hypergraph matching of size $t$ and `no' if $\mathcal{H}$ does not admit a hypergraph matching of size at least $(1-\epsilon'_s)t$.
\end{lemma}
\begin{proof}
    Clearly, be Lemma~\ref{dtfuziguohipjo}, it suffices to prove the statement for $s \geq 4$.
    Fix some $s\geq 4$ and let $\epsilon'_s=\frac{1}{95}$ and $c_s=\frac{1}{9(s-2)}$. Suppose that there exists a polynomial-time algorithm $A$ that, given an $s$-uniform instance $(\mathcal{H}',t')$ of HM with $t' \geq c_sn$, outputs `yes' if $\mathcal{H}'$ admits a hypergraph matching of size $t'$ and `no' if $\mathcal{H}'$ does not admit a hypergraph matching of size at least $(1-\epsilon'_s)t'$. We will show that $P=NP$ using Lemma~\ref{dtfuziguohipjo}. Let $(\mathcal{H},t)$ be a 3-uniform instance of HM with $t \geq \frac{1}{9}|V(\mathcal{H})|$ and such that every $v \in V(\mathcal{H})$ is contained in at most two hyperedges of $E(\mathcal{H})$. Observe that, as $\mathcal{H}$ is 3-uniform and every $v \in V(\mathcal{H})$ is contained in at most two hyperedges of $E(\mathcal{H})$, we have 
    \[
    |E(\mathcal{H})|=\frac{1}{3}\sum_{e \in E(\mathcal{H})}\sum_{v \in e}1=\frac{1}{3}\sum_{v \in V(\mathcal{H})}\sum_{\substack{e \in E(\mathcal{H})\\ v \in e}}1\leq \frac{2}{3}|V(\mathcal{H})|\leq |V(\mathcal{H})|.
    \]
    We now construct a hypergraph $\mathcal{H'}$ in the following way. We let $V(\mathcal{H}')$ contain $V(\mathcal{H})$ and $s-3$ new vertices $x_e^1,\ldots,x_e^{s-3}$ for every $e \in E(\mathcal{H})$. Next, for every $e \in E(\mathcal{H})$, we let $E(\mathcal{H}')$ contain a hyperedge $f_e$ with $f_e=e\cup \{x_e^1,\ldots,x_e^{s-3}\}$. This finishes the description of $\mathcal{H}'$. It is easy to see that $\mathcal{H}'$ can be computed from $\mathcal{H}$ in polynomial time. Moreover, for every $e \in E(\mathcal{H})$, we have $|f_e|=|e|+|\{x_e^1,\ldots,x_e^{s-3}\}|=3+(s-3)=s$, so $\mathcal{H'}$ is $s$-uniform. 
    Finally, observe that, as $|E(\mathcal{H})|\leq V(\mathcal{H})$, we have $|V(\mathcal{H}')|=|V(\mathcal{H})|+(s-3)|E(\mathcal{H})|\leq (s-2)|V(\mathcal{H})|$. It follows that $t\geq \frac{1}{9}|V(\mathcal{H})|\geq c_s|V(\mathcal{H}')|$. We can hence apply $A$ to $(\mathcal{H}',t)$ and output the output of $A$. This finishes the description of our algorithm. As $\mathcal{H}'$ can be computed from $\mathcal{H}$ in polynomial time and $A$ runs in polynomial time, we obtain that the entire algorithm runs in polynomial time.

    In order to show that our algorithm has the desired properties, first suppose that $\mathcal{H}$ admits a hypergraph matching $M$ of size $t$. Let $M'$ be the subset of $E(\mathcal{H}')$ that contains $f_e$ for every $e \in M$. Observe that $M'$ is a hypergraph matching in $\mathcal{H}'$ of size $t$. It follows by assumption that $A$ outputs `yes' and hence our algorithm outputs `yes'.

    Now suppose that $\mathcal{H}$ does not admit a hypergraph matching of size at least $(1-\epsilon'_s)t$. Suppose for the sake of a contradiction that $\mathcal{H}'$ admits a hypergraph matching $M'$ of size at least $(1-\epsilon'_s)t$. Let $M$ be the collection of subsets of $V(\mathcal{H})$ that consists of $f \cap V(\mathcal{H})$ for every $f \in M$. Then $M$ is a hypergraph matching of size at least $(1-\epsilon'_s)t$ in $\mathcal{H}$, a contradiction to the assumption. It follows that $\mathcal{H}'$ does not admit a hypergraph matching of size at least $(1-\epsilon'_s)t$. Hence, by our assumption, $A$ outputs `no' and hence our algorithm outputs `no'. 
    Hence our algorithm has the desired properties and $P=NP$ follows by Lemma~\ref{dtfuziguohipjo}.
\end{proof}

The following result is the main technical contribution of our reduction. We show how to compute a corresponding instance of the inversion problem when given an appropriate instance of HM.

\begin{lemma}\label{ytdexzcu}
    For all fixed integers $s \geq 3$ and $k\geq 1$, given an $s$-uniform hypergraph $\mathcal{H}$, in polynomial time, one can compute an oriented graph $D$ whose size is polynomial in the size of $\mathcal{H}$ such that for $x$ being the maximum size of a hypermatching in $\mathcal{H}$ and $y=\inv{k}{s+1}(D)$, and $n=|V(\mathcal{H})|$, we have $y=\lceil\frac{n-x}{s-1}\rceil$.
\end{lemma}
\begin{proof}
    Let an $s$-uniform hypergraph $\mathcal{H}$ be given. We now construct an oriented graph $D$.
    We first let $V(D)$ contain $V(\mathcal{H})$. Next, for every $e \in E(\mathcal{H})$, we let $V(D)$ contain a vertex $w_e$ and we set $W_1=\{w_e: e \in E(\mathcal{H})\}$. Next, for every $Q \in \binom{V(\mathcal{H})}{s-1}$, we let $V(D)$ contain a vertex $w_Q$ and we set $W_2=\{w_Q: Q \in \binom{V(\mathcal{H})}{s-1}\}$. We further let $V(D)$ contain a set $W_3^v$ of $k$ new vertices for every $v \in V(D)$ and we set $W_3=\bigcup_{v \in V(\mathcal{H})}W_3^v$. We also let $V(D)$ contain a set $W_4^v$ of $k-1$ new vertices for every $v \in V(\mathcal{H})$ and we set $W_4=\bigcup_{v \in V(\mathcal{H})}W_4^v$.  We set $W=W_1\cup W_2 \cup W_3 \cup W_4$ and $V(D)=V(\mathcal{H})\cup W$. 
    
    Next, we let $A(D)$ contain a set of arcs such that $D[W]$ is a $k$-arc-strong tournament. Further, for every $e \in E(\mathcal{H})$ and every $v \in e$, we let $A(D)$ contain an arc from $w_e$ to $v$. Next, for every $Q \in \binom{V(\mathcal{H})}{s-1}$ and every $v \in Q$, we let $A(D)$ contain an arc from $w_Q$ to $v$. Further, for every every $v \in V(\mathcal{H})$ and $w \in W_3^v$, we let $A(D)$ contain an arc from $w$ to $v$. Finally, for every every $v \in V(\mathcal{H})$ and every $w \in W_4^v$, we let $A(D)$ contain an arc from $v$ to $w$. This finishes the description of $D$. It is easy to see that $D$ is an oriented graph and that $D$ can be constructed from $\mathcal{H}$ in polynomial time. It remains to prove that $y=\lceil\frac{n-x}{s-1}\rceil$.

    First, let $M$ be a hypergraph matching in $\mathcal{H}$ of size $x$. Let $V_1=\bigcup_{e\in M} e$ and $V_2=V(\mathcal{H})\setminus V_1$. We now let $\mathcal{X}_1\subseteq \binom{V(D)}{s+1}$ be the collection of sets that contains $e \cup w_e$ for every $e \in M$. Observe that $|V_2|=n-|V_1|=n-sx$. 
    It follows that there exists a collection $\mathcal{Q}$ of sets in $\binom{V(\mathcal{H})}{s-1}$ with $|\Qcal|=\lceil\frac{n-sx}{s-1}\rceil$ such that $\bigcup_{Q\in \mathcal{Q}}Q=V_2$. 
    Let $\mathcal{X}_2\subseteq \binom{V(D)}{s}$ be the collection of sets that contains $Q \cup w_Q$ for every $Q \in \mathcal{Q}$ and let $\mathcal{X}=\mathcal{X}_1 \cup \mathcal{X}_2$. Observe that, as $x$ is integer, we have $|\mathcal{X}|=|\mathcal{X}_1|+|\mathcal{X}_2|=x+\lceil \frac{n-sx}{s-1}\rceil=\lceil \frac{n-x}{s-1}\rceil$. 

    Let $D'=\Inv(D;\mathcal{X})$. 
    We will show that $D'$ is $k$-arc-strong. First observe that $D'[W]=D[W]$ and hence by assumption $D'[W]$ is $k$-arc-strong. 
    Hence, by Proposition~\ref{prop:showing_k_strong}, it suffices to prove that for every $v \in V(\mathcal{H})$, we have $|N^+_{D'}(v) \cap W| \geq k$ and $|N^-_{D'}(v) \cap W| \geq k$. Fix some $v \in V(\mathcal{H})$. 
    First observe that $A(D)$ contains $k$ arcs from $W_3$ to $v$ and that $\bigcup_{X\in \Xcal}X \cap W_3=\emptyset$. 
    It follows that these arcs are also contained in $A(D')$. 
    Next observe that $A(D)$ contains $k-1$ arcs from $v$ to $W_4$ and that $\bigcup_{X\in \Xcal}X \cap W_4=\emptyset$. Again, these arcs are contained in $A(D')$. 
    By \Cref{prop:showing_k_strong}, it hence suffices to prove that $D'$ contains an arc from $v$ to $W_1\cup W_2$. First suppose that there exists some $e \in M$ with $v \in e$. Then $\mathcal{X}$ contains exactly one set $X$ with $\{v,w_e\}\subseteq X$. It follows that $A(D')$ contains an arc from $v$ to $w_e$. Otherwise, by definition, there exists some $Q \in \mathcal{Q}$ such that $v \in Q$. Then $\mathcal{X}$ contains exactly one set $X$ with $\{v,w_Q\}\subseteq X$. It follows that $A(D')$ contains an arc from $v$ to $w_Q$. We obtain that $D'$ is $k$-arc-strong. This yields $y\leq |\mathcal{X}|=\lceil \frac{n-x}{s-1}\rceil$.
    
    \medskip

    Now let $\mathcal{X}$ be a \InvSeq{k}{s+1} of $D$ of size $y$. Without loss of generality, we may suppose that for every $X \in \mathcal{X}$ and every $v \in X$, there exists some $v'\in X \cap N_{\UG(D)}(v)$. Let $\mathcal{X}_1$ contain all $X \in \mathcal{X}$ with $|X\cap V(\mathcal{H})|\geq s$ and $\mathcal{X}_2=\mathcal{X}\setminus \mathcal{X}_1$.
    \begin{claim}
        Let $X \in \mathcal{X}_1$. Then $X=e \cup \{w_e\}$ for some $e \in E(\mathcal{H})$.
    \end{claim}
    \begin{proof}
        If $|X\cap W|\geq 2$, we obtain $|X \cap V(\mathcal{H})|\leq (s+1)-|X\cap W|\leq s-1$, a contradiction to $X \in \mathcal{X}_1$.  If $X\cap W=\emptyset$, then, by construction, we have that $X$ is an independent set in $\UG(D)$, a contradiction to the assumption. Finally, if $X\cap W$ consists of exactly one vertex $w$, then, by assumption and as $V(\mathcal{H})$ is an independent set in $\UG(D)$, we obtain that all vertices in $X\setminus w$ are adjacent to $w$ in $\UG(D)$. As $s \geq 3$ and by construction, the statement follows.
    \end{proof}

    Let $M$ consist of all hyperedges $e \in E(\mathcal{H})$ such that $e \cup \{w_e\} \in \mathcal{X}_1$. Further, let $M_0$ be a maximal collection of elements of $M$ that are pairwise disjoint. Let $M_1$ be the collection of sets that consists of $e \setminus \bigcup_{f\in M_0}f$ for every $e \in M\setminus M_0$. Further, let $M_2$ be the collection of sets that consists of $X \cap V(\mathcal{H})$ for all $X \in \mathcal{X}_2$. Observe that $|m|=s$ for all $m \in M_0$ and $|m|\leq s-1$ for all $m \in M_1 \cup M_2$.
    \begin{claim}\label{dfipei}
        We have $M_0 \cup M_1 \cup M_2=V(\mathcal{H})$.
    \end{claim}
    \begin{proofclaim}
        By definition, we have $M_0 \cup M_1 \cup M_2\subseteq V(\mathcal{H})$. 
        Now suppose that there exists some $v \in V(\mathcal{H})\setminus M_0\cup M_1 \cup M_2$. 
        By construction, it follows that $v \in V(\mathcal{H})\setminus \bigcup_{X\in\mathcal{X}}X$. We hence have $d_{\Inv(D;\mathcal{X})}^+(v)=d_D^+(v)=k-1$, a contradiction to $\Inv(D;\mathcal{X})$ being $k$-arc-strong.
    \end{proofclaim}
    
    By Claim~\ref{dfipei} and the preceding remark, we obtain that $s|M_0|+(s-1)(|M_1|+|M_2|)\geq n$ and hence $|M_0|+|M_1|+|M_2|\geq \frac{n-|M_0|}{s-1}$. By the choice of $\mathcal{X}$ and as $M_0$ is a hypergraph matching in $\mathcal{H}$, we obtain
    \begin{align*}
        y&=|\mathcal{X}|\\
        &\geq |M_0|+|M_1|+|M_2|\\
        &\geq \frac{n-|M_0|}{s-1}\\
        &\geq \frac{n-x}{s-1}.
    \end{align*}
    As $y$ is integer, we obtain $y \geq \lceil\frac{n-x}{s-1}\rceil$. This finishes the proof.
\end{proof} 

We are now ready to prove the following more technical restatement of Theorem~\ref{thm:inapproximability} for $p \geq 4$. The proof relies on Lemma~\ref{ytdexzcu} and it is very similar to the conclusion of Theorem~\ref{inappr1}.

\begin{theorem}\label{inappr2}
    For every fixed $p\geq 4$, there exists $\epsilon_p>0$ such that, for every $k \geq 1$, unless $P=NP$, there does not exist a polynomial-time algorithm that, given a $2k$-edge-connected oriented graph $D$ and a positive integer $t$, returns `yes' if $\inv{k}{p}(D)\leq t$ and `no' if $\inv{k}{p}(D)\geq (1+\epsilon_p)t$.
\end{theorem}
\begin{proof}
    Let $\epsilon_p=\min\{\frac{1}{2}\epsilon'_{p-1}c_{p-1},\frac{1}{2}\}$, where $\epsilon'_{p-1}$ and $c_{p-1}$ are the constants from Lemma~\ref{dutfzigui} and fix some $k \geq 1$. Suppose that there exists a polynomial-time algorithm $A$ that, given a $2k$-edge-connected oriented graph $D$ and a positive integer $t$, returns `yes' if $\inv{k}{p}(D)\leq t$ and `no' if $\inv{k}{p}(D)\geq (1+\epsilon_p)t$. We will show that $P=NP$ using Lemma~\ref{dutfzigui}. Let $(\mathcal{H},t')$ be a $(p-1)$-uniform instance of  hypergraph matching with $t'\geq  c_{p-1}n$ where $n=|V(\mathcal{H})|$.\medskip
    
    First, if $t'\leq \frac{2}{\epsilon_p}$, we solve the problem by a brute force approach and output an appropriate answer. We may hence suppose that $t'> \frac{2}{\epsilon_p}$. Next suppose that $n-t'\leq \frac{2(p-2)}{\epsilon_p}$. We then output `no'. In order to justify this, suppose for the sake of a contradiction that $\mathcal{H}$ admits a hypergraph matching $M$ of size $t'$. As $\mathcal{H}$ is $(p-1)$-uniform and $M$ is a hypergraph matching, we obtain $t'(p-2)=\sum_{e \in M}|e|-t'\leq n-t'\leq \frac{2(p-2)}{\epsilon_p}$, so $t'\leq \frac{2}{\epsilon_p}$, a contradiction. Hence our output is of the desired form. From now on, we may suppose that $n-t'\geq \frac{2(p-2)}{\epsilon_p}$ and hence that $\epsilon_p \lceil\frac{n-t'}{p-2}\rceil\geq 2$.
    
    \medskip
    
    By Lemma~\ref{ytdexzcu}, in polynomial time, one can compute an oriented graph $D$ whose size is polynomial in the size of $\mathcal{H}$ such that for $x$ being the maximum size of a hypermatching in $\mathcal{H}$ and $y=\inv{k}{p}(D)$, we have $y=\lceil\frac{n-x}{p-2}\rceil$. We now apply $A$ to $D$ and $t=\lceil\frac{n-t'}{p-2}\rceil$ and output the output of $A$. This finishes the description of our algorithm. As the size of $D$ is polynomial in the size of $\mathcal{H}$ and $A$ runs in polynomial time by assumption, we obtain that our entire algorithm runs in polynomial time. It remains to prove that the output is of the desired form.

    First suppose that $x\geq t'$. It then follows that $y=\lceil\frac{n-x}{p-2}\rceil\leq \lceil\frac{n-t'}{p-2}\rceil=t$. It follows by assumption that $A$ outputs `yes' and so our algorithm outputs `yes'.

    Next suppose that $x\leq (1-\epsilon'_{p-1})t'$. 
    As $t'\geq c_{p-1}n$, by definition of $\epsilon_p$ and $\epsilon_p \lceil\frac{n-t'}{p-2}\rceil\geq 2$, we obtain that
    \begin{align*}
        y=\left\lceil\frac{n-x}{p-2}\right\rceil&\geq \frac{n-(1-\epsilon'_{p-1})t'}{p-2}\\
        &\geq \lceil\frac{n-t'}{p-2}\rceil-1+\frac{\epsilon'_{p-1}t'}{p-2}\\
        &\geq \lceil\frac{n-t'}{p-2}\rceil-1+\frac{\epsilon'_{p-1}c_{p-1}n}{p-2}\\
        &\geq  \lceil\frac{n-t'}{p-2}\rceil-1+\frac{2\epsilon_{p}n}{p-2}\\
        &\geq  \lceil\frac{n-t'}{p-2}\rceil+2 \epsilon_p \lceil \frac{n-t'}{p-2}\rceil-2\\
        &\geq  \lceil\frac{n-t'}{p-2}\rceil+2 \epsilon_p \lceil \frac{n-t'}{p-2}\rceil-\epsilon_p\lceil \frac{n-t'}{p-2}\rceil\\
        &=(1+\epsilon_p)\lceil\frac{n-t'}{p-2}\rceil.
    \end{align*}
    Hence, by assumption, we have that $A$ outputs `no' and hence our algorithm outputs `no'. Hence the output of our algorithm is always appropriate. It follows from Lemma~\ref{dutfzigui} that $P=NP$.
    \end{proof}

\section{Approximation algorithms}
\label{sec:APX_algorithm}

In the previous section, we proved that, unless $P=NP$, for every pair of integers $p\geq 3$ and $k\geq1$, the problem of computing $\inv{k}{p}$ does not admit a PTAS. In this section, we show, on the positive side, that it is actually approximable. In particular, we prove \Cref{thm:APX,thm:APX_p=4}. Formally, we prove the following more general and stronger statement. 

\begin{restatable}{theorem}{thmAPXalgo}
    \label{thm:APX_algo}
    For all fixed integers $p\geq 3$ and $k\geq 1$, there exists a polynomial-time algorithm that, given a $2k$-edge-connected digraph $D$, outputs a \InvSeq{k}{p} $\mathcal{X}$ of $D$ such that 
    \[
        |\Xcal|\leq \eta(p,k) \cdot \inv{k}{p}(D) +c_{k,p},
    \]
    where $\eta(p,k) =\frac{\min\left\{\binom{p}{2},(2k-1)(p-1)\right\}}{\lfloor p/2\rfloor}$ and $c_{k,p}$ is a constant depending only on $p$ and $k$. 
\end{restatable}

Given a digraph $D$ of order $n$ and positive integers $c$ and $p$, it can be checked in  time $n^{pc+O(1)}$ whether $\inv{k}{p}(D) \leq c$. Therefore, we derive the following results, which are of independent interest depending on the respective values of $k$ and $p$. Observe that \Cref{thm:APX_algo} implies  Corollaries \ref{cor:APX_algo_kstrong_ksmaller_p} and ~\ref{cor:APX_algo_kstrong_klarger_p} and that Theorems~\ref{thm:APX} and~\ref{thm:APX_p=4} follow from Corollary~\ref{cor:APX_algo_kstrong_klarger_p}.

\begin{corollary}
    \label{cor:APX_algo_kstrong_ksmaller_p}
    For all fixed integers $p \geq 3$, $k \geq 1$ and real number $\epsilon>0$, there exists a polynomial-time algorithm that, given an digraph $D$, computes an $\alpha$-approximation of $\inv{k}{p}(D)$, where:
    \[
    \alpha = \left\{
    \begin{array}{ll}
        p-1 + \epsilon & \text{if $p$ is even,}\\
        p +\epsilon & \text{otherwise.}
    \end{array}
    \right.
    \]
\end{corollary}

\begin{corollary}
    \label{cor:APX_algo_kstrong_klarger_p}
    For all fixed integers $p \geq 3$, $k \geq 1$ and real number $\epsilon>0$, there exists a polynomial-time algorithm that, given an digraph $D$, computes an $\alpha$-approximation of $\inv{k}{p}(D)$, where:
    \[
    \alpha = \left\{
    \begin{array}{ll}
        (4k-2)(1-\frac{1}{p}) + \epsilon & \text{if $p$ is even,}\\
        4k-2 +\epsilon & \text{otherwise.}
    \end{array}
    \right.
    \]
\end{corollary}

The remainder of this section is dedicated to the proof of Theorem~\ref{thm:APX_algo}. Roughly speaking, the proof of Theorem~\ref{thm:APX_algo} is based on relating  $\inv{k}{p}$ and  $\inv{k}{2}$. We first show in \Cref{lem:bounding_invk2_invkp} that $\inv{k}{2}$ is not larger than $\inv{k}{p}$ by more than a factor of $(2k-1)(p-1)$. After, in the main proof of \Cref{thm:APX_algo}, we show that given a sufficiently large \InvSeq{k}{2}, we can efficiently compute a \InvSeq{k}{p} that is smaller by a factor of roughly $\lfloor p/2 \rfloor$. 

We start with the following observation.

\begin{proposition}\label{ftuziu}
    Let $k$ be a positive integer, $G$ a $2k$-edge-connected graph, $H$ a $2k$-edge-connected  subgraph of $G$, and $\vec{G},\vec{H}$ be $k$-arc-strong orientations of $G$ and $H$, respectively. Let $\vec{G}_1$ be the unique orientation of $G$ in which all edges in $E(H)$ have the orientation they have in $\vec{H}$ and all remaining edges have the orientation they have in $\vec{G}$. Then $\vec{G}_1$ is $k$-arc-strong.
\end{proposition}
\begin{proof}
    Let $\emptyset \subsetneq X \subsetneq V(G)$. If $V(H)\subseteq X$ or $V(H)\cap X = \emptyset$, we have $d_{\vec{G}_1}^+(X)= d_{\vec{G}}^+(X)\geq k$. Otherwise, we have $d_{\vec{G}_1}^+(X)\geq d_{\vec{H}}^+(X)\geq k$. In either case, we have $d_{\vec{G}_1}^+(X)\geq k$, so $\vec{G}_1$ is $k$-arc-strong.   
\end{proof}

Our proof of Theorem~\ref{thm:APX_algo} relies on the notion of {\it frames}, which is an analogue of the so-called {\it blocks} for the edge-connectivity. Given an integer $k\geq 1$, a {\it $k$-frame} $F$ of a multigraph $G$ is a subset of vertices of $G$ such that $G[F]$ is $k$-edge-connected and $F$ is inclusion-wise maximal with respect to this property.
Recall that, by definition, a single vertex induces a $k$-edge-connected graph.
We now prove the following easy proposition on $k$-frames.

\begin{proposition}
    \label{prop:frames}
    Let $G$ be a multigraph and $k\geq 1$ be an integer. The $k$-frames $F_1,\dots,F_r$ of $G$ form a partition of $V(G)$. Moreover, contracting the $k$-frames of $G$ yields a multigraph $G'$ with at most $(k-1)(|V(G')|-1)$ edges.
\end{proposition}
\begin{proof}
    By definition, the single-vertex graph is $k$-edge-connected. Therefore, every vertex of $G$ belongs to at least one frame. Moreover, it is straightforward to check that if two frames $F_1,F_2$ intersect, then $F_1\cup F_2$ is also a frame, which by inclusion-wise maximality implies $F_1=F_2$. This shows the first part of the statement.

    We now prove the second part of the statement. 
    Let $G'$ be the multigraph obtained from $G$ by contracting the $k$-frames $F_1,\dots,F_r$ of $G$. 
    We first justify that every subgraph $H'$ of $G'$ on at least two vertices satisfies $\lambda(H')\leq k-1$.
    Let thus $H'$ be an arbitrary subgraph of $G'$ on at least two vertices.
    Let $F^\star=\bigcup_{F\in V(H')} F$.
    Since $G[F^\star]$ contains at least two distinct frames, by inclusion-wise maximality, $F^\star$ is not a frame. Hence, $G[F^\star]$ admits a cut $\delta_{G[F^\star]}(U)$ of size at most $k-1$. 
    If there exists $F\in V(H')$ such that $U\setminus F \neq \emptyset$ and $F\setminus U \neq \emptyset$, then
    \[
        d_{G[F]} (U\cap F) \leq d_{G[F^\star]}(U) \leq k-1,
    \]
    a contradiction to $F$ being a $k$-frame. Therefore, every frame $F\in V(H')$ is either included in $U$ or disjoint from $U$. Moreover, by definition of $\delta_{G[F^\star]}(U)$ being a cut, at least one frame $F\in V(H')$ is included in $U$, and at least one is disjoint from $U$. 
    Let $U' = \{ F\in V(H') : F\subseteq U\}$. It follows that
    \[
        d_{H'}(U') \leq d_{G[F^\star]}(U) \leq k-1.
    \]
    This shows that every subgraph of $G'$ on at least two vertices has edge-connectivity at most $k-1$. The statement then follows from Proposition~\ref{prop:density_non_k_edge_connected}.
\end{proof}

Using Proposition~\ref{prop:frames}, we derive another intermediate result, which contains the main technical difficulty for the proof of Theorem~\ref{thm:APX_algo}, for which we need one more definition. Namely, given a digraph $D$ and two collections $\mathcal{X}$ and $\mathcal{Z}$ of subsets of $V(D)$, we say that $\mathcal{Z}$ {\it dominates} $\mathcal{X}$ if for every $X \in \mathcal{X}$, there exists some $Z \in \mathcal{Z}$ with $X \subseteq Z$. 
\begin{lemma}
    \label{lem:technical_frames}
    Let  $k\geq 1$ be an integer, $D$ a $2k$-edge-connected digraph, $\mathcal{Z}\subseteq 2^{V(D)}$, and $\mathcal{X}\subseteq \binom{V(D)}{2}$ a minimum set such that $\mathcal{Z}$ dominates $\mathcal{X}$ and $\Inv(D;\mathcal{X})$ is $k$-arc-strong. Then, for every $Z \in \mathcal{Z}$, we have $|\mathcal{X}\cap \binom{Z}{2}|\leq (2k-1)(|Z|-1)$.  
\end{lemma}
\begin{proof}
    Let $Z \in \mathcal{Z}$, let  $G$ be the underlying graph of $D[Z]$, and let $\mathcal{F}$ be the unique partition of $V(G)$ into $2k$-frames. Observe that $\mathcal{F}$ is well-defined by \Cref{prop:frames}. Now for every $F \in \mathcal{F}$, we set $\mathcal{X}_F=\mathcal{X}\cap \binom{F}{2}$ and $\mathcal{X}_0=(\mathcal{X}\cap \binom{Z}{2})\setminus \bigcup_{F \in \mathcal{F}}\mathcal{X}_F$. It follows directly from the minimality of $\mathcal{X}$ that for every $X \in \mathcal{X}_0$, there exists an edge in $E(G)$ linking the two vertices in $X$. By \Cref{prop:frames}, this yields $|\mathcal{X}_0|\leq (2k-1)(|\mathcal{F}|-1)$. The next result also bounds the number of elements in $\mathcal{X}_F$ for $F \in \mathcal{F}$.
    \begin{claim}\label{tfuziguo}
        Let $F \in \mathcal{F}$. Then $|\mathcal{X}_F|\leq k(|F|-1)$.
    \end{claim}
    \begin{proofclaim}
        Let $H$ be a minimally $2k$-edge-connected subgraph of $G[F]$ with $V(H) =F$, and let $Q$ be the set of simple edges of $H$.
        Further, let $\vec{H}$ be a $k$-arc-strong orientation of $H$ without parallel arcs, which exists by \Cref{thm:nashwilliams}. 
        
        First let $\mathcal{Y}_1$ consist of all pairs of vertices linked by some $q \in Q$ that has a distinct orientation in $D$ and $\vec{H}$. We obtain that $\Inv(D[F];\mathcal{Y}_1)$ is $k$-arc-strong. It hence follows by Proposition \ref{ftuziu} that $\Inv(D;(\mathcal{X}\setminus \mathcal{X}_F)\cup \mathcal{Y}_1)$ is $k$-arc-strong. Moreover, as $Y \subseteq Z$ for every $Y \in \mathcal{Y}_1$, we have that $\mathcal{Z}$ dominates $(\mathcal{X}\setminus \mathcal{X}_F)\cup \mathcal{Y}_1$. Hence, the minimality of $\mathcal{X}$ yields $|\mathcal{X}_F|\leq |\mathcal{Y}_1|$. Now let $\mathcal{Y}_2$ consist of all pairs of vertices linked by some $q \in Q$ that has the same orientation in $D$ and $\vec{H}$. As $\Inv(D[F];\mathcal{Y}_2)$ contains an orientation of $H$ that is obtained by reversing all arcs of $\vec{H}$, we obtain that $\Inv(D[F];\mathcal{Y}_2)$ is $k$-arc-strong. It follows from \Cref{ftuziu} that $\Inv(D;(\mathcal{X}\setminus \mathcal{X}_F)\cup \mathcal{Y}_2)$ is $k$-arc-strong. Moreover, as $Y \subseteq Z$ for every $Y \in \mathcal{Y}_2$, we have that $\mathcal{Z}$ dominates $(\mathcal{X}\setminus \mathcal{X}_F)\cup \mathcal{Y}_2$. Hence, the minimality of $\mathcal{X}$ yields $|\mathcal{X}_F|\leq |\mathcal{Y}_2|$. As $\mathcal{Y}_1$ and $\mathcal{Y}_2$ are disjoint, and since $H$ is minimally $2k$-edge-connected, by \Cref{thm:mader} we obtain
        \[
            |\mathcal{X}_F|\leq \min \{|\mathcal{Y}_1|,|\mathcal{Y}_2|\}\leq \frac{1}{2}(|\mathcal{Y}_1|+|\mathcal{Y}_2|)= \frac{1}{2}|Q|\leq \frac{1}{2}|E(H)| \leq k(|F|-1).\qedhere
        \]
    \end{proofclaim}
    By \Cref{tfuziguo} and $|\mathcal{X}_0|\leq (2k-1)(|\mathcal{F}|-1)$, we obtain
    \begin{align*}
        |\mathcal{X}\cap \binom{Z}{2}|&=|\mathcal{X}_0|+\sum_{F \in \mathcal{F}}|\mathcal{X}_F|\\&\leq (2k-1)(|\mathcal{F}|-1)+\sum_{F \in \mathcal{F}}k(|F|-1)\\
        &= (2k-1)(|\mathcal{F}|-1)+k(|Z|-|\mathcal{F}|)\\
        &\leq (2k-1)(|Z| -1).\qedhere
    \end{align*}
\end{proof}

We make use of Lemma~\ref{lem:technical_frames} in the following form, which is the key lemma for the proof of Theorem~\ref{thm:APX_algo}.

\begin{lemma}
    \label{lem:bounding_invk2_invkp}
   Let  $k\geq 1$ and $p\geq 2$ be integers and $D$ a $2k$-edge-connected digraph. Then 
   \[
   \inv{k}{2}(D)\leq (2k-1)(p-1)\inv{k}{p}(D).
   \]
\end{lemma}
\begin{proof}
    Let $\mathcal{Z}$ be a \InvSeq{k}{p} of $D$ of size $\inv{k}{p}(D)$. It follows that there exists a minimum set $\mathcal{X}\subseteq \binom{V(D)}{2}$ such that $\mathcal{Z}$ dominates $\mathcal{X}$ and $\Inv(D;\mathcal{X})$ is $k$-arc-strong. By \Cref{lem:technical_frames}, we have $|\mathcal{X}\cap \binom{Z}{2}|\leq (2k-1)(|Z|-1)\leq (2k-1)(p-1)$ for every $Z\in \Zcal$. As $\mathcal{Z}$ dominates $\mathcal{X}$, we obtain that 
    \[
    \inv{k}{2}(D)\leq |\mathcal{X}|\leq \sum_{Z\in \mathcal{Z}} |\mathcal{X}\cap \textstyle \binom{Z}{2}|\leq (2k-1)(p-1)\inv{k}{p}(D),
    \]
    as desired.
\end{proof}

In the main proof of \Cref{thm:APX_algo}, we will show that $\inv{k}{p}$ is by a factor of roughly $\lfloor \frac{p}{2}\rfloor$ smaller than $\inv{k}{2}$ when $\inv{k}{2}$ is sufficiently large. The idea is that, given a set of \leqp{2}-inversions, a single \leqp{p}-inversion can replace a collection of these \leqp{2}-inversions. In order to prove the existence of an appropriate collection of \leqp{2}-inversions, we  need the following well-known general form of Ramsey's Theorem~\cite{ramsey1930} (see for instance~\cite[Theorem~9.1.3]{diestel2017}).

\begin{theorem}
    \label{thm:ramsey}
    For all integers $c,r\geq 1$ there exists an integer $n$ such that every $c$-coloring of $E(K_n)$ yields a monochromatic clique on $r$ vertices.
\end{theorem}

We are now ready to derive Theorem~\ref{thm:APX_algo}, that we first recall here for convenience.

\thmAPXalgo*
 
\begin{proof}[Proof of Theorem~\ref{thm:APX_algo}]
    Given integers $n_1,n_2,n_3$, we denote by $R(n_1,n_2,n_3)$ the least integer $n$ such that every coloring of the edges of $K_n$ with $\{1,2,3\}$ yields a clique of size $n_i$ colored $i$ for some $i\in \{1,2,3\}$.
    The existence of $R(n_1,n_2,n_3)$ is guaranteed by Theorem~\ref{thm:ramsey}. We prove the statement for 
    \[
    c_{k,p}=R(\lfloor p/2\rfloor, 4k,8k)-1.
    \]
    
    Let us fix an arbitrary $2k$-edge-connected digraph $D=(V,A)$ with underlying graph $G$. 
    Given a graph $H$, an {\it edge set} in $H$ is a set of two vertices in $V(H)$ that are linked by exactly edge in $E(H)$. 
    We say that two distinct edge sets $\{x,y\}$ and $\{u,v\}$ of $H$ are {\it independent} if $H[\{x,y,u,v\}]$ contains precisely the two edges $xy$ and $uv$, {\it dependent} otherwise. Note that two independent edge sets $\{x,y\}$ and $\{u,v\}$ may intersect.
    We consider the following algorithm applied to $D$.
    
    \begin{enumerate}
        \item Compute a \InvSeq{k}{2} $\Zcal_2$ of $D$ of size $\inv{k}{2}(D)$.
        \label{enumitem:algo_apx_1}
        
        \item Compute a minimally $k$-arc-strong subdigraph $D'$ of $\Inv(D;\Zcal_2)$ and let $G'$ be the underlying graph of $D'$. 
        
        \item Build $\Zcal$ from $\Zcal_2$ as follows, starting from $\Zcal=\emptyset$. As long as $\Zcal_2$ contains a set $I$ of $\lfloor p/2\rfloor$ edge sets that are pairwise independent in $G'$, we choose any such set $I$ and let 
        \begin{itemize}
            \item $\Zcal \gets \Zcal \cup \big\{\bigcup_{e\in I} e\big\}$, and
            \item $\Zcal_2 \gets \Zcal_2 \setminus I$.
        \end{itemize}
        \item Output $\Zcal \cup \Zcal_2$.
    \end{enumerate}

    Note that Step~\ref{enumitem:algo_apx_1} of the algorithm can be done in polynomial time by Theorem~\ref{thm:frank}, and that all other steps can be done in polynomial time trivially as $p$ is fixed.
    From now on, let $\Xcal_2$ denote the value of $\Zcal_2$ as computed at Step~\ref{enumitem:algo_apx_1}, and let $\Xcal$ denote the output of the algorithm. 
    Note that $\Xcal$ is a \InvSeq{k}{p} of $D$ since, by construction, $|X|\leq p$ holds for every $X \in \mathcal{X}$ and $\Inv(D;\Xcal)$ contains $D'$ as a subdigraph. 
    In what remains, we prove the following three inequalities, which together imply the statement:
    \begin{align}
        \label{eq:X_by_X2}
        |\Xcal| &\leq  \frac{|\Xcal_2|}{\lfloor p/2\rfloor} + c_{k,p},\\[1em]
        \label{eq:X2_by_Y_bis}
        |\Xcal_2| &\leq \binom{p}{2} \cdot \inv{k}{p}(D),\\[1em]
        \label{eq:X2_by_Y}
        |\Xcal_2| &\leq (2k-1)(p-1) \cdot \inv{k}{p}(D).
    \end{align}
    First observe that Inequality~\eqref{eq:X2_by_Y} follows from Lemma~\ref{lem:bounding_invk2_invkp}.
    Let $\Ycal$ be a \InvSeq{k}{p} of~$D$ of size $\inv{k}{p}(D)$. 
    Observe that a \InvSeq{k}{2} of~$D$ can be obtained by taking all the edge sets $Z$ of $G$ such that the number of sets $Y\in \mathcal{Y}$ containing $Z$ is odd. 
    Clearly, the resulting set has size at most $\binom{p}{2}\cdot |\Ycal| =  \binom{p}{2}\cdot\inv{k}{p}(D)$, and the Inequality~\eqref{eq:X2_by_Y_bis} follows from the optimality of $\Xcal_2$.
    It thus remains to prove~\eqref{eq:X_by_X2}. First observe that every $X \in \Xcal_2$, we have that $X$ is an edge set in $G'$ by the minimality of $\mathcal{X}_2$. Hence, by construction, there is a partition $(\Xcal',\Xcal'')$ of $\Xcal$ such that:
    \begin{itemize}[itemsep=0pt]
        \item for each $X\in \Xcal'$, $E(G'[X])$ has size $\lfloor p/2\rfloor$;
        \item $\Xcal'' \subseteq \Xcal_2$ and any collection of $\lfloor p/2\rfloor$ sets of $\Xcal''$ contains two dependent edge sets; and
        \item $(E(G'[X]))_{X\in \Xcal'}$ partitions $\Xcal_2 \setminus \Xcal''$.
    \end{itemize}
    In particular, it follows that 
    \[
        |\Xcal| \leq \frac{1}{\lfloor p/2\rfloor}\Big(|\Xcal_2| - |\Xcal''|\Big) + |\Xcal''| \leq \frac{|\Xcal_2|}{\lfloor p/2\rfloor}+ |\Xcal''|.
    \]
    Hence, we only have to prove $|\Xcal''| \leq c_{k,p}$.
    Assume for a contradiction that $|\Xcal''| \geq c_{k,p} + 1$. Consider the auxiliary edge-colored complete graph $K$ whose vertex set is $\Xcal''$, and each edge linking two sets $X_1$ and $X_2$ is colored:
    \begin{itemize}[itemsep=0pt]
        \item $1$ if $X_1$ and $X_2$ are independent in $G'$,
        \item $2$ if $X_1$ and $X_2$ are dependent in $G'$ and intersect,
        \item $3$ otherwise, that is $X_1$ and $X_2$ are dependent in $G'$ and disjoint.
    \end{itemize}
    By construction, every set of $\lfloor p/2\rfloor$ edge sets of $\Xcal''$ contains two dependent edges. Therefore, by the choice of $c_{k,p}$, one of the two following cases holds.
    \begin{description}
        \item[Case 1:] {\it $\Xcal''$ contains a set $L$ of $4k$ edge sets that are pairwise intersecting and dependent in $G'$.}

        Recall that, by minimality of $\Xcal_2$, two intersecting edge sets of $\Xcal_2$ share exactly one vertex.
        Let $V_L$ be the union of all sets in $L$. 
        Since $4k\geq 4$, and because the edge sets in $L$ pairwise intersect, there exists a vertex $u$  contained in all sets in $L$, and in particular $|V_L|=4k+1$. Now, for every pair of distinct sets $\{u,v\},\{u,w\} \in L$, since these sets are dependent in $G'$, the graph $G'[\{u,v,w\}]$ is isomorphic to $K_3$. It follows that $G'[V_L]$ is isomorphic to $K_{4k+1}$. Hence, $D'[V_L]$ has order $4k+1$ and $\binom{4k+1}{2} > 8k^2$ arcs, and it is a subdigraph of a minimally $k$-arc-strong digraph, a contradiction to Theorem~\ref{thm:dalmazzo}.
        
        \item[Case 2:] {\it $\Xcal''$ contains a set $L$ of $8k$ edge sets that are pairwise disjoint and dependent in $G'$.}

       Let $V_L$ be the union of all sets in $L$, so $|V_L| = 16k$. For every pair of distinct edge sets $\{u,v\},\{x,y\} \in L$, since these edge sets are dependent, the graph $G'[\{u,v,x,y\}]$ contains at least one edge between $\{u,v\}$ and $\{x,y\}$. It follows that $D'[V_L]$ has order $16k$ and at least  $8k+\binom{8k}{2} > 2k(16k-1)$ arcs, a contradiction to Theorem~\ref{thm:dalmazzo}.
    \end{description}

    This shows~\eqref{eq:X_by_X2} and concludes the proof of Theorem~\ref{thm:APX_algo}.
\end{proof}

\section{Parameterized Complexity}
\label{sec:W1_hardness}

This section is devoted to the proofs of \Cref{thm:W1_hardness_invkp_by_p,thm:W1_hardness_invkp_by_ell}, which follow from the more general statements of \Cref{thm:W1_hardness_invkp_by_p_ETH,thm:W1_hardness_invkp_by_ell_ETH} below.

Both hardness results are obtained by reducing from the following problem.

\defproblem{Partitioned Subgraph Isomorphism (PSI)}{
A graph $G$, a partition $(V_1,\ldots,V_k)$ of $V(G)$, and a graph $H$ with $V(H)=[k]$.
}{
 Is there $(v_1,\ldots,v_k) \in V_1 \times \ldots \times V_k$ such that, for every edge $ij \in E(H)$, we have $v_iv_j \in E(G)$?
}

In the following, when given an instance $(G,(V_1,\ldots,V_k),H)$ of PSI, we use $k'$ for $|E(H)|$.
We make use of the famous following hardness result for PSI, which follows directly from the well-known $W[1]$-hardness of Multicolored Clique, see \cite{cyganPA2015}.

\begin{theorem}[\cite{marxTC6}]
    \label{thm:hardness_PSI}
    PSI is W[1]-hard with respect to $k'=|E(H)|$. 
\end{theorem}
We now present the formal decision problems and the more general statements of \Cref{thm:W1_hardness_invkp_by_p,thm:W1_hardness_invkp_by_ell}, which we prove in this section.

\defproblem{$k$-Single-Inversion ($k$-SI)}{An oriented graph $D$ and an integer $p$.}{Is there a set $X\subseteq V(D)$, $|X|\leq p$, such that $\Inv(D,X)$ is $k$-arc-strong?}

\begin{restatable}{theorem}{Whardnessbyp}
    \label{thm:W1_hardness_invkp_by_p_ETH}
    For every $k \geq 2$, {\sc $k$-SI} is W[1]-hard with respect to $p$. 
    Moreover, if for some absolute constant $c$ with $0<c\leq 1$, there exists some $\alpha>0$ such that PSI cannot be solved in $f(k')n^{\alpha (k')^c}$, then there exists some $\alpha_0>0$ such that  {\sc $k$-SI} cannot be solved in $f(p)n^{\alpha_0 p^c}$.
\end{restatable}

When parameterized by $\ell$ instead of $p$, the parameterized complexity of the problem remains open for oriented graphs and digraphs. However, as a partial result, we obtain that it is hard for multidigraphs, even in the first nontrivial case. 

\defproblem{$(2,2)$-Inversion}{A multidigraph $D$ and an integer $\ell\geq 1$.}{Is there a set of \eqp{2}-inversions $\Xcal$ such that $|\Xcal|\leq \ell$ and $\Inv(D;\Xcal)$ is $2$-arc-strong?}

\begin{restatable}{theorem}{Whardbyell}
    \label{thm:W1_hardness_invkp_by_ell_ETH}
    {\sc $(2,2)$-Inversion} is W[1]-hard with respect to $\ell$.
    Moreover, if for some absolute constant $c$ with $0<c\leq 1$, there exists some $\alpha>0$ such that PSI cannot be solved in $f(k')n^{\alpha (k')^c}$, then there exists some $\alpha_0>0$ such that  {\sc $(2,2)$-Inversion} cannot be solved in $f(\ell)n^{\alpha_0 \ell^c}$. 
\end{restatable}
We wish to remark that in an earlier version of this article, we included hardness results based on the Exponential Time Hypothesis (ETH). As AI-based disproofs of the ETH have recently been announced \cite{OAI:The-Exponential-Time-Hypothesis-is-False-October-25-2026,anthropic2026eth}, this approach does not seem to be appropriate anymore at this point. Due to a result of Marx \cite{marxTC6}, the ETH implied that PSI cannot be solved in $f(k')n^{ak'/\log (k')}$ for some fixed constant $\alpha$. While, to the best of our knowledge, no faster algorithm for PSI is currently known, it is not clear what kind of implications a good weaker complexity hypothesis replacing ETH would have for PSI. We hence decided to use the current forms for the theorem statements, which are flexible in case of further developments.

We prove Theorems~\ref{thm:W1_hardness_invkp_by_p_ETH} and~\ref{thm:W1_hardness_invkp_by_ell_ETH} in Sections~\ref{sec:Whardnessbyp} and~\ref{sec:Whardbyell} respectively. In order to do so, we need a slightly adapted version of PSI.

Formally, we say that an instance $(G,(V_1,\ldots,V_k),H)$ of PSI is {\it normalised} if:
\begin{itemize}[itemsep=0pt]
    \item $k \geq 2$,
    \item every vertex of $H$ is incident to at least two edges,
    \item $|V_i|\geq 3$ holds for every $i\in [k]$,
    \item $d_{G}(V_i,V_j)\geq 2$ holds for every $ij \in E(H)$, and
    \item for all $uv \in E(G)$, we have $i_u \neq i_v$ and $i_ui_v \in E(H)$, where $i_u$ (resp. $i_v$) is the unique integer in $[k]$ such that $u \in V_{i_u}$ (resp. $v\in V_{i_v}$).
\end{itemize}

We use {\sc normalised  Partitioned Subgraph Isomorphism (NPSI)} for the restriction of PSI to normalised instances. It is straightforward to check that PSI reduces to NPSI.

\begin{corollary}
    \label{tufizgouhjnik}
    NPSI is W[1]-hard with respect to $k'= |E(H)|$. Moreover, if for some absolute constant $c>0$, there exists some $\alpha >0$ such that PSI cannot be solved in $f(k')n^{\alpha (k')^c}$, then  NPSI cannot be solved in $f(k')n^{\alpha_1 (k')^c}$ for some $\alpha_1>0$.
\end{corollary}

\subsection{W[1]-hardness when parameterised by \texorpdfstring{$p$}{p}}
\label{sec:Whardnessbyp}

The technical difficulty for the proof of Theorem~\ref{thm:W1_hardness_invkp_by_p_ETH} is contained in the following lemma.

\begin{lemma}
    \label{lem:W1_hardness_invkp_by_p}
    The following holds for every fixed $k\geq 2$.
    Given an instance $(G,(V_1,\dots,V_r),H)$ of NPSI, in polynomial time one can compute an equivalent instance $(D,p)$ of $k$-SI such that $|V(D)| \leq (3k+1)\cdot(|V(G)| + |E(G)|+1)$ and $p\leq 3 \cdot |E(H)|$.
\end{lemma}
\begin{proof}
    Let $(G,(V_1,\dots,V_r), H)$ be an instance of NPSI.
    Let $T$ be an arbitrary $k$-arc-strong tournament of order $3k$. 
    We build an oriented graph $D$ from $(G,(V_1,\dots,V_r), H)$ as follows, see Figure~\ref{fig:W1_hardness_invkp_by_p} for an illustration. We start from the vertex set $V(G)$, and we add a disjoint copy $\tilde{T}$ of $T$. 
    Next, for every integer $i\in [r]$, we add a new copy $T_i$ of $T$, and we put $k$ arcs in both directions between $V(T_i)$ and $V(\tilde{T})$ without creating parallel arcs or  digons. We further add a new vertex $x_i$ with $k$ arcs from $x_i$ to $V(T_i)$, and $k-1$ arcs from  $V(T_i)$ to $x_i$,  again without creating parallel arcs or digons. Finally, for every vertex $v\in V_i$, we add $k$ arcs in both directions between $v$ and $V(T_i)$ without creating parallel arcs or digons, and an arc from $x_i$ to $v$.
    
    For every pair of integers $1\leq i <j\leq r$ such that $ij\in E(H)$, we add a new copy $T_{i,j}$ of $T$, together with $k$ arcs from $V(T_{i,j})$ to $V(\tilde{T})$, and $k-2$ arcs from $V(\tilde{T})$ to $V(T_{i,j})$ without creating parallel arcs or digons. Furthermore, for every pair of vertices $(u,v)\in V_i\times V_j$ which are adjacent in $G$, we further add a vertex $z_{u,v}$ with the arcs $z_{u,v}u,z_{u,v}v$, and $k$ arcs in both directions between $z_{u,v}$ and $V(T_{i,j})$ without creating parallel arcs or digons. We denote by $Z_{i,j}$ the set of such vertices.

    \begin{figure}[ht]
    \centering
	\begin{tikzpicture}[
    region/.style={draw=black, rounded corners, semithick,inner sep=3pt, outer sep=2pt},
    Hvertex/.style = {rectangle, draw=gray, very thick, fill=gray!5, minimum size=22pt, inner sep=0pt, outer sep=2pt}]
        \node[svertex, label=left:$s$] (s) at (0,0) {};
        \node[region, dashed, label=above:$V_1$ ,fit=(s)] (Y) {};
        
        \node[svertex, label=below:$t$] (t) at (1.5,1.5) {};
        \node[svertex, label=above:$u$] (u) at (1.5,0) {};
        \node[region, dashed, label=above:$V_2$, fit=(t)(u)] (Y) {};
        
        \node[svertex, label=below:$v$] (v) at (3,1.5) {};
        \node[svertex, label=above:$w$] (w) at (3,0) {};
        \node[region, dashed, label=above:$V_3$, fit=(v)(w)] (Y) {};

        \draw[tedge] (s) -- (t);
        \draw[tedge] (s) -- (u);
        \draw[tedge] (v) -- (t);
        \draw[tedge] (w) -- (u);
        \draw[tedge, bend left=25] (w) to (s);

        \begin{scope}[xshift=7cm, yshift=0.5cm]
            \node[svertex, label=below:$s$] (s) at (0,0) {};
            \node[svertex, label=below:$t$] (t) at (2.5,1.5) {};
            \node[svertex, label=below:$u$] (u) at (2.5,0) {};
            \node[svertex, label=below:$v$] (v) at (5,1.5) {};
            \node[svertex, label=below:$w$] (w) at (5,0) {};

            \node[Hvertex] (Htilde1) at (0,3) {\footnotesize $T_1$};
            \node[svertex, label=left:$x_1$] (x1) at (-1.5,1.5) {};
            \draw[tarc, bend left=15, g-blue] (Htilde1) to (x1);
            \draw[tarc, bend left=15, red] (x1) to (Htilde1);
            \draw[digon,red] (Htilde1) to (s);
            \draw[tarc] (x1) to (s);
            
            \node[Hvertex] (Htilde2) at (2.5,3) {\footnotesize $T_2$};
            \node[svertex, label=right:$x_2$] (x2) at (3.75,2.5) {};
            \draw[tarc, bend left=15, g-blue] (Htilde2) to (x2);
            \draw[tarc, bend left=15, red] (x2) to (Htilde2);
            
            \node[Hvertex] (Htilde3) at (5,3) {\footnotesize $T_3$};
            \node[svertex, label=right:$x_3$] (x3) at (6.5,1.5) {};
            \draw[tarc, bend left=15, g-blue] (Htilde3) to (x3);
            \draw[tarc, bend left=15, red] (x3) to (Htilde3);
            
            \draw[tarc] (x3) to (v);
            \draw[tarc] (x3) to (w);

            \node[svertex, label=above:$z_{s,t}$] (zst) at (1.25,0.75) {};
            \draw[tarc] (zst) to (s);
            \draw[tarc] (zst) to (t);
            
            \node[svertex, label={[label distance=-5pt]45:$z_{s,u}$}] (zsu) at (1.25,0) {};
            \draw[tarc] (zsu) to (s);
            \draw[tarc] (zsu) to (u);
            \node[Hvertex] (H12) at (0,-2.8) {\footnotesize $T_{1,2}$};

            \node[svertex, label=above:$z_{t,v}$] (ztv)at (3.75,0.75) {};
            \draw[tarc] (ztv) to (t);
            \draw[tarc] (ztv) to (v);
            \node[svertex, label={[label distance=-5pt]135:$z_{u,w}$}] (zuw)at (3.75,0) {};
            \draw[tarc] (zuw) to (u);
            \draw[tarc] (zuw) to (w);
            \node[Hvertex] (H23) at (5,-2.8) {\footnotesize $T_{2,3}$};
            
            \node[svertex, label=above:$z_{s,w}$] (zsw) at (2.5,-1.5) {};
            \draw[tarc] (zsw) to[out=180, in=-45] (s);
            \draw[tarc] (zsw) to[out=0, in=-135] (w);
            \node[Hvertex] (H13) at (2.5,-2.8) {\footnotesize $T_{1,3}$};
            \draw[digon,red] (H13) -- (zsw);

            \node[Hvertex] (Htilde4) at (2.5,-4.3) {\footnotesize $\tilde{T}$};
            \node[Hvertex] (Htilde5) at (2.5,4.5) {\footnotesize $\tilde{T}$};

            \foreach \i in {12, 13, 23}{
                \draw[tarc,red, bend left=15] (H\i) to (Htilde4);
                \draw[tarc,g-green, bend left=15] (Htilde4) to (H\i);
            }
            \foreach \i in {1, 2, 3}{
                \draw[digon,red] (Htilde\i) to (Htilde5);
            }
            
            \draw[tdigon,white] (H23.120) -- (zuw);
            \draw[tdigon,white] (H23.90) -- (ztv);
            \draw[digon,red] (H23.120) -- (zuw);
            \draw[digon,red] (H23.90) -- (ztv);
            
            \draw[tdigon,white] (H12.60) -- (zsu);
            \draw[tdigon,white] (H12.90) -- (zst);
            \draw[digon,red] (H12.60) -- (zsu);
            \draw[digon,red] (H12.90) -- (zst);
            
            \draw[digon,red] (Htilde2) -- (t);
            \draw[tdigon,white] (Htilde2) to[out=-120, in=120] (u);
            \draw[digon,red] (Htilde2) to[out=-120, in=120] (u);
            
            \draw[digon,red] (Htilde3) -- (v);
            \draw[tdigon,white] (Htilde3) to[out=-120, in=120] (w);
            \draw[digon,red] (Htilde3) to[out=-120, in=120] (w);
            
            \draw[tarc] (x2) to (t);
            \draw[vtarc,white] (x2) to (u);
            \draw[tarc] (x2) to (u);
        \end{scope}
	\end{tikzpicture}
    \caption{An illustration of the construction of $D$ (right) from a graph $G$ (left).  Here, $r=3$ and $H$ is the complete graph $K_3$. 
    Note that, for the sake of better readability, we chose here a graph $G$ that is not an instance of NPSI.
    Each grey square illustrates a copy of $T$. For better readability, we put two copies of $\tilde{T}$, but these are the same one copy of $T$. Red, blue, and green arcs illustrate that there exist respectively $k$, $k-1$, and $k-2$ such arcs.}
    \label{fig:W1_hardness_invkp_by_p}
    \end{figure}
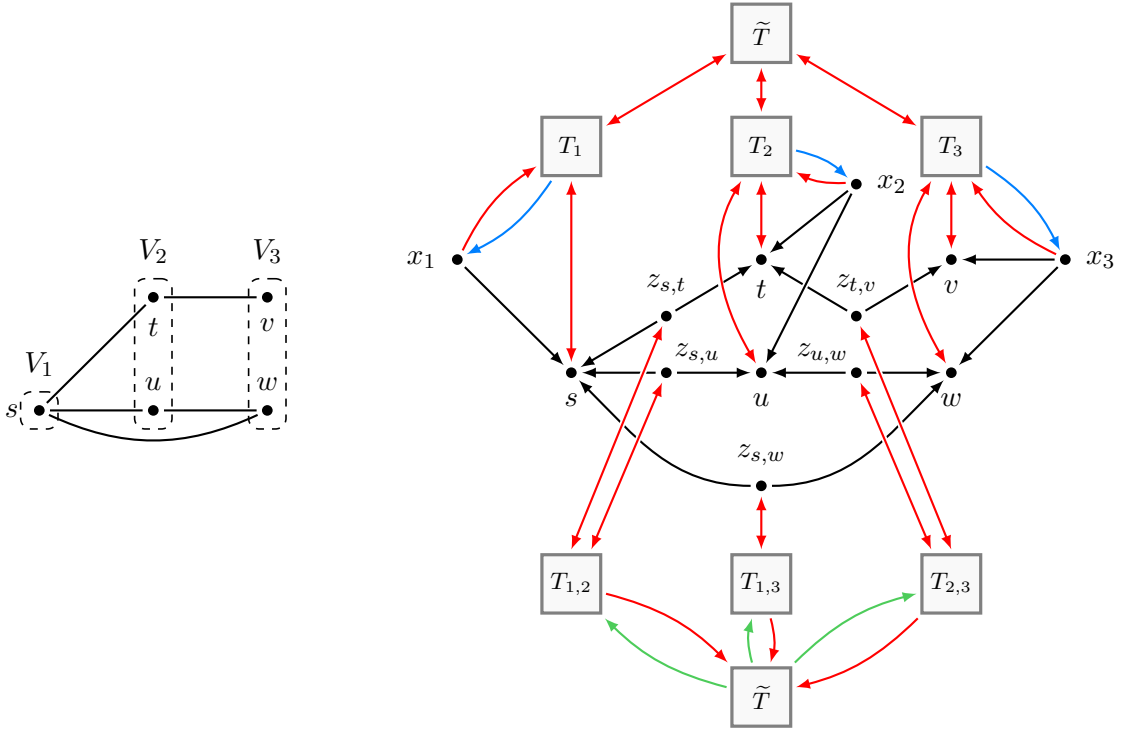

    Let $p=2r+|E(H)|$. Note that $r\leq |E(H)|$ as every vertex of $H$ has at least two incident edges. Hence $p\leq 3|E(H)|$.
    In what remains, we prove that $(G,(V_1,\dots,V_r),H)$ is a yes-instance if and only if $\inv{k}{p}(D) \leq 1$, hence showing the lemma. 
    
    \begin{claim}
        If $(G,(V_1,\dots,V_r),H)$ is a yes-instance of NPSI, then $\inv{k}{p}(D) \leq 1$.
    \end{claim}
    \begin{proofclaim}
        Let $U=(u_1,\dots,u_r) \in V_1\times \ldots \times V_r$ be such that, for every $ij\in E(H)$, $u_iu_j \in E(G)$. Let 
        \[
            X= U \cup \{x_i:1\leq i\leq r\} \cup  \{z_{u_i,u_j}  : 1\leq i<j\leq r \text{~and~} ij\in E(H)\}.
        \]
        Clearly $|X|=2r + |E(H)|=p$, and we claim that $D'=\Inv(D,X)$ is $k$-arc-strong. To show this, let $\emptyset \subsetneq S \subsetneq V(D)$ be an arbitrary subset of $V(D')$, and let us show that $d^+_{D'}(S), d^-_{D'}(S) \geq k$.
        
        Since $X$ does not contain any vertex in $V(\tilde{T})$, we can assume that $V(\tilde{T})$ is either included in $S$ or disjoint from $S$. By a similar argument, we may assume that $V(T_i)$ is either contained in $S$ or disjoint from $S$ for any $i \in [r]$ and that $V(T_{i,j})$ is either contained in $S$ or disjoint from $S$ for any $i,j \in [r]$ with $1\leq i<j \leq r$ and $ij \in E(H)$. By symmetry of $S$ and $V(D')\setminus S$, we assume that $V(\tilde{T}) \cap S = \emptyset$. Let $v$ be an arbitrary vertex in $S$.
    
       First assume that $v\in V(T_i)$ for some $i\in [r]$. Then, by the remark above, $V(T_i) \subseteq S$, and by construction, we have $\min\{d_{D'}^+(S),d_{D'}^-(S)\}\geq k$ as $A(D')$ contains $k$ arcs in both directions between $V(\tilde{T})$ and $V(T_i)$. Henceforth, we thus assume that $\bigcup_{i\in [r]}  V(T_i) \cap S = \emptyset$.
        Similarly, for every $i\in [r]$, if $v\in V_i$, then there exist $k$ arcs in both directions between $v$ and $V(T_i)$. Therefore, we further assume that $V(G)\cap S =\emptyset$.
        
        Now assume that $v=x_i$ for some $i\in [r]$. By construction, $k$ arcs from $x_i$ to $V(T_i)$ are contained in $\delta_{D'}^+(S)$, and there exist $k-1$ arcs in $\delta_{D'}^-(S)$ from $V(T_i)$ to $x_i$.  Moreover, both $x_i$ and $u_i$ belong to $X$, hence $\delta_{D'}^-(S)$ also contains an arc from $u_i$ to $x_i$. Hence we indeed have $k$ arcs in both directions.

      Finally assume that, for some $1\leq i<j\leq r$ with $ij\in E(H)$, we have $v\in V(T_{i,j}) \cup Z_{i,j}$. Note that, for every vertex $z\in Z_{i,j}$, there exist $k$ arcs in both directions between $z$ and and $V(T_{i,j})$ in $D'$. Therefore, we may assume that 
        \[
            V(T_{i,j}) \cup Z_{i,j}  \subseteq S.
        \]
        There exist $k$ arcs from $V(T_{i,j})$ to $V(\tilde{T})$ that are contained in $\delta_{D'}^+(S)$. Further, there exist $k-2$ arcs from $V(\tilde{T})$ to $V(T_{i,j})$ in  $\delta_{D'}^-(S)$ and by the choice of $U$ there exist two arcs from $\{u_i,u_j\}$ to $z_{u_i,u_j}$ in  $\delta_{D'}^-(S)$. The claim follows.
    \end{proofclaim}

    \begin{claim}
        If $\inv{k}{p}(D) \leq 1$ then $(G,(V_1,\dots,V_r),H)$ is a yes-instance of NPSI.
    \end{claim}
    \begin{proofclaim}
        Let $X$ be a set of size at most $p=2r+|E(H)|$ such that $D'$ is $k$-arc-strong, where $D'=\Inv(D,X)$. Note that, for every $i\in [r]$, $X$ contains $x_i$ and a vertex of $V_i \cup V(T_i)$, for otherwise $d^-_{D'}(x_i)=k-1$.
        Moreover, for every pair of integers $1\leq i<j\leq r$ with $ij\in E(H)$, $X$ contains at least one vertex of $V(T_{i,j})\cup Z_{i,j}$ and two vertices in $N^+_D((V(T_{i,j})\cup Z_{i,j})\cap X) \setminus (V(T_{i,j})\cup Z_{i,j})$, for otherwise 
        \[
            d_{D'}^-(V(T_{i,j}) \cup Z_{i,j}) \leq k-1.
        \]
        By construction, and since $|X|\leq 2r+|E(H)|$, this is possible only if
        \begin{itemize}[itemsep=0pt]
            \item $\{x_1,\dots,x_r\} \subseteq X$;
            \item for every $i\in [r]$, $X\cap V(T_i) = \emptyset$;
            \item for every $i\in [r]$, $|X\cap V_i| = 1$;
            \item $V(\tilde{T}) \cap X=\emptyset$;
            \item for every $1\leq i<j\leq r$ with $ij\in E(H)$, $X\cap V(T_{i,j}) = \emptyset$;
            \item for every $1\leq i<j\leq r$ with $ij\in E(H)$, $|X\cap Z_{i,j}| = 1$;
            \item the unique vertex in $X\cap Z_{i,j}$ has an out-neighbor in $X\cap V_i$ and an out-neighbor in $X\cap V_j$.
        \end{itemize}
        The claim follows.
    \end{proofclaim}
\end{proof}

We are now able to derive Theorem~\ref{thm:W1_hardness_invkp_by_p}, that we restate here in a stronger and more technical form.

\Whardnessbyp*

\begin{proof}
  Let $(G,(V_1,\ldots,V_r),H)$ be an instance of NPSI. Then, by Lemma~\ref{lem:W1_hardness_invkp_by_p}, in polynomial time, one can compute an equivalent instance $(D,p)$ of $k$-SI such that $|V(D)| \leq (3k+1)\cdot(|V(G)| + |E(G)|+1)$ and $p\leq 3 \cdot |E(H)|$. It follows directly from \Cref{tufizgouhjnik} and the fact that that $|V(D)|$ is polynomial in $|V(G)|$ and $p$ is bounded by a computable function of $|E(H)|$ that {\sc $k$-Single Inversion} is W[1]-hard with respect to $p$.
    
    For the second part, fix some absolute constant $c$ with $0<c\leq 1$ for which there exists some $\alpha>0$ such that PSI cannot be solved in $f(k')n^{\alpha (k')^c}$. By Corollary~\ref{tufizgouhjnik}, there exists some $\alpha_1>0$ such that NPSI cannot be solved in $f(k')n^{\alpha_1 (k')^c}$. Let $\alpha_0=\frac{1}{9}\alpha_1$. Suppose that {\sc $k$-Single Inversion} can be solved in time $f(p)n^{\alpha_0 p^c}$. 
    First, as $p \leq 3 |E(H)|$ and $|V(D)|\leq (3k+1)\cdot(|V(G)| + |E(G)|+1)\leq 12|V(G)|^3$, it follows that
    $(G,(V_1,\ldots,V_r),H)$ can be solved in time 
    \[
    f(|E(H)|)\cdot |V(G)|^{3\alpha_0 p^c}.
    \]
    Next,  since $p \leq 3 |E(H)|$ and $c \leq 1$, we obtain that 
    \[
    3\alpha p^c\leq 9\alpha |E(H)|^c= \alpha_1 |E(H)|^c.
    \]
    It follows that $(G,(V_1,\ldots,V_r),H)$ can be solved $f(|E(H)|)\cdot |V(G)|^{\alpha_1|E(H)|^c}$, contradicting our assumption.
\end{proof}

\subsection{W[1]-hardness when parameterised by \texorpdfstring{$\ell$}{l} for multidigraphs}
\label{sec:Whardbyell}

We now move to the proof of Theorem~\ref{thm:W1_hardness_invkp_by_ell_ETH}, for which we need two easy remarks on 2-arc-strong multidigraphs.

\begin{proposition}
    \label{gguuuo}
    Let $D$ be a multidigraph and $v \in V(D)$ with $d_D^+(v)=d_D^-(v)=2$ such that there are two parallel arcs linking $v$ and a vertex $w_1 \in V(D)$, one arc linking $v$ and a vertex $w_2 \in V(D)$ and one arc linking $v$ and a vertex $w_3 \in V(D)$. Moreover, let $\mathcal{X}\subseteq \binom{V(D)}{2}$ such that $\Inv(D;\mathcal{X})$ is 2-arc-strong. Then either $\{\{v,w_1\}, \{v,w_2\},\{v,w_3\}\}\subseteq \mathcal{X}$ or $\{\{v,w_1\}, \{v,w_2\},\{v,w_3\}\}\cap\mathcal{X}=\emptyset$.
\end{proposition}
\begin{proof}
    By symmetry, we may suppose that $A(D)$ contains an arc from $v$ to $w_2$, an arc from $v$ to $w_3$ and two parallel arcs from $w_1$ to $v$. Let $D'=\Inv(D;\mathcal{X})$.
    
    First suppose that $\{v,w_1\}\in \mathcal{X}$, so $D'$ contains two parallel arcs from $v$ to $w_1$. As $D'$ is 2-arc-strong, we obtain that $d_{D'}^-(v)\geq 2$. As the only arcs incident to $v$ in $D$ distinct from the arcs linking $v$ and $w_1$ are the arcs $vw_2$ and $vw_3$, we obtain that $D'$ contains the arcs $w_2v$ and $w_3v$. This yields that $\{\{v,w_2\}, \{v,w_3\}\}\subseteq \mathcal{X}$, so $\{\{v,w_1\}, \{v,w_2\},\{v,w_3\}\}\subseteq \mathcal{X}$. 

    Now suppose that $\{v,w_1\}$ is not contained in $ \mathcal{X}$, so $D'$ contains two parallel arcs from $w_1$ to $v$. As $D'$ is 2-arc-strong, we obtain that $d_{D'}^+(v)\geq 2$. As the only arcs incident to $v$ in $D$ distinct from the arcs linking $v$ and $w_1$ are the arcs $vw_2$ and $vw_3$, we obtain that $D'$ also contains the arcs $vw_2$ and $vw_3$. This yields that $\{\{v,w_2\}, \{v,w_3\}\}\cap \mathcal{X}=\emptyset$, so $\{\{v,w_1\}, \{v,w_2\},\{v,w_3\}\}\cap \mathcal{X}=\emptyset$. 
\end{proof} 

\begin{proposition}\label{awsertdrzftugzi}
    Let $D$ be a digraph and $Q \subseteq V(D)$ such that $D$ is $2$-arc-strong in $Q$. Moreover, let $x_1,\ldots,x_k,y_1,\ldots,y_k\in V(D)$ be pairwise distinct vertices for some positive integer $k$ such that $D$ contains two parallel arcs from $Q$ to $x_1$, two parallel arcs from $y_k$ to $Q$, two parallel arcs from $y_i$ to $x_{i+1}$ for all $i \in [k-1]$, and an arc from $x_i$ to $y_i$, an arc from $x_i$ to $Q$ and an arc from $Q$ to $y_i$ for $i \in [k]$. Then $D$ is $2$-arc-strong in $Q \cup \{x_1,\ldots,x_k,y_1,\ldots,y_k\}$.
\end{proposition}
\begin{proof}
    By symmetry, it suffices to prove that $\lambda_D(Q,x_i)\geq 2$ and $\lambda_D(Q,y_i)\geq 2$ hold for all $i \in [k]$. First consider some $i \in [k]$ and observe that $A(D)$ contains an arc from $Q$ to $y_i$ by assumption. Next, if $i=1$, we have that $D$ contains the directed path $qx_1y_1$ and if $i \geq 2$, then $D$ contains the directed path $qy_{i-1}x_iy_i$ for some $q \in Q$. We obtain that $\lambda_D(Q,y_i)\geq 2$ for all $i\in [k]$. Next observe that $A(D)$ contains two parallel arcs from $Q$ to $x_1$, so $\lambda_D(Q,x_1)\geq 2$. Finally, for every $i \geq 2$, we have that $A(D)$ contains two parallel arcs from $y_{i-1}$ to $x_i$. As $\lambda_D(Q,y_{i-1})\geq 2$, we obtain that $\lambda_D(Q, x_i)\geq 2$. This finishes the proof.
\end{proof}

The main technical difficulty for the proof of Theorem~\ref{thm:W1_hardness_invkp_by_ell_ETH} is contained in the following lemma that shows how to construct an equivalent instance of {\sc $(2,2)$-Inversion} from an instance of NPSI.

\begin{lemma}\label{rzxdctuvfzink}
Given an instance $(G,(V_1,\ldots,V_k),H)$ of NPSI, one can compute an equivalent instance $(D,\ell)$ of {\sc $(2,2)$-Inversion} with $|V(D)|\leq 11 \cdot |V(G)|^3$ and $\ell \leq 12 \cdot |E(H)|$ in polynomial time.
\end{lemma}
\begin{proof}
    Let $(G,(V_1,\ldots,V_k),H)$ be an instance of NPSI .
    We construct an instance $(D,\ell)$ of {\sc $(2,2)$-Inversion}.
    We first let $V(D)$ contain a vertex $s$, which will have a particular role in the construction.
    
    Now consider some $c \in [k]$. We create a color gadget $D_c$, illustrated in Figure~\ref{drftg}.
    We first let $V(D_c)$ contain $s$ and two new vertices $a_c$ and $b_c$. 
    We add an arc from $s$ to $a_c$ and an arc from $b_c$ to $s$. Now let $c_1,\ldots,c_\mu$ be an arbitrary ordering of $N_H(c)$, where $\mu=d_H(c)$. For every $v \in V_c$ and every $i \in [\mu]$, we let $V(D_c)$ contain three vertices $x_{v,c_i},y_{v,c_i}$, and $z_{v,c_i}$. 
    Next, for every $v \in V_c$, we let $A(D_c)$ contain two parallel arcs from $a_c$ to $x_{v,c_1}$ and two parallel arcs from $y_{v,c_\mu}$ to $b_c$. 
    Next, for every $v \in V_c$ and $i \in [\mu-1]$, we let $A(D_c)$ contain two parallel arcs from $y_{v,c_i}$ to $x_{v,c_{i+1}}$. Further, for every $v \in V_c$ and $i \in [\mu]$, we let $A(D_c)$ contain an arc from $x_{v,c_i}$ to $y_{v,c_i}$, an arc from $x_{v,c_i}$ to $z_{v,c_i}$, an arc from $s$ to $y_{v,c_i}$, an arc from $s$ to $z_{v,c_i}$, and two parallel arcs from $z_{v,c_i}$ to $s$. This finishes the description of~$D_c$. 
    We let $D$ contain the gadget $D_c$ for all $c\in [k]$. Observe that for all distinct $c,c' \in [k]$, we have $V(D_c)\cap V(D_{c'})=\{s\}$.

    For every $c \in [k]$ and $v \in V_c$, we let $\mathcal{R}_{v}$ consist of all pairs $\{w_1,w_2\}\subseteq V(D_c)$ such that:
    \begin{itemize}[itemsep=0pt]
        \item $w_1$ and $w_2$ are adjacent in $D_c$, and
        \item $\{w_1,w_2\}$ contains $x_{v,c_i}$ or $y_{v,c_i}$ for some $i \in [\mu]$.
    \end{itemize}
    Observe that $|\mathcal{R}_{v}|=4d_H(c)+1$ for every $v \in V_c$. 

    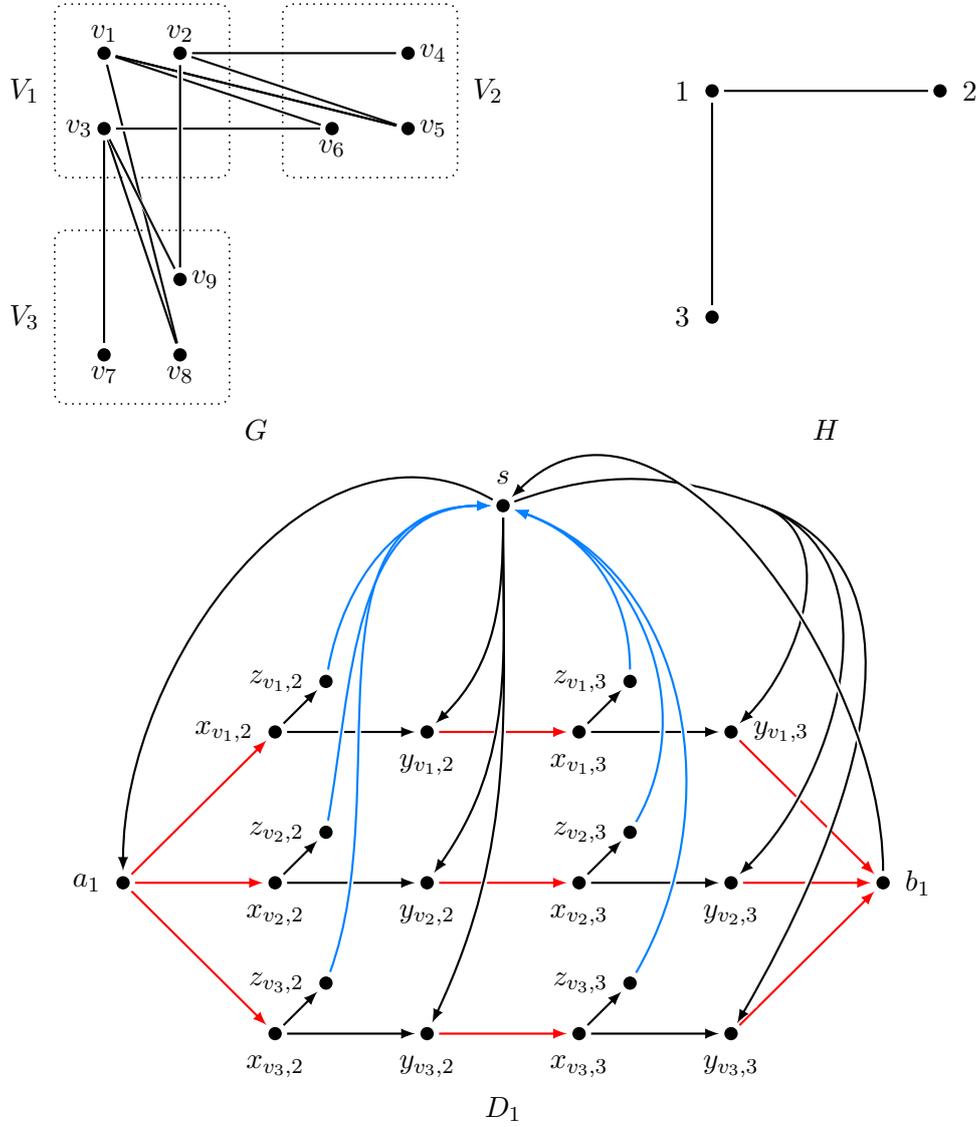
\begin{figure}
        \centering
        \begin{tikzpicture}[
        region/.style={draw=black, dotted, rounded corners, semithick,inner sep=14pt, outer sep=2pt}
        ]   
            \def\R{1}
            \def\halfR{0.5}
            \begin{scope}
                \node[] at (2*\R, -5*\R) {$G$};
                \begin{scope}
                    \node[mvertex, label={[label distance=-4pt]above:$v_1$}] (v1) at (0,0) {}; 
                    \node[mvertex, label={[label distance=-4pt]above:$v_2$}] (v2) at (\R,0) {}; 
                    \node[mvertex, label={[label distance=-4pt]left:$v_3$}] (v3) at (0,-\R) {}; 
                    \node[region,fit=(v1)(v2)(v3), label=left:$V_1$] (V1) {};
                \end{scope}
                \begin{scope}[xshift=3*\R cm]
                    \node[mvertex, label={[label distance=-4pt]right:$v_4$}] (v4) at (\R,0) {}; 
                    \node[mvertex, label={[label distance=-4pt]right:$v_5$}] (v5) at (\R,-\R) {}; 
                    \node[mvertex, label={[label distance=-4pt]below:$v_6$}] (v6) at (0,-\R) {}; 
                    \node[region,fit=(v4)(v5)(v6), label=right:$V_2$] (V1) {};
                \end{scope}
                \begin{scope}[yshift=-3*\R cm]
                    \node[mvertex, label={[label distance=-4pt]below:$v_7$}] (v7) at (0,-\R) {}; 
                    \node[mvertex, label={[label distance=-4pt]below:$v_8$}] (v8) at (\R,-\R) {}; 
                    \node[mvertex, label={[label distance=-4pt]right:$v_9$}] (v9) at (\R,0) {}; 
                    \node[region,fit=(v7)(v8)(v9), label=left:$V_3$] (V1) {};
                \end{scope}
                \draw[tedge] (v1) -- (v5);
                \draw[tedge] (v1) -- (v5);
                \draw[tedge] (v1) -- (v6);
                \draw[tedge] (v1) -- (v8);
                \draw[tedge] (v2) -- (v4);
                \draw[tedge] (v2) -- (v5);
                \draw[stedge,white] (v3) -- (v6);
                \draw[tedge] (v3) -- (v6);
                \draw[stedge,white] (v2) -- (v9);
                \draw[tedge] (v2) -- (v9);
                \draw[tedge] (v3) -- (v7);
                \draw[tedge] (v3) -- (v8);
                \draw[vtedge,white] (v3) -- (v9);
                \draw[tedge] (v3) -- (v9);
            \end{scope}
            \begin{scope}[xshift = 8cm, yshift = -\halfR cm]
                \node[] at (1.5*\R, -4.5*\R) {$H$};
                \node[mvertex, label=left:$1$] (h1) at (0,0) {};
                \node[mvertex, label=right:$2$] (h2) at (3*\R,0) {};
                \node[mvertex, label=left:$3$] (h3) at (0,-3*\R) {};
                \draw[tedge] (h1) -- (h3);
                \draw[tedge] (h1) -- (h2);
            \end{scope}

            \begin{scope}[xshift=5.25cm, yshift=-11cm]
                \def\L{2}
                
                \node[] (D1) at (0, -1.5*\L) {$D_1$};
                
                \node[mvertex, label=left:$a_1$] (a1) at (-2.5*\L, 0) {};
                \node[mvertex, label=right:$b_1$] (b1) at (2.5*\L, 0) {};

                \node[mvertex, label=below:$x_{v_3,2}$] (x1v32) at (-1.5*\L, -\L) {};
                \node[mvertex, label=below:$y_{v_3,2}$] (y1v32) at (-0.5*\L, -\L) {};
                \node[mvertex, label=below:$x_{v_3,3}$] (x1v33) at (0.5*\L, -\L) {};
                \node[mvertex, label=below:$y_{v_3,3}$] (y1v33) at (1.5*\L, -\L) {};
                \draw[tarc,red] (a1) -- (x1v32);
                \draw[tarc] (x1v32) -- (y1v32);
                \draw[tarc,red] (y1v32) -- (x1v33);
                \draw[tarc] (x1v33) -- (y1v33);
                \draw[tarc,red] (y1v33) -- (b1);
                
                \node[mvertex, label=below:$x_{v_2,2}$] (x1v22) at (-1.5*\L, 0) {};
                \node[mvertex, label=below:$y_{v_2,2}$] (y1v22) at (-0.5*\L, 0) {};
                \node[mvertex, label=below:$x_{v_2,3}$] (x1v23) at (0.5*\L, 0) {};
                \node[mvertex, label=below:$y_{v_2,3}$] (y1v23) at (1.5*\L, 0) {};
                \draw[tarc,red] (a1) -- (x1v22);
                \draw[tarc] (x1v22) -- (y1v22);
                \draw[tarc,red] (y1v22) -- (x1v23);
                \draw[tarc] (x1v23) -- (y1v23);
                \draw[tarc,red] (y1v23) -- (b1);
                
                \node[mvertex, label=left:$x_{v_1,2}$] (x1v12) at (-1.5*\L, \L) {};
                \node[mvertex, label=below:$y_{v_1,2}$] (y1v12) at (-0.5*\L, \L) {};
                \node[mvertex, label=below:$x_{v_1,3}$] (x1v13) at (0.5*\L, \L) {};
                \node[mvertex, label=right:$y_{v_1,3}$] (y1v13) at (1.5*\L, \L) {};
                \draw[tarc,red] (a1) -- (x1v12);
                \draw[tarc] (x1v12) -- (y1v12);
                \draw[tarc,red] (y1v12) -- (x1v13);
                \draw[tarc] (x1v13) -- (y1v13);
                \draw[tarc,red] (y1v13) -- (b1);

                \def\P{0.33}
                \node[mvertex, label=left:$z_{v_3,2}$] (z132) at (-\L-\P, -\L/2-\P) {};
                \node[mvertex, label=left:$z_{v_2,2}$] (z122) at (-\L-\P, \L/2-\P) {};
                \node[mvertex, label=left:$z_{v_1,2}$] (z112) at (-\L-\P, 3*\L/2-\P) {};
                \draw[tarc] (x1v32) to (z132);
                \draw[tarc] (x1v22) to (z122);
                \draw[tarc] (x1v12) to (z112);

                \def\P{0.33}
                \node[mvertex, label=left:$z_{v_3,3}$] (z133) at (\L-\P, -\L/2-\P) {};
                \node[mvertex, label=left:$z_{v_2,3}$] (z123) at (\L-\P, \L/2-\P) {};
                \node[mvertex, label=left:$z_{v_1,3}$] (z113) at (\L-\P, 3*\L/2-\P) {};
                \draw[tarc] (x1v33) to (z133);
                \draw[tarc] (x1v23) to (z123);
                \draw[tarc] (x1v13) to (z113);
                
                \node[mvertex, label=above:$s$] (s) at (0, 2.5*\L) {};
                
                \draw[vtarc, white] (z132) to[out=70, in=180] (s);
                \draw[vtarc, white] (z122) to[out=80, in=180] (s);
                \draw[tarc, g-blue] (z132) to[out=70, in=180] (s);
                \draw[tarc, g-blue] (z122) to[out=80, in=180] (s);
                \draw[tarc, g-blue] (z112) to[out=80, in=180] (s);
                
                \draw[vtarc, white] (s) to[out=-90, in=55] (y1v22);
                \draw[vtarc, white] (s) to[out=-90, in=65] (y1v32);
                \draw[tarc] (s) to[out=-90, in=45] (y1v12);
                \draw[tarc] (s) to[out=-90, in=55] (y1v22);
                \draw[tarc] (s) to[out=-90, in=65] (y1v32);

                \draw[vtarc, white] (z133) to[out=60, in=-20] (s);
                \draw[vtarc, white] (z123) to[out=60, in=-20] (s);
                \draw[tarc, g-blue] (z133) to[out=60, in=-20] (s);
                \draw[tarc, g-blue] (z123) to[out=60, in=-20] (s);
                \draw[tarc, g-blue] (z113) to[out=90, in=-20] (s);
                
                \draw[thick] (s) to[out=20, in=160] (1.7*\L, 2.5*\L) {};
                \draw[vtarc, white] (1.7*\L, 2.5*\L)  to[out=0, in=45] (y1v23);
                \draw[vtarc, white] (1.7*\L, 2.5*\L)  to[out=0, in=60] (y1v33);
                \draw[tarc] (1.7*\L, 2.5*\L)  to[out=-20, in=45] (y1v13);
                \draw[tarc] (1.7*\L, 2.5*\L)  to[out=-20, in=45] (y1v23);
                \draw[tarc] (1.7*\L, 2.5*\L)  to[out=-20, in=60] (y1v33);
                
                \draw[tarc] (s) to[out=150, in=90] (a1);
                \draw[vtarc,white] (b1) to[out=90, in=45] (s);
                \draw[tarc] (b1) to[out=90, in=45] (s);
            \end{scope}
        \end{tikzpicture}
        \caption{An example for the construction of a color gadget. An instance $(G,(V_1,\ldots,V_k),H)$ is depicted on the top part of the figure. Again, for the sake of readability, we have chosen $(G,(V_1,\ldots,V_k),H)$ not being an instance of NPSI. The clause gadget $D_1$ is depicted in the bottom part of the figure. Red arcs indicate double arcs and blue arcs indicate two arcs in the indicated direction and one arc in the opposite direction.}
        \label{drftg}
    \end{figure}
    
    Next, for every $e\in E(H)$, we let $V(D)$ contain a vertex $u_{e}$ and we add an arc from $s$ to $u_{e}$.
    
    Now consider some $f=vv'\in E(G)$ and let $c,c'\in [k]$ be such that $v \in V_c$ and $v' \in V_{c'}$. We let $V(D)$ contain a vertex $t_{f}$ and two parallel arcs from $t_{f}$ to $u_{cc'}$. Observe that $u_{cc'}$ exists as $(G,(V_1,\ldots,V_k),H)$ is normalised. We further add an arc from $z_{v,c'}$ to $t_f$ and an arc from $z_{v',c}$ to $t_f$. We do this for all $e \in E(G)$. This finishes the description of $D$, see Figure~\ref{tgtu} for an illustration.   
    
    Observe that $d_D^-(t_{f})=d_D^+(t_{f})=2$ for every $f\in E(G)$. For every $f=vv'\in E(G)$, we further define 
    \[
    \mathcal{R}_{f} = \{ \{t_{f},w\} : w \in \{u_{cc'},z_{v,c'},z_{v',c}\}\},
    \]
    where $c,c'\in [k]$ are such that $v \in V_c$ and $v' \in V_{c'}$. Observe that $|\mathcal{R}_{f}|=3$ for every $f \in E(G)$.
    
    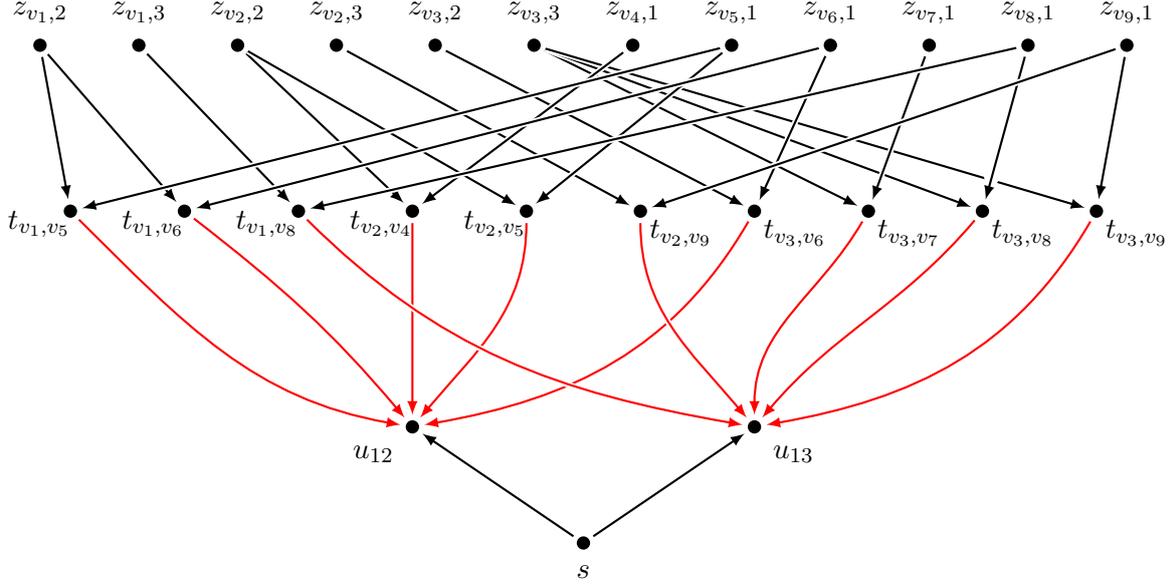
\begin{figure}
    \centering
        \begin{tikzpicture}
            \def\L{1.3}
            \def\Lt{1.5}
            \def\h{-2.2}

            \begin{scope}[xshift=-5.5*\L cm]
                \node[mvertex,label=above:$z_{v_1,2}$] (z12) at (0,0) {};
                \node[mvertex,label=above:$z_{v_1,3}$] (z13) at (\L,0) {};
                \node[mvertex,label=above:$z_{v_2,2}$] (z22) at (2*\L,0) {};
                \node[mvertex,label=above:$z_{v_2,3}$] (z23) at (3*\L,0) {};
                \node[mvertex,label=above:$z_{v_3,2}$] (z32) at (4*\L,0) {};
                \node[mvertex,label=above:$z_{v_3,3}$] (z33) at (5*\L,0) {};
                \node[mvertex,label=above:$z_{v_4,1}$] (z41) at (6*\L,0) {};
                \node[mvertex,label=above:$z_{v_5,1}$] (z51) at (7*\L,0) {};
                \node[mvertex,label=above:$z_{v_6,1}$] (z61) at (8*\L,0) {};
                \node[mvertex,label=above:$z_{v_7,1}$] (z71) at (9*\L,0) {};
                \node[mvertex,label=above:$z_{v_8,1}$] (z81) at (10*\L,0) {};
                \node[mvertex,label=above:$z_{v_9,1}$] (z91) at (11*\L,0) {};
            \end{scope}

            \begin{scope}[xshift=-5.5*\Lt cm]
                \node[mvertex,label={[label distance=-10pt]225:$t_{v_1,v_5}$}] (t15) at (1*\Lt,\h) {};
                \node[mvertex,label={[label distance=-10pt]225:$t_{v_1,v_6}$}] (t16) at (2*\Lt,\h) {};
                \node[mvertex,label={[label distance=-10pt]225:$t_{v_1,v_8}$}] (t18) at (3*\Lt,\h) {};
                \node[mvertex,label={[label distance=-10pt]225:$t_{v_2,v_4}$}] (t24) at (4*\Lt,\h) {};
                \node[mvertex,label={[label distance=-10pt]225:$t_{v_2,v_5}$}] (t25) at (5*\Lt,\h) {};
                \node[mvertex,label={[label distance=-5pt]-45:$t_{v_2,v_9}$}] (t29) at (6*\Lt,\h) {};
                \node[mvertex,label={[label distance=-5pt]-45:$t_{v_3,v_6}$}] (t36) at (7*\Lt,\h) {};
                \node[mvertex,label={[label distance=-5pt]-45:$t_{v_3,v_7}$}] (t37) at (8*\Lt,\h) {};
                \node[mvertex,label={[label distance=-5pt]-45:$t_{v_3,v_8}$}] (t38) at (9*\Lt,\h) {};
                \node[mvertex,label={[label distance=-5pt]-45:$t_{v_3,v_9}$}] (t39) at (10*\Lt,\h) {};
            \end{scope}

            \foreach \i/\j in {12/15, 12/16, 13/18, 22/24, 22/25, 23/29, 32/36, 33/37, 33/38, 33/39, 41/24, 51/15, 51/25, 61/16, 61/36, 71/37, 81/18, 81/38, 91/29, 91/39} {
                \draw[vtarc, white] (z\i) to (t\j);
                \draw[tarc] (z\i) to (t\j);
            }

            \node[mvertex, label=225:$u_{12}$] (u12) at (-1.5*\Lt, 2.3*\h) {};
            \node[mvertex, label=-45:$u_{13}$] (u13) at (1.5*\Lt, 2.3*\h) {};
            \node[mvertex, label=below:$s$] (s) at (0, 3*\h) {};
            \draw[tarc] (s) to (u12);
            \draw[tarc] (s) to (u13);
            
            \draw[tarc, red] (t15) to[out=-45, in=170] (u12);
            \draw[tarc, red] (t16) to[out=-38, in=130] (u12);
            \draw[tarc, red] (t24) to[out=-90, in=90] (u12);
            \draw[tarc, red] (t25) to[out=-90, in=50] (u12);
            \draw[tarc, red] (t36) to[out=-120, in=10] (u12);

            \draw[vtarc, white] (t18) to[out=-45, in=170] (u13);
            \draw[tarc, red] (t18) to[out=-45, in=170] (u13);
            \draw[vtarc, white] (t29) to[out=-90, in=130] (u13);
            \draw[tarc, red] (t29) to[out=-90, in=130] (u13);
            \draw[tarc, red] (t37) to[out=-120, in=90] (u13);
            \draw[tarc, red] (t38) to[out=-130, in=50] (u13);
            \draw[tarc, red] (t39) to[out=-120, in=10] (u13);
        \end{tikzpicture}
        \caption{An example of the construction of $D$ for the example depicted in Figure~\ref{drftg}. Again, the red arrows indicate two parallel arcs. 
        For the sake of better readability, we omit the vertices $a_c$ and $b_c$ for every $c \in [k]$, and, for every $c \in [k], v \in V_{c}$ and $c'\in [k]\setminus \{c\}$ with $cc' \in E(H)$, we omitted the vertices $x_{v,c'}$ and $y_{v,c'}$ and all arcs linking $z_{v,c'}$ and $s$.}\label{tgtu}
\end{figure}

    We finally set $\ell=11|E(H)|+|V(H)|$.
    It is not difficult to see that $(D,\ell)$ can be computed in polynomial time from $(G,(V_1,\ldots,V_k),H)$. Moreover, by construction and as $(G,(V_1,\ldots,V_k),H)$ is normalised, we have that 
    \[
    |V (D)| = |E(G)| + |E(H)| + 3 \sum_{c \in [k]}d_H(c)|V_c| + 2k + 1 \leq 11\cdot|V (G)|^3.
    \]
    Further, as $(G,(V_1,\ldots,V_k),H)$ is normalised, we have $\ell = 11|E(H)| + |V (H)| \leq 12\cdot |E(H)|$.

    \medskip
    
    In the following, we show that $(D,\ell)$ is a yes-instance of {\sc $(2,2)$-Inversion} if and only if $(G,(V_1,\ldots,V_k),H)$ is a yes-instance of NPSI. First suppose that $(G,(V_1,\ldots,V_k),H)$ is a yes-instance of NPSI. That is, there exists some $v_c \in V_c$ for all $c \in [k]$ such that $v_cv_{c'}\in E(G)$ holds for all $cc'\in E(H)$. We let $\mathcal{X}$ be the set that contains $\mathcal{R}_{v_c}$ for all $c \in [k]$ and $\mathcal{R}_{f}$ for every $f \in E(G[\{v_1,\ldots,v_k\}])$. Let $D'=\Inv(D;\mathcal{X})$. 
    Observe that 
    \begin{align*}
        |\mathcal{X}|=\sum_{c \in [k]}|\mathcal{R}_{v_c}|+\sum_{f \in E(G[\{v_1,\ldots,v_k\}])}|\mathcal{R}_{f}|&=\sum_{c \in [k]}(4d_H(c)+1)+3|E(G[\{v_1,\ldots,v_k\}]|\\
        &=8|E(H)|+|V(H)|+3|E(H)|=\ell.
    \end{align*}

    It hence suffices to prove that $D'$ is 2-arc-strong. Let $Q\subseteq V(D)$ be a set of maximum cardinality such that $s \in Q$ and $D'$ is $2$-arc-strong in $Q$. In the next sequence of claims, we show that $Q=V(D)$, implying that $D'$ is $2$-arc-strong.
    
    \begin{claim}
        \label{claim:subset_Q:first}
        For every $e\in E(H)$, we have $u_{e}\in Q$.
    \end{claim}
    \begin{proofclaim}
        Let $e = cc' \in E(H)$. Observe that $D'$ contains an arc from $s$ to $u_{e}$. Further consider some $f \in E(G[V_c\cup V_{c'}])$ such that $f\neq v_cv_{c'}$. Observe that such an edge exists because $(G,(V_1,\ldots,V_k),H)$ is normalised. Then $D'$ contains the directed path $sz_{v_c,c'}t_{f}u_{e}$. We obtain that $\lambda_{D'}(s,u_{e})\geq 2$. Next observe that $D'$ contains two parallel arcs from $u_{e}$ to $t_{v_cv_{c'}}$ and the arc-disjoint directed paths $t_{v_cv_{c'}}z_{v_c,c'}s$ and $t_{v_cv_{c'}}z_{v_{c'},c}s$. It follows that $\lambda_{D'}(u_{e},s)\geq 2$. We obtain by \Cref{prop:showing_k_strong} that $u_{e}\in Q$, as desired.
    \end{proofclaim}
    
    \begin{claim}
        For every $f\in E(G)$, we have  $t_{f}\in Q$.
    \end{claim}
    \begin{proofclaim}
        Let $f=vv'\in E(G)$ and let $c,c'\in [k]$ be such that $v \in V_c$ and $v' \in V_{c'}$. Suppose first that $f=v_cv_{c'}$. Then $D'$ contains two parallel arcs from $u_{cc'}$ to $t_{f}$ and the two arc-disjoint paths $t_{f}z_{v,c'}s$ and $t_{f}z_{v',c}s$. 
        Now suppose that $f\neq v_cv_{c'}$. Then  $D'$ contains two parallel arcs from $t_{f}$ to $u_{cc'}$ and the two arc-disjoint paths $sz_{v,c'}t_{f}$ and $sz_{v',c}t_{f}$. It follows from the previous claim and \Cref{prop:showing_k_strong} that $t_{f}\in Q$, as desired.
    \end{proofclaim}

    \begin{claim}
        For every edge $cc'\in E(H)$ and every $v\in V_c$, we have $z_{v,c'}\in Q$.
    \end{claim}
    \begin{proofclaim}
        Observe that $D'$ contains two parallel arcs from $z_{v,c'}$ to $s$, so $\lambda_{D'}(z_{v,c'},Q)\geq 2$. Next observe that $D'$ contains an arc from $s$ to $z_{v,c'}$. Next, if $v=v_c$, then $A(D')$ contains an arc from $t_{vv_{c'}}$ to $z_{v,c'}$, so $\lambda_{D'}(Q,z_{v,c'})\geq 2$. Now suppose that $v \neq v_c$ and let $c_1,\ldots,c_\mu$ be the ordering of $N_H(c)$ used when constructing $D_c$. If $c'=c_1$, we obtain that $D'$ contains the directed path $sa_cx_{v,c'}z_{v,c'}$ and if $c'=c_i$ for some $i \geq 2$, we obtain that $D'$ contains the directed path $sy_{v,c_{i-1}}x_{v,c'}z_{v,c'}$. In either case, we obtain that $\lambda_{D'}(Q,z_{v,c'})\geq 2$. The claim follows from \Cref{prop:showing_k_strong}.
    \end{proofclaim}
    
    \begin{claim}
        For every $c\in [k]$, we have  $\{a_c,b_c\}\in Q$.
    \end{claim}
    \begin{proofclaim}
        Consider some $c \in [k]$. Let $c_1,\ldots,c_\mu$ be the ordering of $N_H(c)$ used when constructing $D_c$. Observe that $D'$ contains an arc from $s$ to $a_c$ and the directed path $z_{v_c,c_1}x_{v_c,c_1}a_c$. It follows that $\lambda_{D'}(Q,a_c)\geq 2$. Now let $v,v'$ be distinct vertices in $V_c\setminus \{v_c\}$. Observe that such vertices exist as $|V_c|\geq 3$ by definition of a normalised instance. Then $D'$ contains the arc-disjoint paths $a_cx_{v,c_1}z_{v,c_1}$ and $a_cx_{v',c_1}z_{v',c_1}$. It follows that $\lambda_{D'}(a_c,Q)\geq 2$. We obtain that $a_c \in Q$.
    
        Next observe that $D'$ contains the arc-disjoint directed paths $z_{v,c_\mu}y_{v,c_\mu}b_c$ and $z_{v',c_\mu}y_{v',c_\mu}b_c$, so $\lambda_{D'}(Q,b_c)\geq 2$. Moreover, $D'$ contains an arc from $b_c$ to $s$ and $D'$ contains the directed path $b_cy_{v_c,c_\mu}s$. It follows that  $\lambda_{D'}(b_c,Q)\geq 2$, so $b_c \in Q$. We obtain by \Cref{prop:showing_k_strong} that $\{a_c,b_c\}\subseteq Q$, as desired.
    \end{proofclaim}

    \begin{claim}
        \label{claim:subset_Q:last}
        For every $c,c' \in [k]$ and $v \in V_c$, we have  $\{x_{v,c'},y_{v,c'}\}\subseteq Q$.
    \end{claim}
    \begin{proofclaim}
        Let $c_1,\ldots,c_\mu$ be the ordering of $N_H(c)$ used when constructing $D_c$. If $v\neq v_c$, we obtain that $\{x_{v,c'},y_{v,c'}\}\subseteq Q$ for all $c' \in N_H(c)$ by applying Proposition~\ref{awsertdrzftugzi} to $Q$ and $x_{v,c_1},\ldots,x_{v,c_\mu},y_{v,c_1},\ldots,y_{v,c_\mu}$.  If $v= v_c$, we obtain that $\{x_{v,c'},y_{v,c'}\}\subseteq Q$ for all $c' \in N_H(c)$ by applying Proposition~\ref{awsertdrzftugzi} to $Q$ and $y_{v,c_\mu},\ldots,y_{v,c_1},x_{v,c_\mu},\ldots,x_{v,c_1}$. 
    \end{proofclaim}
    
    It follows from Claims~\ref{claim:subset_Q:first} to~\ref{claim:subset_Q:last} that $Q=V(D')$ and hence that $D'$ is 2-arc-strong. We obtain that $(D,\ell)$ is a yes-instance of {\sc $(2,2)$-Inversion}.

    \medskip

    Now suppose that $(D,\ell)$ is a yes-instance of {\sc $(2,2)$-Inversion}, so there exists a set $\mathcal{X}\subseteq \binom{V(D)}{2}$ with $|\mathcal{X}|\leq \ell$ such that $D'$ is 2-arc-strong, where $D'=\Inv(D;\mathcal{X})$. We prove some results on the structure of $\mathcal{X}$, which then imply the existence of the desired copy of $H$.
    \begin{claim}
        For every $e=cc'\in E(H)$, there exists some $f_e\in \delta_G(V_c,V_{c'})$ such that $\mathcal{R}_{f_e}\subseteq \mathcal{X}$. 
    \end{claim}
    \begin{proofclaim}
        Observe that there is only one arc in $D$ that is incident to $u_{e}$ and that is not of the form $t_fu_e$ for some $f \in \delta_G(V_c,V_{c'})$. As $D'$ is 2-arc-strong, we have $d_{D'}^+(u_{e})\geq 2$ and hence there exists some $f_e \in \delta_G(V_c,V_{c'})$  such that $\{u_e,t_{f_e}\}\in \mathcal{X}$. The statement now follows by applying Proposition~\ref{gguuuo} to $t_{f_e}$.
    \end{proofclaim}

    \begin{claim}
        For every $c \in [k]$, there exists some $v_c \in V_c$ such that $\mathcal{R}_{v_c}\subseteq \mathcal{X}$. 
    \end{claim}
    \begin{proofclaim}
    Let $c \in [k]$ and let $c_1,\ldots,c_\mu$ be the ordering of $N_H(c)$ used when constructing $D_c$.
    Observe that there is only one arc in $D$ that is incident to $a_c$ and that is not of the form $a_cx_{v,c_1}$ for some $v \in V_c$. As $D'$ is 2-arc-strong, we have $d_{D'}^-(a_c)\geq 2$ and hence there exists some $v_c \in V_c$  such that $\{a_c,x_{v_c,c_1}\}\in \mathcal{X}$. The statement now follows by applying Proposition~\ref{gguuuo} to $x_{v_c,c'}$ and $y_{v_c,c'}$ for all $c' \in N_H(c)$.
    \end{proofclaim}
    
\begin{claim}\label{tufzgiuk} 
    We have $\mathcal{X}=\bigcup_{e \in E(H)}\mathcal{R}_{f_e}\cup \bigcup_{c \in [k]}\mathcal{R}_{v_c}$.
\end{claim}
\begin{proofclaim}
    Observe that the sets $\mathcal{R}_{f_e}$ for all $ e\in E(H)$ and $\mathcal{R}_{v_c}$ for all $c \in [k]$ are pairwise disjoint. Hence, as $|\mathcal{X}|\leq \ell, |\mathcal{R}_{f_e}|=3$ for all $ e\in E(H)$ and $|\mathcal{R}_{v_c}|=4d_H(c)+1$ for all $c \in [k]$, we obtain that 
    \begin{align*}
        \ell\geq |\mathcal{X}| &\geq \sum_{e \in E(H)}|\mathcal{R}_{f_e}|+\sum_{c \in [k]}|\mathcal{R}_{v_c}|\\
        &=3|E(H)|+\sum_{c \in [k]}(4d_H(c)+1)\\
        &=11|E(H)|+|V(H)|\\
        &=\ell.
    \end{align*}
    Hence equality holds throughout and the claim follows.
\end{proofclaim}

\begin{claim}\label{hikiki}
    For every $e=cc'\in E'(H)$, the endvertices of $f_e$ are $v_c$ and $v_{c'}$.
\end{claim}
\begin{proofclaim}
    Suppose otherwise and, without loss of generality, that $f_e$ is not incident to $v_c$. It then follows by the definition of $v_c$ and Claim~\ref{tufzgiuk} that $\{x_{v_c,c'},z_{v_c,c'}\}\in \mathcal{X}$ and $\mathcal{X}$ does not contain any other set containing $z_{v_c,c'}$. It follows by construction that $d_{D'}^-(z_{v_c,c'})\leq 1$, a contradiction to $D'$ being 2-arc-strong.
\end{proofclaim}

By Claim~\ref{hikiki}, it follows in particular that $v_cv_{c'}\in E(G)$ holds for every $cc'\in E(H)$. We obtain that  $(G,(V_1,\ldots,V_k),H)$ is a yes-instance of NPSI.
\end{proof}
We are now ready to conclude \Cref{thm:W1_hardness_invkp_by_ell_ETH}, which we restate here for convenience.
\Whardbyell*

\begin{proof}
    Let $(G,(V_1,\ldots,V_r),H)$ be an instance of NPSI. Then, by Lemma~\ref{rzxdctuvfzink}, in polynomial time, one can compute an equivalent instance $(D,\ell)$ of {\sc $(2,2)$-Inversion} with $|V(D)|\leq 11|V(G)|^3$ and $\ell \leq 12 |E(H)|$. It follows directly from \Cref{tufizgouhjnik} and the fact that that $|V(D)|$ is polynomial in $|V(G)|$ and $\ell$ is bounded by a computable function of $|E(H)|$ that {\sc $(2,2)$-Inversion} is W[1]-hard with respect to $\ell$.
    
    For the second part, fix some absolute constant $c$ with $0<c\leq 1$ for which there exists some $\alpha>0$ such that PSI cannot be solved in $f(k')n^{\alpha (k')^c}$. By Corollary~\ref{tufizgouhjnik}, there exists some $\alpha_1>0$ such that NPSI cannot be solved in $f(k')n^{\alpha_1 (k')^c}$. Let $\alpha_0=\frac{1}{36}\alpha_1$. Suppose that {\sc $(2,2)$-Inversion} can be solved in time $f(\ell)n^{\alpha_0\ell^c}$. 
    First, as $\ell \leq 12 |E(H)|$ and $|V(D)|\leq 11|V(G)|^3$, it follows that
    $(G,(V_1,\ldots,V_r),H)$ can be solved in time 
    \[
    f(|E(H)|)\cdot |V(G)|^{3\alpha\ell^c}.
    \]
    Next, since $\ell \leq 12 |E(H)|$ and $c \leq 1$, we obtain that 
    \[
    3\alpha_0\ell^c\leq 36\alpha_0 |E(H)|^c= \alpha_1|E(H)|^c.
    \]
    It follows that $(G,(V_1,\ldots,V_r),H)$ can be solved in $f(|E(H)|)\cdot n^{\alpha_1|E(H)|^c}$, contradicting our assumption.
\end{proof}
\section{Conclusion}
We have given a characterization of digraphs for which a certain arc-connectivity property can be obtained by applying inversions of fixed size, assuming a moderate minimum size of the graph. Further, we have proved a collection of algorithmic results for the case when the same objective is supposed to be obtained using a minimum number of bounded-size inversions. Our work leaves many questions open.

For \Cref{thm:invertibility_even_case,thm:invertibility_odd_case}, it follows from \Cref{thm:NP_hardness_n-1_inversions} that the condition $n\geq p+2$ cannot be omitted. However, it would be interesting to see whether the condition $n \geq 4k+2$ is really crucial.
\begin{question}
    Are there integers $p\geq 3$ and $k\geq 1$ and a $2k$-edge-connected digraph $D$ on $n$ vertices with $n \geq p+2$ such that $D$ is not $(k,p)$-invertible and $p$ is even or $D$ is not a $k$-obstruction?
\end{question}

For the approximation questions, it would be interesting to better understand the influence of the inversion size and the connectivity requirement on the best approximation factor. An important question is whether the connectivity requirement can be eliminated from the approximation factor.
\begin{question}
    Is there a constant $\alpha>1$ such that, for any fixed integers $p\geq 2$ and $k\geq 1$, one can approximate $\inv{k}{p}$ in digraphs within a factor of $\alpha$ in polynomial time?
\end{question}

Our results indicate that increasing the value of $p$ does not make the approximation significantly worse. We conjecture the following hardness result that, roughly speaking, states that the approximation does at least not get easier when increasing $p$.
\begin{conjecture}
    There exists $\epsilon>0$ such that, unless $P=NP$, for any $p \geq 3$ and any $k \geq 1$, there does not exist a polynomial-time $(1+\epsilon)$-approximation algorithm for $\inv{k}{p}$ in digraphs.
\end{conjecture}

Also the parameterized complexity section leaves open questions. In \Cref{thm:W1_hardness_invkp_by_p}, the case of strong connectivity remains open.

\begin{question}
    Is there an algorithm that, given a digraph $D$, decides whether $\inv{1}{p}(D)\leq 1$ holds in time $f(p)\cdot n^{O(1)}$ for some computable function $f$?
\end{question}

The parameterization by $\ell$ is even less understood. In particular, \Cref{thm:W1_hardness_invkp_by_ell} only handles multidigraphs. We pose the following question: 
\begin{question}
    Is there an algorithm that, given a digraph $D$, decides whether $\inv{k}{p}(D)\leq \ell$ holds in time $f(\ell)\cdot n^{g(k,p)}$ for some computable functions $f,g$?
\end{question}
\bibliographystyle{abbrv}
\bibliography{refs}

@incollection{FRANK198297,
title = {An Algorithm for Submodular Functions on Graphs},
editor = {Achim Bachem and Martin Grötschel and Bemhard Korte},
series = {North-Holland Mathematics Studies},
publisher = {North-Holland},
volume = {66},
pages = {97-120},
year = {1982},
booktitle = {Bonn Workshop on Combinatorial Optimization},
issn = {0304-0208},
author = {András Frank}
}

@article{CHLEBIK2006320,
title = {Complexity of approximating bounded variants of optimization problems},
journal = {Theoretical Computer Science},
volume = {354},
number = {3},
pages = {320-338},
year = {2006},
issn = {0304-3975},
author = {Miroslav Chlebík and Janka Chlebíková}
}

@article{MONNOT2007677,
title = {The path partition problem and related problems in bipartite graphs},
journal = {Operations Research Letters},
volume = {35},
number = {5},
pages = {677-684},
year = {2007},
issn = {0167-6377},
author = {Jérôme Monnot and Sophie Toulouse}
}

@book{karp2010book,
  title={Reducibility among combinatorial problems},
  author={Karp, Richard M.},
  year={2010},
  publisher={Springer}
}

@article{APSSW,
    author = {Alon, Noga and Powierski, Emil and Savery, Michael and Scott, Alex and Wilmer, Elizabeth},
    title = {Invertibility of Digraphs and Tournaments},
    journal = {SIAM Journal on Discrete Mathematics},
    volume = {38},
    number = {1},
    pages = {327--347},
    year = {2024}
}

@article{BELKHECHINE2010703,
title = {Inversion dans les tournois},
journal = {Comptes Rendus Mathematique},
volume = {348},
number = {13},
pages = {703-707},
year = {2010},
issn = {1631-073X},
author = {Houmem Belkhechine and Moncef Bouaziz and Imed Boudabbous and Maurice Pouzet}
}

@Article{nishimuraALGO11,
    AUTHOR = {Nishimura, Naomi},
    TITLE = {Introduction to Reconfiguration},
    JOURNAL = {Algorithms},
    VOLUME = {11},
    YEAR = {2018},
    NUMBER = {4},
    ARTICLE-NUMBER = {52},
    ISSN = {1999-4893}
}

@BOOK{bang2009,
  AUTHOR =       {Bang-Jensen, J{\o}rgen and Gutin, Gregory Z.},
  TITLE =        {{Digraphs: Theory, Algorithms and Applications}},
  PUBLISHER =    {Springer-Verlag},
  address = {London},
  YEAR =         {2009},
  edition ={2nd}
 }

@article{ramsey1930,
  title = {On a Problem of Formal Logic},
  volume = {s2-30},
  ISSN = {0024-6115},
  url = {http://dx.doi.org/10.1112/plms/s2-30.1.264},
  DOI = {10.1112/plms/s2-30.1.264},
  number = {1},
  journal = {Proceedings of the London Mathematical Society},
  publisher = {Wiley},
  author = {Ramsey,  F. P.},
  year = {1930},
  pages = {264–286}
}

@book{diestel2017,
  title     = "Graph Theory",
  author    = "Diestel, Reinhard",
  publisher = "Springer",
  series    = "Graduate texts in mathematics",
  edition   =  {5th},
  year      =  2017,
  address   = "Berlin, Germany",
  language  = "en"
}

@article{dmtcs:7474,
  title={On the inversion number of oriented graphs},
  author={Bang-Jensen, J{\o}rgen and da Silva, Jonas Costa Ferreira and Havet, Fr{\'e}d{\'e}ric},
  journal={Discrete Mathematics \& Theoretical Computer Science},
  volume={23},
  number={2},
  year={2022},
  publisher={Episciences. org}
}

@article{Behague_Johnston_Morrison_Ogden_2025,
  title = {A note on inverting the dijoin of oriented graphs},
  volume = {32},
  number = {1},
  pages = {\#P1.44},
  journal = {The Electronic Journal of Combinatorics},
  author = {Behague, Natalie and Johnston, Tom and Morrison, Natasha and Ogden, Shannon},
  year = {2025}
}

@article{IGT_2026__3__49_0,
     author = {Havet, Fr\'ed\'eric and H\"orsch, Florian and Rambaud, Cl\'ement},
     title = {Diameter of the inversion graph},
     journal = {Innovations in Graph Theory},
     pages = {49--88},
     year = {2026},
     publisher = {Stichting Innovations in Graph Theory},
     volume = {3}
}

@article{behague2025casedijoinconjectureinverting,
      title={A case of the dijoin conjecture on inverting oriented graphs}, 
      author={Natalie Behague and Patrick Gaudart-Wifling},
      year={2025},
      journal={Preprint arXiv:2509.10232}
}

@article{wang2024inversionnumberdijoinsblowup,
      title={The inversion number of dijoins and blow-up digraphs}, 
      author={Haozhe Wang and Yuxuan Yang and Mei Lu},
      year={2024},
      journal={Preprint arXiv:2404.14937}
}

@article{inversion,
  title = {Problems, Proofs, and Disproofs on the Inversion Number},
  volume = {32},
  ISSN = {1077-8926},
  pages = {\#P1.42},
  number = {1},
  journal = {The Electronic Journal of Combinatorics},
  author = {Aubian, Guillaume and Havet, Fr\'ed\'eric and H\"orsch, Florian and Klingelhoefer, Felix and Nisse, Nicolas and Rambaud, Cl\'ement and Vermande, Quentin },
  year = {2025}
}

@article{wang2026inversiondiametertreewidth,
  title={Inversion diameter and treewidth},
  author={Wang, Yichen and Wang, Haozhe and Yang, Yuxuan and Lu, Mei},
  journal={Discrete Mathematics \& Theoretical Computer Science},
  volume={28},
  year={2026},
  publisher={Episciences. org}
}

@article{arana2026inversiondiameter2edgecoloredhomomorphisms,
      title={Inversion diameter and 2-edge-colored homomorphisms}, 
      author={Carmen Arana and Thomas Bellitto and Hector Buffière and Quentin Chuet and Théo Pierron and Amadeus Reinald},
      year={2026},
	   journal = {Preprint arXiv:2602.24171}
}

@article{https://doi.org/10.1002/jgt.23290,
author = {Duron, Julien and Havet, Frédéric and Hörsch, Florian and Rambaud, Clément},
title = {On the Minimum Number of Inversions to Make a Digraph $k$-(Arc-)Strong},
journal = {Journal of Graph Theory},
volume = {111},
number = {2},
pages = {31-62},
year = {2026}
}

@article{nashwilliamsCJM12,
  title={On orientations, connectivity and odd-vertex-pairings in finite graphs},
  author={Nash-Williams, Crispin St J. A.},
  journal={Canadian Journal of Mathematics},
  volume={12},
  pages={555--567},
  year={1960},
  publisher={Cambridge University Press}
}

@article{HORSCH2021103292,
title = {Connectivity of orientations of 3-edge-connected graphs},
journal = {European Journal of Combinatorics},
volume = {94},
pages = {103292},
year = {2021},
issn = {0195-6698},
author = {Florian Hörsch and Zoltán Szigeti}
}

@article{bangjensen2025makingorientedgraphacyclic,
      title={Making an oriented graph acyclic using inversions of bounded or prescribed size}, 
      author={Jørgen Bang-Jensen and Frédéric Havet and Florian Hörsch and Clément Rambaud and Amadeus Reinald and Caroline Silva},
      year={2025},
      journal = {Preprint arXiv:2511.22562}
}

@article{havet2026leqpinversiondiameteroriented,
      title={On the $(\leq p)$-inversion diameter of oriented graphs}, 
      author={Frédéric Havet and Clément Rambaud and Caroline Silva},
      year={2026},
      journal = {Preprint arXiv:2604.04633}
}

@article{doi:10.1137/090756144,
author = {Guruswami, Venkatesan and H\r{a}stad, Johan and Manokaran, Rajsekar and Raghavendra, Prasad and Charikar, Moses},
title = {{Beating the random ordering is hard: every ordering CSP is approximation resistant}},
journal = {SIAM Journal on Computing},
volume = {40},
number = {3},
pages = {878-914},
year = {2011}
}

@article{v009a024,
 author = {Svensson, Ola},
 title = {Hardness of Vertex Deletion and Project Scheduling},
 year = {2013},
 pages = {759--781},
 publisher = {Theory of Computing},
 journal = {Theory of Computing},
 volume = {9},
 number = {24}
}

@article{https://doi.org/10.1002/jgt.23251,
author = {Yuster, Raphael},
title = {On Tournament Inversion},
journal = {Journal of Graph Theory},
volume = {110},
number = {1},
pages = {82-91},
year = {2025}
}

@article{v012a006,
 author = {Guruswami, Venkatesan and Lee, Euiwoong},
 title = {Simple Proof of Hardness of Feedback Vertex Set},
 year = {2016},
 pages = {1--11},
 publisher = {Theory of Computing},
 journal = {Theory of Computing},
 volume = {12},
 number = {6}
}

@article{charbitCPC16,
  author       = {Pierre Charbit and
                  St{\'{e}}phan Thomass{\'{e}} and
                  Anders Yeo},
  title        = {The Minimum Feedback Arc Set Problem is {NP}-Hard for Tournaments},
  journal      = {Combinatorics, Probability, and Computing},
  volume       = {16},
  number       = {1},
  pages        = {1--4},
  year         = {2007}
}

@article{alonSJDM20,
    author = {Alon, Noga},
    title = {Ranking Tournaments},
    journal = {SIAM Journal on Discrete Mathematics},
    volume = {20},
    number = {1},
    pages = {137--142},
    year = {2006}
}

@article{maderMA191,
  title={Minimale n-fach kantenzusammenh{\"a}ngende Graphen},
  author={Mader, Wolfgang},
  journal={Mathematische Annalen},
  volume={191},
  number={1},
  pages={21--28},
  year={1971},
  publisher={Springer}
}

@incollection{frankNHMS66,
  title={An algorithm for submodular functions on graphs},
  author={Frank, Andr{\'a}s},
  booktitle={North-Holland Mathematics Studies},
  volume={66},
  pages={97--120},
  year={1982},
  publisher={Elsevier}
}

@article{dalmazzo1977,
  title={Nombre d'arcs dans les graphes $k$-arc-fortement connexes minimaux},
  author={Dalmazzo, Michel},
  journal={Compte-Rendu de l'Académie des Sciences de Paris A},
  volume={285},
  number={5},
  pages={341--344},
  year={1977}
}

@incollection{edmonds1973,
    author={Jack Edmonds},
    title={Edge-disjoint branchings},
    booktitle={Combinatorial Algorithms},
    publisher={Academic Press},
    editor={Randall Rustin},
    year={1973},
    pages={91--96}
}

@article{klostermeyerAC51,
  title={Pushing vertices and orienting edges},
  author={Klostermeyer, William F.},
  journal={Ars Combinatoria},
  volume={51},
  pages={65--75},
  year={1999}
}

@article{jacksonJGT12,
  title={Some remarks on Arc-connectivity, vertex splitting, and orientation in graphs and digraphs},
  author={Jackson, Bill},
  journal={Journal of Graph Theory},
  volume={12},
  number={3},
  pages={429--436},
  year={1988},
  publisher={Wiley Online Library}
}

@misc{anthropic2026eth,
  author       = {{Claude (Anthropic)}},
  title        = {A $2^{o(n)}$ Time Algorithm for 3-SAT},
  year         = {2026},
  month        = {October},
  howpublished = {\url{https://github.com/anthropics/formal-math/blob/main/eth/paper.pdf}},
  note         = {Unreviewed preprint and Lean 4 formalization, Anthropic Formal-Math Repository. Accessed: 2026-11-11}
}

@misc{OAI:The-Exponential-Time-Hypothesis-is-False-October-25-2026,
  author = {{OpenAI}},
  title = {{The Exponential Time Hypothesis is False}},
  howpublished = {OpenAI Math Release preprint
                  \href{https://github.com/openai/math/blob/main/preprints/The-Exponential-Time-Hypothesis-is-False-October-25-2026/paper.pdf}{OAI:The-Exponential-Time-Hypothesis-is-False-October-25-2026}},
  year = {2026}
}

@book{cyganPA2015,
  title     = "Parameterized Algorithms",
  author    = "Cygan, Marek and Fomin, Fedor V and Marx, Daniel and Saurabh,
               Saket and Kowalik, Lukasz and Lokshtanov, Daniel and Pilipczuk,
               Marcin",
  publisher = "Springer International Publishing",
  edition   =  1,
  month     =  jul,
  year      =  2015,
  address   = "Cham, Switzerland",
  copyright = "https://www.springernature.com/gp/researchers/text-and-data-mining",
  language  = "en"
}

@BOOK{schrijver2002,
  title     = "Combinatorial optimization: Polyhedra and efficiency",
  author    = "Schrijver, Alexander",
  publisher = "Springer",
  month     =  jul,
  year      =  2004,
  address   = "Berlin, Germany",
  language  = "en"
}

@article{marxTC6,
 author = {Marx, D{\'a}niel},
 title = {Can You Beat Treewidth?},
 year = {2010},
 pages = {85--112},
 publisher = {Theory of Computing},
 journal = {Theory of Computing},
 volume = {6},
 number = {5}
}

@article{boArxiv26,
  title={Edge-Number Bounds for the Inversion Diameter of Graphs},
  author={Bo, Jiawen and Li, Anqi and Lian, Xiaopan and Yan, Xin},
  journal={Preprint arXiv:2606.17974},
  year={2026}
}

\end{document}